\documentclass[11pt,reqno]{amsart}
\usepackage{geometry} 
\usepackage{amsmath,amssymb,amsthm, mathrsfs}
\usepackage{xcolor}
\usepackage{graphicx} 

\usepackage{enumerate}
\usepackage{indentfirst}
\usepackage[colorlinks,
            linkcolor=black,
            anchorcolor=black,
            citecolor=black
            ]{hyperref}
\DeclareMathOperator{\diver}{\mathrm{div}}

\newtheorem{mydef}{Definition}
\newtheorem{thm}{Theorem}[section]
\newtheorem{lem}{Lemma}[section]

\newtheorem{rk}{Remark}
\newtheorem{cor}{Corollary}[section]
\numberwithin{rk}{section}
\numberwithin{prop}{section}
\numberwithin{mydef}{section}
\numberwithin{lem}{section}
\numberwithin{equation}{section}
\numberwithin{thm}{section}
\allowdisplaybreaks[2]

\title[Compressible Navier-Stokes equations]{Global  spherically symmetric solutions to  degenerate compressible Navier-Stokes  equations with large data and  far field vacuum}

\date{\today}

\author{Yue Cao}
\address[Yue Cao]{School of Mathematical Sciences, and MOE-LSC, Shanghai Jiao Tong University,
Shanghai 200240, P. R. China} \email{\tt cao\_yue712@sjtu.edu.cn}

\author{Hao Li}
\address[Hao Li]{School of Mathematical Sciences,  Fudan University, Shanghai 200433, P. R. China} \email{\tt  hao\_li@fudan.edu.cn}

\author{Shengguo Zhu }
\address[Shengguo Zhu]{School of Mathematical Sciences,  CMA-Shanghai,   and MOE-LSC,   Shanghai Jiao Tong University, Shanghai 200240, P. R. China.}
 \email{\tt  zhushengguo@sjtu.edu.cn}

\begin{document}
\begin{abstract}

We consider  the initial-boundary value problem (IBVP) for the  isentropic compressible Navier-Stokes equations (\textbf{CNS})   in the domain  exterior to a ball in  $\mathbb R^d$ $(d=2\ \text{or} \ 3)$. When viscosity coefficients are
given as a constant multiple of the mass density $\rho$, based on
some analysis of the nonlinear structure of this system,  we prove the global existence   of the unique spherically symmetric classical   solution for (large)  initial data with spherical symmetry and  far field vacuum in some inhomogeneous Sobolev spaces. Moreover,  the solutions we obtained have  the conserved total mass and finite total energy.  $\rho$  keeps positive in the domain  considered  but decays to zero in the far field, which is consistent with the facts  that the total  mass  is conserved, and  \textbf{CNS} is a model of non-dilute fluids where $\rho$ is bounded away from the vacuum.
To prove the existence, on the one hand, we  consider   a  well-designed reformulated structure by introducing some new variables, which, actually, can transfer the degeneracies of the time evolution and the viscosity to the possible singularity of some special source terms. On the other hand, it is observed  that, for  the spherically symmetric flow,  the radial projection of the   so-called effective velocity   $\boldsymbol{v} =U+\nabla \varphi(\rho)$ ($U$ is the velocity of the fluid, and $\varphi(\rho)$ is a function of $\rho$ defined via the shear viscosity coefficient $\mu(\rho)$: $\varphi'(\rho)=2\mu(\rho)/\rho^2$),
verifies a damped transport equation which provides the possibility to obtain its upper bound.  Then combined with the BD entropy estimates, one can obtain the required uniform a priori estimates of the solution. 
It is worth pointing out that the frame work on the well-posedness theory established here can be applied to the  shallow water equations.
\end{abstract}


\maketitle

\section{introduction}

The time evolution of the mass density $\rho\geq 0$ and the velocity $U=(U^{(1)},\cdots, U^{(d)})^\top$ $\in \mathbb{R}^d$  of a general viscous isentropic compressible fluid occupying a spatial domain $\Omega\subset \mathbb{R}^d$ is governed by the following isentropic \textbf{CNS}:
\begin{equation}
\label{eq:1.1}
\begin{cases}
\rho_t+\text{div}(\rho U)=0,\\[5pt]
(\rho U)_t+\text{div}(\rho U\otimes U)
+\nabla
P =\text{div} \mathbb{T}.
\end{cases}
\end{equation}
Here,  $x=(x_1,\cdots, x_d)^{\top}\in \Omega$, $t\geq 0$ are the space and time variables, respectively.  
 For the polytropic gases, the constitutive relation is given by
\begin{equation}
\label{eq:1.2}
P=A\rho^{\gamma}, \quad A>0,\quad  \gamma> 1,
\end{equation}
where $A$ is  an entropy  constant and  $\gamma$ is the adiabatic exponent. $\mathbb{T}$ denotes the viscous stress tensor with the  form
\begin{equation}\label{eq:1.3}
\mathbb{T}=2\mu(\rho)D(U)+\lambda(\rho)\text{div}U\,\mathbb{I}_d,
\end{equation}
 where $D(U)=\frac{1}{2}\big(\nabla U+(\nabla U)^\top\big)$ is the deformation tensor,    $\mathbb{I}_d$ is the $d\times d$ identity matrix,
\begin{equation}
\label{fandan}
\mu(\rho)=\alpha  \rho^\delta,\quad \lambda(\rho)=\beta  \rho^\delta,
\end{equation}
for some  constant $\delta\geq 0$,
 $\mu(\rho)$ is the shear viscosity coefficient, $\lambda(\rho)+\frac{2}{d}\mu(\rho)$ is the bulk viscosity coefficient,  $\alpha$ and $\beta$ are both constants satisfying
 \begin{equation}\label{10000}\alpha>0 \quad \text{and} \quad   2\alpha+d\beta\geq 0.
 \end{equation}

In the theory of gas dynamics, the \textbf{CNS} can be derived from the Boltzmann equations through the Chapman-Enskog expansion, cf. Chapman-Cowling \cite{chap} and Li-Qin \cite{tlt}. Under some proper physical assumptions, one can find that  the viscosity coefficients and heat conductivity coefficient $\kappa$ are not constants but functions of the absolute temperature $\theta$ such as:
\begin{equation}
\label{eq:1.5g}
\begin{split}
\mu(\theta)=&a_1 \theta^{\frac{1}{2}}F(\theta),\quad \lambda(\theta)=a_2 \theta^{\frac{1}{2}}F(\theta), \quad \kappa(\theta)=a_3 \theta^{\frac{1}{2}}F(\theta),
\end{split}
\end{equation}
for some constants $a_i$ $(i=1,2,3)$ (see \cite{chap}). Actually for the cut-off inverse power force models, if the intermolecular potential varies as $\ell^{-\varkappa}$,
where $\ell$ is intermolecular distance and $\varkappa$ is a positive constant, then in (\ref{eq:1.5g}):
$$F(\theta)=\theta^{b}\quad \text{with}\quad b=\frac{2}{\varkappa} \in [0,\infty).$$
In particular (see \S 10 of \cite{chap}), for ionized gas,
$\varkappa=1$ and $ b=2$;
for Maxwellian molecules,
$\varkappa=4$ and $ b=\frac{1}{2}$; 
while for rigid elastic spherical molecules,
$\varkappa=\infty$ and $  b=0$.

According to Liu-Xin-Yang \cite{taiping}, for isentropic and polytropic fluids, such a dependence is inherited through the laws of Boyle and Gay-Lussac:
$$
P=\mathcal{R}\rho \theta=A\rho^\gamma \quad \text{for \  constant} \quad \mathcal{R}>0,
$$i.e., $\theta=A\mathcal{R}^{-1}\rho^{\gamma-1}$, and one can see that the viscosity coefficients are functions of $\rho$ taking   the form $(\ref{fandan})$.  Actually, there do exist some physical models that satisfy the density-dependent viscosities  assumption \eqref{fandan}, such as Korteweg system, shallow water equations, lake equations and quantum Navier-Stokes system (see \cite{bd6,bd2, BN1, BN2, Gent, jun, Mar}).
In particular,  the  viscous shallow water model in two-dimensional (2-D) space   reads as 
\begin{equation}\label{shallow}
\left\{
\begin{aligned}
&h_t+\text{div}(h W)=0,\\
&(hW)_t+\text{div}(h W\otimes W)+\nabla h^2=\mathcal{V}(h,W),
\end{aligned}
\right.
\end{equation}
where $h$ denotes the height of the free surface, $W=(W^{(1)}, W^{(2)})^{\top}\in \mathbb R^2$ is the mean horizontal velocity of the fluid. $\mathcal{V}(h,W)$ is the viscous term. There are  several different viscous terms imposed, such as $$\text{div}(h D(W)), \quad \text{div}(h\nabla W), \quad h\Delta W, \quad \Delta(h W).$$
In particular, the case for $\mathcal{V}=\text{div}(h\nabla W)$ is corresponding to the well-known viscous Saint-Venant model.  The derivation of Gent \cite{Gent} suggests $\mathcal{V}=\text{div}(h D(W))$. A more recent careful derivation by Marche \cite{Mar} and  Bresch-Noble \cite{BN1, BN2} suggests that
\begin{equation*}\label{newvis}
\mathcal{V}(h, W)=\text{div}(2h D(W)+2h \text{div} W \mathbb{I}_2).
\end{equation*}

In the current paper,  let  $\Omega=\{x\in\mathbb R^d|\,  |x|>a\}$ ($a>0$ is a constant) be an exterior domain in  $\mathbb R^d$ $(d=2\ \text{or} \ 3)$, and  $(\mu(\rho), \lambda(\rho))$ in \eqref{fandan} satisfy 
\begin{equation}\label{bdrelation}
\mu(\rho)=\alpha\rho,\quad \lambda(\rho)=0.
\end{equation}
We are concerned with global spherically symmetric (smooth) solutions taking the form 
\begin{equation}\label{duichenxingshi}
(\rho,U)(t,x)
=(\rho(t,|x|),u(t,|x|)\dfrac{x}{|x|})
\end{equation}
of the equations 
\eqref{eq:1.1}-\eqref{10000}  in the domain $\Omega$ with the initial data:
\begin{equation}\label{eqs:CauchyInit}
  (\rho,U)(0,x)=(\rho_0,U_0)(x)= ( \rho_0(|x|),u_0(|x|)\dfrac{x}{|x|})\, \quad\, \ \text{for}\quad  x \in \Omega,
\end{equation}
and the following   boundary conditions and far field behavior:
\begin{equation}\label{e1.3}
\begin{split}
\displaystyle
U(t,x)|_{|x|=a}=0\quad  \text{for}\quad   t\ge 0,\\[5pt]
\left(\rho(t,x),U(t,x)\right)\to \left(0,0\right)\quad  \text{as}\quad  \left|x\right|\to \infty\quad  \text{for}\quad   t\ge 0.&
\end{split}
\end{equation}

 There is a lot of literature on the global   well-posedness  of smooth solutions to the IBVP and Cauchy problem of    \eqref{eq:1.1}.  For  constant viscous flows   ($\delta=0$ in (\ref{fandan})),  when  $\inf_x {\rho_0(x)}>0$, the global well-posedness of strong solutions with arbitrarily large data in some bounded, one-dimensional (1-D)  domains has been proven by   Kazhikhov-Shelukhin \cite{KS}, and later, Kawashima-Nishida \cite{KN}  extended this theory  to the unbounded domains. In $\mathbb{R}^3$, Matsumura-Nishida \cite{MN} obtained  a unique global classical solution    for initial data close to a non-vacuum equilibrium in some Sobolev space $H^s(\mathbb{R}^3)$ ($s>\frac{5}{2}$) (see also Danchin \cite{dr} in some critical spaces).  In particular, Jiang \cite{J}  obtained the global existence of spherically symmetric smooth  solutions for (large) initial data with spherical symmetry to   the non-isentropic flow in the domain  exterior to a ball  in $\mathbb R^d$ $(d=2,3)$. 
 However, these approaches do not work when vacuum appears (i.e., $\inf_x {\rho_0(x)}=0$), which occurs when some physical requirements are imposed, such as finite total initial mass  and energy in the whole space.  One of the  main issues  in the presence of vacuum is the degeneracy of the time evolution operator, which makes it hard to  understand the behavior of $U$ near the vacuum.
A  remedy was suggested  by Cho-Choe-Kim \cite{CK3} in three-dimensional (3-D) space, where they imposed  initially a {\it compatibility condition}:
\begin{equation*}
-\text{div} \mathbb{T}_{0}+\nabla P(\rho_0)=\sqrt{\rho_{0}} g \quad \text{for\  \ some} \ \ g\in L^2(\mathbb{R}^3),
\end{equation*}
which  leads to  $\sqrt{\rho}U_t\in L^\infty ([0,T_{*}]; L^2(\mathbb{R}^3))$  for some  time $T_{*}>0$, and  they established    the local well-posedness of  strong  solutions with vacuum, which,  recently,  has been shown to be   a global one with small energy by Huang-Li-Xin \cite{HX1} in  $\mathbb{R}^3$. Later,  Jiu-Li-Ye \cite{j2}  proved that the 1-D Cauchy problem of  \eqref{eq:1.1} admits  a unique global classical solution  with arbitrarily large data and vacuum. Some interesting progress on the global  spherically symmetric strong solutions in annular or exterior  domains of $\mathbb R^d$ $(d\ge 2)$ can also be found in   Choe-Kim \cite{ck},  Ding-Wen-Yao-Zhu \cite{dwyz} and so on.
We also refer  Feireisl \cite{fu3}, Hoff \cite{H},   Hoff-Jenssen  \cite{HJ}, Lions \cite{lions} to readers      and the references therein for the existence theory of weak solutions with large data in multi-dimensional (M-D) space.

 For degenerate viscous flow ($\delta>0$ in (\ref{fandan})), when  $\inf_x {\rho_0(x)}>0$, some important progresses  on  the global well-posedness of  smooth solutions to the IBVP and Cauchy problem  of  \eqref{eq:1.1}  have been obtained, which include Constantin-Drivas-Nguyen-Pasqualotto \cite{cons}, Haspot \cite{HB}, Kang-Vasseur \cite{kv},  Mellet-Vasseur \cite{vassu2}  for 1-D flow with arbitrarily large data, and Sundbye \cite{Sundbye2}, Wang-Xu \cite{weike}    for initial data close to a non-vacuum equilibrium in some Sobolev spaces $H^s(\mathbb{R}^2)$ ($s>2$). 
 When vacuum appears, instead of the uniform elliptic structure in the constant viscous flow, the viscosity degenerates when density vanishes, which raises the difficulty of the problem to another level, i.e.,
\begin{equation}\label{doubled}
\displaystyle
 \underbrace{\rho(U_t+U\cdot \nabla U)}_{Degenerate   \ time \ evolution \ operator}+\nabla P= \underbrace{\text{div}(2\mu(\rho)D(U)+\lambda(\rho)\text{div}U \mathbb{I}_d)}_{Degenerate\ elliptic \ operator}.
\end{equation}
Usually, it is very hard to control the behavior of the fluids velocity $U$ near the vacuum.  A remarkable discovery of a new mathematical entropy function was made by Bresch-Desjardins \cite{bd6} for $\lambda(\rho)$ and $\mu(\rho)$   satisfying the relation 
\begin{equation}\label{bds}
 \lambda(\rho)=2(\mu'(\rho)\rho-\mu(\rho)),
\end{equation}
which offers an estimate
$$
\mu'(\rho)\nabla \sqrt{\rho}\in L^\infty ([0,T];L^2(\mathbb{R}^d))
$$
provided that 
$ \mu'(\rho_0)\nabla \sqrt{\rho_0}\in L^2(\mathbb{R}^d)$ for any $d\geq 1$.
This observation plays an important role in the development of the global existence of weak solutions with vacuum for system (\ref{eq:1.1}) and some related models, see Bresch-Desjardins \cite{bd6}, Bresch-Vasseur-Yu \cite{bvy},  Guo-Jiu-Xin \cite{gjx},  Guo-Li-Xin \cite{guo},  Li-Xin \cite{lz},  Vasseur-Yu \cite{vayu} and so on. However, the regularities and uniquness of such weak solutions with  vacuum  remain open.  Recently,  for the cases $\delta> 1$ in \eqref{fandan}, the  global well-posedness of regular solutions with vacuum for a class of smooth initial data that are of small density but possibly large velocities  in some homogeneous Sobolev spaces  has been established  by Xin-Zhu \cite{zz}.  Some other interesting results can also be seen in  Haspot \cite{Has}, Luo-Xin-Zeng \cite{luotao}, Yang-Zhao \cite{zyj} and so on.

In the current paper,  we establish  the global existence and uniqueness of  spherically symmetric smooth  solutions for (large)  initial data with spherical symmetry and  far field vacuum   in some inhomogeneous Sobolev spaces to the  IBVP \eqref{eq:1.1}-\eqref{10000} with   \eqref{eqs:CauchyInit}-\eqref{e1.3}  in the domain ($\Omega=\{x\in\mathbb R^d|\,  |x|>a\}$)  exterior to a ball in  $\mathbb R^d$ $(d=2\ \text{or} \ 3)$.
Notice that, under the assumption \eqref{bdrelation},  if  $\rho>0$, $\eqref{eq:1.1}_2$ can be formally rewritten as
\begin{equation}\label{qiyi}
\begin{split}
U_t+U\cdot\nabla U +\frac{A\gamma}{\gamma-1}\nabla\rho^{\gamma-1}+ LU=\nabla \ln\rho \cdot  Q(U),
\end{split}
\end{equation}
where the quantities $L U$ and  $Q(U)$ are given by
\begin{equation}\label{operatordefinition}
LU=-\alpha\Delta U-\alpha\nabla \text{div}U,\quad Q(U)= \alpha(\nabla U+(\nabla U)^\top).
\end{equation}
Therefore, the two quantities
$$(\rho^{\gamma-1}, \nabla \ln\rho)$$
will play significant roles in our analysis on the high order  regularities of the fluid velocity $U$. Due to  this observation, we first introduce a proper class of solutions called regular solutions to  the  IBVP \eqref{eq:1.1}-\eqref{10000} with  \eqref{eqs:CauchyInit}-\eqref{e1.3}.
\begin{mydef}\label{cjk}
Let $\Omega=\{x\in\mathbb R^d|\,  |x|>a\}$,  $\delta=1$, $q\in (3,6]$ and $T>0$. The pair $(\rho, U)$ is called a regular solution to the  IBVP \eqref{eq:1.1}-\eqref{10000} with  \eqref{eqs:CauchyInit}-\eqref{e1.3} in $[0,T]\times\Omega$, if $(\rho, U)$ satisfies this problem in the sense of distributions and:
\begin{equation*}
\begin{split}
 &(1)   \inf_{x\in \Omega} {\rho(t,x)}=0 \ \  \textrm{for} \ \ 0\leq t \leq T, \ \    0<\rho^{\gamma-1}\in C([0,T];H^2(\Omega)),\\
 &\ \    \ \ \nabla\ln\rho\in C([0,T]; L^q(\Omega)\cap D^1(\Omega));\\
 &(2) \  U\in C([0,T]; H^2(\Omega))\cap L^2([0,T]; D^3(\Omega)),\  U_t\in C([0,T]; L^2(\Omega))\cap L^2([0,T]; D^1(\Omega)).
\end{split}
\end{equation*}
\end{mydef}
Here and throughout this paper, we adopt the following  notations for the standard homogeneous  Sobolev spaces (see Galdi \cite{gandi}):
\begin{equation*}\begin{split}
&D^{k,r}(\Omega)=\{f\in L^1_{loc}(\Omega): \|\nabla^kf\|_{L^r(\Omega)}<\infty\}, \ \ \  D^k(\Omega)=D^{k,2}(\Omega).
\end{split}
\end{equation*}

Our main theorem on the  global-in-time  well-posedness theory  can be stated as follows.
\begin{thm}\label{th1}
Assume the physical parameters $(\delta, \gamma, \alpha,\beta)$ satisfy 
 \begin{equation}\label{cd1}
\delta=1,\quad \gamma>\frac32,\quad \alpha>0 \quad \text{and} \quad \beta=0.
\end{equation}
Let  $q\in(3,6]$ be some constant, and  the initial data $(\rho_0(x), U_0(x))= ( \rho_0(|x|),u_0(|x|)x/|x|)$ be   spherically symmetric and   satisfy
 \begin{equation}\label{id1}
 \begin{split}
&0<\rho_0(x)\in L^1(\Omega), \quad  (\rho_0^{\gamma-1}(x), U_0(x))\in H^2(\Omega),\\
& \nabla\ln\rho_0(x)\in L^q(\Omega)\cap L^\infty(\Omega)\cap D^1(\Omega).
\end{split}
\end{equation}
 Then  the  IBVP \eqref{eq:1.1}-\eqref{10000} with  \eqref{eqs:CauchyInit}-\eqref{e1.3} admits a unique global classical solution $(\rho(t,x), U(t,x))$
in $(0,\infty)\times\Omega$, and for any  $0<T<\infty$, $(\rho(t,x), U(t,x))$ is a regular one  in $[0,T]\times \Omega$ as defined in Definition \ref{cjk}, and
\begin{equation}\label{er2}
\begin{split}
&\rho\in C([0,T];L^1(\Omega)),
\quad (\rho^{\gamma-1})_t \in C([0,T]; H^1(\Omega)),\\
& (\nabla\ln\rho)_t\in C([0,T]; L^2(\Omega)),\quad \nabla\ln\rho\in L^\infty([0,T]\times \Omega),\\
& t^{\frac12}U_t\in L^\infty([0,T];D^1(\Omega))\cap L^2([0,T];D^2(\Omega)),\quad t^{\frac12}U_{tt}\in L^2([0,T];L^2(\Omega)).
\end{split}
\end{equation}
Moreover,  $(\rho(t,x), U(t,x))$ is also a spherically symmetric one  taking the form \eqref{duichenxingshi}.
\end{thm}

\begin{rk}
First, observe that we do not need any smallness assumption on $(\rho_0,U_0)$. 

Second, one can find the following class of spherically symmetric initial data $(\rho_0, U_0)$ satisfying the condition \eqref{id1}:
$$\rho_0(x)=\frac{1}{1+|x|^{2\sigma}},\quad U_0(x)=u_0(|x|)\frac{x}{|x|}\in C_0^2(\Omega),$$
for some $\sigma>\max\Big\{\frac{d}{2},\frac{d}{4(\gamma-1)}\Big\}=\frac{d}{2}$, and $u_0(r)\in C_0^2([a,\infty))$ with $u_0(a)=0$.

\end{rk}

Another purpose  of this article is to establish the well-posedness  with large data applicable to most of physically relevant models in shallow water theory. Our framework for system \eqref{shallow} is applicable  to the viscous terms of the forms  $\text{div}( h D(W))$, $\text{div}(2h D(W))$ and $\text{div}(h \nabla W)$. Actually,  when $\mathcal{V}=\text{div}(h D(W))$, this is a special case of
system \eqref{eq:1.1} with $\alpha=\frac{1}{2}$, $\beta=0$, $\delta=1$ and $\gamma=2$; when $\mathcal{V}=\text{div}(2h D(W))$, this is also a special case of system \eqref{eq:1.1} with $\alpha=1$, $\beta=0$, $\delta=1$ and $\gamma=2$.
Therefore, one simply replaces $(\rho,U)$ by $(h,W)$ in Theorem \ref{th1} to obtain the same conclusion  for these two classes of shallow water models, without further modifications. For the third case: $\mathcal{V}=\text{div}(h \nabla W)$, since the key information: BD entropy estimates (see Lemma \ref{l4.2}) in the proof of Theorem \ref{th1} also holds, then   one  can still obtain the desired global-in-time well-posedness from our framework with minor changes.   More precisely, considering the IBVP of system \eqref{shallow} with the initial data:
\begin{equation}\label{shallowchu}
  (h,W)(0,x)=(h_0,W_0)(x)= ( h_0(|x|),w_0(|x|)\dfrac{x}{|x|})\, \quad\, \ \text{for}\quad  x \in \Omega,
\end{equation}
and the following   boundary conditions and far field behavior:
\begin{equation}\label{shallowbian}
\begin{split}
\displaystyle
W(t,x)|_{|x|=a}=0\quad  \text{for}\quad   t\ge 0,\\[5pt]
\left(h(t,x),W(t,x)\right)\to \left(0,0\right)\quad  \text{as}\quad  \left|x\right|\to \infty\quad  \text{for}\quad   t\ge 0,&
\end{split}
\end{equation}
one has  the following theorem.
\begin{thm}\label{thshallow}
Assume  $\mathcal{V}=\text{div}(h D(W))$,  $\text{div}(2h D(W))$, or  $\text{div}(h \nabla W)$. Let  the initial data $(h_0(x), W_0(x)) = ( h_0(|x|),w_0(|x|)x/|x|)$ be   spherically symmetric and   satisfy
 \begin{equation*}
 \begin{split}
&0<h_0(x)\in L^1(\Omega), \quad  (h_0(x), W_0(x))\in H^2(\Omega),\\
& \nabla\ln h_0(x)\in L^q(\Omega)\cap L^\infty(\Omega)\cap D^1(\Omega),
\end{split}
\end{equation*}
for  some constant $q\in(3,6]$.
 Then  the  IBVP \eqref{shallow} with  \eqref{shallowchu}-\eqref{shallowbian} admits a unique global classical solution $(h(t,x), W(t,x))$
in $(0,\infty)\times\Omega$  satisfying  for any  $0<T<\infty$,
\begin{equation*}
\begin{split}
&\inf_{x\in \Omega} {h(t,x)}=0 \ \  \textrm{for} \ \ 0\leq t \leq T, \quad  0<h\in C([0,T];L^1(\Omega)\cap H^2(\Omega)),\\
&\nabla\ln h\in  C([0,T]; L^q(\Omega)\cap D^1(\Omega))\cap L^\infty([0,T]\times \Omega),\\
& W\in C([0,T]; H^2(\Omega))\cap L^2([0,T]; D^3(\Omega)),\quad h_t \in C([0,T]; H^1(\Omega)),\\
&   (\nabla \ln h)_t \in C([0,T]; L^2(\Omega)),\quad  W_t\in C([0,T]; L^2(\Omega))\cap L^2([0,T]; D^1(\Omega)),\\
& t^{\frac12}W_t\in L^\infty([0,T];D^1(\Omega))\cap L^2([0,T];D^2(\Omega)),\quad t^{\frac12}W_{tt}\in L^2([0,T];L^2(\Omega)).
\end{split}
\end{equation*}
Moreover,  $(h(t,x), W(t,x))$ is also a spherically symmetric one  taking the form:
\begin{equation*}
(h,W)(t,x)
=(h(t,|x|),w(t,|x|)\dfrac{x}{|x|}).
\end{equation*}
\end{thm}

The rest of the paper is organized as follows: \S 2 is devoted to establishing the local-in-time well-posedness of  regular  solutions  in the M-D Eulerian  coordinate  to the  IBVP \eqref{eq:1.1}-\eqref{10000} with  \eqref{eqs:CauchyInit}-\eqref{e1.3}.   Here we need to consider   a  well-designed reformulated structure by introducing some new variables, which, actually, can transfer the degeneracies of the time evolution and the viscosity shown in \eqref{doubled}  to the possible singularity of some special source terms. In \S 3, we prove the global-in-time well-posedness of regular solutions   in the spherically symmetric Eulerian  coordinate to the reformulated IBVP \eqref{e1.5} (see Theorem \ref{jordan}) by deriving some  a priori estimates globally in time.  Here,  we have employed some  arguments motivated by  Bresch-Desjardins \cite{bd6} and  Bresch-Desjardins-Lin \cite{bd2}   to deal with the strong degeneracy of the density-dependent viscosity   $\mu(\rho)=\alpha \rho^\delta$, which include the well-known  BD entropy estimates (see Lemma \ref{l4.2}) and the effective velocity (see Lemma \ref{l4.4}). In \S 4, we see that  the global well-posedness  asserted in Theorem \ref{th1}  can be obtained by means of  Theorem \ref{jordan}. We also give a brief proof for Theorem \ref{thshallow}.
  Finally, we give two appendixes to list some lemmas that are frequently used in our proof, and  the conversion of some Sobolev  spaces between   the pure M-D  coordinate and  the spherically symmetric coordinate.

\section{Local-in-time well-posedness}

This section is devoted to  showing  the local-in-time existence  of the unique  regular solution to  the IBVP   \eqref{eq:1.1}-\eqref{10000} with  \eqref{eqs:CauchyInit}-\eqref{e1.3} in M-D spaces. Moreover, we show that if the initial data is  spherically symmetric, so is the  corresponding multi-dimensional  regular solution to  the IBVP   \eqref{eq:1.1}-\eqref{10000} with  \eqref{eqs:CauchyInit}-\eqref{e1.3}  in positive time.   Throughout this section, we adopt the following simplified notations, most of them are for the standard homogeneous and inhomogeneous Sobolev spaces:
\begin{equation*}\begin{split}
 \displaystyle
 &  L^p=L^p(\Omega),\quad H^s=H^s(\Omega), \quad D^{k,l}=D^{k,l}(\Omega),\quad D^k=D^{k,2}(\Omega),  \\[1pt]
 \displaystyle
 & W^{m,p}=W^{m,p}(\Omega),\quad  |f|_p=\|f\|_{L^p(\Omega)},\quad \|f\|_s=\|f\|_{H^s(\Omega)},\quad  \|f\|_{m,p}=\|f\|_{W^{m,p}(\Omega)},\\[1pt]
   \displaystyle
& |f|_{D^{k,l}}=\|\nabla^k f\|_{L^l(\Omega)},\quad
|f|_{D^{k}}=\|\nabla^k f\|_{L^2(\Omega)},\quad \|f\|_{X_1 \cap X_2}=\|f\|_{X_1}+\|f\|_{X_2},
 \end{split}
\end{equation*}
where $\Omega=\{x\in\mathbb R^d|\,  |x|>a\}$ is the domain  exterior to a ball in  $\mathbb R^d$ $(d=2\ \text{or} \ 3)$.

The main theorem of this section can be stated as follows.
\begin{thm}\label{zth}
Assume $d=2$ or $3$, and \eqref{cd1} holds except $\gamma >\frac{3}{2}$.
 If the initial data $( \rho_0,U_0)$ satisfy 
\begin{equation}\label{th78qq}
0<\rho_0\in L^1,\quad  (\rho_0^{\gamma-1}, U_0)\in H^2, \quad  \nabla\ln\rho_0\in  L^q \cap D^1,
\end{equation}
for some $q\in(3,6]$, then there exist a time $T_*>0$ and a unique regular solution $( \rho,U)(t,x)$ in $[0,T_*]\times\Omega$ to the IBVP \eqref{eq:1.1}-\eqref{10000} with  \eqref{eqs:CauchyInit}-\eqref{e1.3} which satisfies \eqref{er2} with $T$ replaced by $T_*$. 
 Moreover, if $( \rho_0,U_0)$ are  spherically symmetric in the sense of \eqref{eqs:CauchyInit}, then 
  the solution  $( \rho,U)$  is also  a spherically symmetric one taking the form \eqref{duichenxingshi}.

\end{thm}

Next we  only give the proof of Theorem \ref{zth} in 3-D space in  the following Sections \ref{see1}-\ref{com}, and   the 2-D case can be dealt with via the  completely same argument. We first make a reformulation for the  IBVP \eqref{eq:1.1}-\eqref{10000} with  \eqref{eqs:CauchyInit}-\eqref{e1.3}.





\subsection{Reformulation}\label{see1}
Via introducing the following new variables 
\begin{equation}\label{bianhuan}
\phi=\frac{A\gamma}{\gamma-1}\rho^{\gamma-1},\quad \boldsymbol{\psi} =\frac{1}{\gamma-1}\nabla\ln\phi=\nabla\ln\rho,
\end{equation}
 the  IBVP \eqref{eq:1.1}-\eqref{10000} with  \eqref{eqs:CauchyInit}-\eqref{e1.3} can be rewritten as 
\begin{equation}\label{eqn1}
\left\{
\begin{aligned}
&\ \phi_t+U\cdot \nabla\phi+(\gamma-1)\phi \text{div}U=0,\\[3pt]
&\ U_t+U\cdot \nabla U+\nabla \phi+LU=\boldsymbol{\psi}\cdot Q(U),\\
&\boldsymbol{\psi}_t+\sum\limits_{l=1}^3A_l(U)\partial_l\boldsymbol{\psi}+B(U) \boldsymbol{\psi}+\nabla\text{div}U=0,\\
&\ \big(\phi(0,x),U(0,x),\boldsymbol{\psi}(0,x)\big)=\Big(\frac{A\gamma}{\gamma-1}\rho_0^{\gamma-1},U_0,\nabla \ln \rho_0\Big)\quad \text{for} \quad x\in \Omega,\\[3pt]
&\ U(t,x)|_{|x|=a}=0  \quad \text{for} \quad t\ge 0,\\[3pt]
&\  (\phi(t,x),U(t,x),\boldsymbol{\psi}(t,x))\to (0,0,0) \quad \text{as}\quad  |x| \to \infty  \quad \text{for} \quad  t\geq 0,
\end{aligned}
\right.
\end{equation}
where    $A_l(U)=(a_{ij}^{(l)})_{3\times3}(i,j,l=1,2,3)$ are symmetric with $a_{ij}^{(l)}=U^{(l)}$ for $i=j$;  otherwise  $a_{ij}^{(l)}=0$, and $B(U)=(\nabla U)^{\top}$.
In order to solve the  IBVP \eqref{eq:1.1}-\eqref{10000} with  \eqref{eqs:CauchyInit}-\eqref{e1.3} locally in time, we need to prove the following theorem.

\begin{thm}\label{thh1}
Assume $d=3$, and \eqref{cd1} holds  except $\gamma >\frac{3}{2}$. If the initial data $\big(\phi_0(x), U_0(x),\boldsymbol{\psi}_0(x)\big)=\big(\phi(0,x),U(0,x),\boldsymbol{\psi}(0,x)\big)$ satisfy \eqref{th78qq},
then there exist a time $T_*>0$ and a unique strong solution $( \phi,U,\boldsymbol{\psi})$ in $[0,T_*]\times\Omega$ to the IBVP \eqref{eqn1},  satisfying
\begin{equation}\label{reg11qq}\begin{split}
& 0<\phi \in C([0,T_*];L^{\frac{1}{\gamma-1}}\cap H^2),\quad  \phi_t \in C([0,T_*];H^1),\\
& \boldsymbol{\psi} \in C([0,T_*]; L^q\cap D^1),\quad  \boldsymbol{\psi}_t \in C([0,T_*]; L^2),\\
& U\in C([0,T_*]; H^2)\cap L^2([0,T_*]; D^3),\quad U_t\in C([0,T_*]; L^2)\cap L^2([0,T_*]; D^1),\\
&t^{\frac12} U_t\in L^{\infty}([0,T_*]; D^1)\cap L^2([0,T_*]; D^2),\quad t^{\frac12} U_{tt} \in L^2([0,T_*]; L^2).
\end{split}
\end{equation}
\end{thm}


\subsection{Linearization}\label{linear2}
We begin the proof of Theorem \ref{thh1} by   considering  the following linearized problem of $(\phi, U,\boldsymbol{\psi})$:
\begin{equation}\label{li4}
\left\{
\begin{aligned}
&\ \phi_t+V\cdot \nabla\phi+(\gamma-1)\phi \text{div}V=0,\\[3pt]
&\ U_t+V\cdot \nabla V+\nabla \phi+LU=\boldsymbol{\psi}\cdot Q(V),\\
&\boldsymbol{\psi}_t+\sum\limits_{l=1}^3A_l(V)\partial_l\boldsymbol{\psi}+B(V) \boldsymbol{\psi}+\nabla\text{div}V=0,\\
&\ \big(\phi(0,x),U(0,x),\boldsymbol{\psi}(0,x)\big)=\Big(\frac{A\gamma}{\gamma-1}\rho_0^{\gamma-1},U_0,\nabla \ln \rho_0\Big)\quad \text{for} \quad x\in \Omega,\\[3pt]
&\ U(t,x)|_{|x|=a}=0  \quad \text{for} \quad t\ge 0,\\[3pt]
&\  (\phi(t,x),U(t,x),\boldsymbol{\psi}(t,x))\to (0,0,0) \quad \text{as}\quad  |x| \to \infty  \quad \text{for} \quad  t\geq 0,
\end{aligned}
\right.
\end{equation} 
where 
$V=\left(V^{(1)},V^{(2)}, V^{(3)}\right)^{\top}\in \mathbb{R}^3$ is a known vector satisfying for any $T>0$,
\begin{equation}\label{vg}
\begin{split}
&V(t=0,x)=U(0,x),\quad  V(t,x)|_{|x|=a}=0 \quad \text{for} \quad 0\leq t \leq T,\\
&V\in C([0,T]; H^2)\cap L^2([0,T]; D^3),\quad V_t\in C([0,T]; L^2)\cap L^2([0,T]; D^1),\\
&t^{\frac12} V_t\in L^{\infty}([0,T]; D^1)\cap L^2([0,T]; D^2),\quad t^{\frac12} V_{tt} \in L^2([0,T]; L^2).
\end{split}
\end{equation}

Next, the following global well-posedness of strong solutions $(\phi, U,\boldsymbol{\psi})$ to \eqref{li4} can be obtained by classical  arguments shown in  (\cite{CK3, oar, amj}). 
 \begin{lem}\label{lem1}
 Assume $d=3$, \eqref{cd1} holds  except $\gamma >\frac{3}{2}$ and $T>0$ is an arbitrarily large time. If the initial data $\big(\phi_0(x), U_0(x),\boldsymbol{\psi}_0(x)\big)=\big(\phi(0,x),U(0,x),\boldsymbol{\psi}(0,x)\big)$ satisfy \eqref{th78qq}, then there exists a unique strong solution $(\phi,U,\boldsymbol{\psi})$ in $[0,T]\times\Omega$ to (\ref{li4}) which satisfies the regularities in \eqref{reg11qq} with $T_*$ replaced by $T$.
\end{lem}

\subsection{A priori estimates}\label{uape}
Let $(\phi, U,\boldsymbol{\psi})$ be a  strong solution in $[0,T]\times\Omega$ obtained in Lemma \ref{lem1}, and we  will establish the corresponding  a priori estimates later. For this purpose, we first choose a positive constant $c_0$ such that 
\begin{equation}\label{houmian}
\begin{split}
1+\|\phi_0\|_{L^{\frac{1}{\gamma-1}}\cap H^2}+\|U_0\|_{2}+\|\boldsymbol{\psi}_0\|_{L^q\cap D^1}\leq c_0.
\end{split}
\end{equation}
We  assume there exist some time $T^*\in (0,T)$ and constants $c_i$ ($i=1,2,3$) such that
$$1< c_0\leq c_1 \leq c_2 \leq c_3, $$
and
\begin{equation}\label{jizhu1}
\begin{split}
\sup_{0\leq t \leq T^*}\| V(t)\|^2_1+\int_{0}^{T^*}\Big(  \|\nabla V(t)\|^2_{1}+|V_t(t)|_2^2\Big)\text{d}t \leq& c^2_1,\\
\sup_{0\leq t \leq T^*}\Big(|V(t)|^2_{D^2}+|V_t(t)|^2_{2}\Big)+\int_{0}^{T^*} \Big( |V(t)|^2_{D^3}+|V_t(t)|^2_{D^1}\Big)\text{d}t \leq& c^2_2,\\
\text{ess}\sup_{0\leq t \leq T^*} t|V_t(t)|^2_{D^1}+\int_{0}^{T^*} t\Big(|V_{t}(t)|^2_{D^2}+|V_{tt}(t)|^2_{2}\Big)\text{d}t \leq& c^2_3,
\end{split}
\end{equation}
where  $T^*$ and  $c_i$ ($i=1,2,3$) will be determined  later, and depend only on $c_0$ and the fixed constants $A,\alpha, \gamma, T$.
 In the rest of \S 2, we use $C\geq 1$ to denote  a generic  constant depending only on fixed constants $A,\alpha, \gamma,T$ and may be different from line to line.
\subsubsection{The a priori estimates for $\phi$}
 We first give the estimates for $\phi$.

\begin{lem}\label{bos} 
\begin{equation*}
\|\phi(t)\|_{L^{\frac{1}{\gamma-1}}\cap H^2}\leq Cc_0, \quad |\phi_t(t)|_2\leq Cc_0c_1,\quad  |\phi_t(t)|_{D^1}\leq Cc_0c_2
\end{equation*}
for $0\leq t \leq T_1=\min \{T^{*}, (1+Cc_3)^{-2}\}$.
\end{lem}

\begin{proof}

It follows from \eqref{jizhu1} and standard arguments  for transport equation that 
\begin{equation}\label{gb}
\begin{split}
\|\phi(t)\|_{L^{\frac{1}{\gamma-1}}\cap H^2}\leq & \|\phi_0\|_{L^{\frac{1}{\gamma-1}}\cap H^2} \exp\Big(C\int_0^t \| V(s)\|_{3}\text{d}s\Big)\leq Cc_0,
\end{split}
\end{equation}
for $0\leq t \leq T_1$, which, along with   $\eqref{li4}_1$, H\"older's inequality and \eqref{gb}, yields that
\begin{equation}\label{zhen6}
\begin{split}
|\phi_t(t)|_2\leq&  Cc_0c_1,\quad 
| \phi_t(t)|_{D^1}\leq Cc_0c_2 \quad \text{for} \quad   0\le t\le T_1.
\end{split}
\end{equation}

The proof of  Lemma \ref{bos} is complete.
\end{proof}

\subsubsection{The a priori estimates for $\boldsymbol{\psi}$}

Now we establish the estimates for $\boldsymbol{\psi}$.

\begin{lem}\label{pij} 
\begin{equation*}
\|\boldsymbol{\psi}(t)\|_{L^q\cap D^1}\leq Cc_0,\quad |\boldsymbol{\psi}_t(t)|_2\leq Cc_0c_2 \quad \ {\rm{for}} \quad 0\le t\le T_1.
\end{equation*}

\end{lem}

\begin{proof}
First, multiplying $\eqref{li4}_3$ by $q|\boldsymbol{\psi}|^{q-2}\boldsymbol{\psi}$ and integrating over $\Omega$, one has
\begin{equation*}
\frac{\text{d}}{\text{d}t}|\boldsymbol{\psi}|^q_q\leq C\big(|\nabla V|_\infty|\boldsymbol{\psi}|_q^q+|\nabla^2V|_{q}|\boldsymbol{\psi}|_q^{q-1}\big)\le C\big(|\nabla V|_\infty|\boldsymbol{\psi}|_q^q+\|\nabla^2 V\|_1|\boldsymbol{\psi}|_q^{q-1}\big),
\end{equation*}
which, along with \eqref{jizhu1} and the Gronwall inequality, yields that
\begin{equation}\label{beimian}|\boldsymbol{\psi}(t)|_q\le Cc_0\quad \text{for}\quad 0\le t\le T_1.\end{equation}

Second,  let $\varsigma=(\varsigma_1, \varsigma_2, \varsigma_3)^{\top} (|\varsigma|=1\  \text{and}\  \varsigma_i=0,1)$. Applying $\partial_x^{\varsigma}$ to $\eqref{li4}_3$,  multiplying by $\partial_x^\varsigma \boldsymbol{\psi}$ and then integrating over $\Omega$, one can get
\begin{equation*}
\begin{split}
\frac{\text{d}}{\text{d}t} |\nabla \boldsymbol{\psi}|_2
\le &C\big(|\nabla V|_{\infty} | \nabla\boldsymbol{\psi}|_2 +\|\nabla^2 V\|_{1}\big),
\end{split}
\end{equation*}
which, along with \eqref{jizhu1} and the Gronwall inequality, yields that 

\begin{equation}\label{y4}
|\boldsymbol{\psi}(t)|_{D^1}\le Cc_0 \quad \ {\rm{for}} \quad 0\le t\le T_1.
\end{equation}

Finally, it follows from the equations $\eqref{li4}_3$, \eqref{beimian} and \eqref{y4} that for $ 0\le t\le T_1,$
\begin{align*}
|\boldsymbol{\psi}_t|_2\le &C\big(|\nabla V|_{q^{*}} |\boldsymbol{\psi}|_q+|V|_{\infty}|\nabla \boldsymbol{\psi}|_2+|\nabla^2 V|_2\big)\le Cc_0c_2,
\end{align*}
where $q^{*}=\frac{2q}{q-2}\in[3,6)$. 

The proof of  Lemma \ref{pij} is complete.
\end{proof}

\subsubsection{The a priori estimates for $U$}

\begin{lem}\label{llm3}
\begin{equation*}
\begin{split}
&\|U(t)\|_1^2+|U_t(t)|_2^2+\int_0^t \big( \| \nabla U(s)\|_{1}^2+|U_t(s)|_2^2+|U_t(s)|_{D^1}^2\big) {\rm{d}} s\le Cc_0^2,\\
&|U(t)|_{D^2}^2+t|U_t(t)|_{D^1}^2+\int_0^t \big(|U(s)|_{D^3}^2+s|U_{tt}(s)|_2^2+s|U_t(s)|_{D^2}^2\big){\rm{d}} s\le Cc_1^3 c_2^{\frac6q},
\end{split}
\end{equation*}
for $ 0\le t\le T_2=\min\big\{T^*, (1+Cc_3)^{-\frac{6q}{q-3}}\big\}$.

\end{lem}
\begin{proof}
\textbf{Step 1:} estimate on $|U|_2$.  Multiplying  $\eqref{li4}_2$ by $U$ and integrating over $\Omega$, then
\begin{equation}\label{y5}
\begin{split}
&\frac12\frac{\text{d}}{\text{d}t} |U|_2^2+\alpha |\nabla U|_2^2+\alpha|\text{div}U|_2^2
=\int_{\Omega} \big(-V\cdot \nabla V-\nabla\phi +\boldsymbol{\psi}\cdot Q(V)\big)\cdot U\text{d}x \\
\le & C\big( |V|_\infty |\nabla V|_2 + |\nabla\phi|_2+|\boldsymbol{\psi}|_q |\nabla V|_{q^{*}}\big)|U|_2 \le  C|U|_2^2+Cc_2^4,
\end{split}
\end{equation}
which, along with the Gronwall inequality, yields that 
\begin{equation}\label{dilingjie}
|U(t)|_2^2+\int_0^t | U(s)|_{D^1}^2 {\rm{d}} s\le Cc_0^2\quad \ {\rm{for}} \quad 0\le t\le T_2.
\end{equation}

\textbf{Step 2:} estimate on $|\nabla U|_2$. 
Multiplying  $\eqref{li4}_2$ by $U_t$ and integrating over $\Omega$, one gets
\begin{equation}\label{xl}
\begin{split}
&\frac{\text{d}}{\text{d}t} \big(\alpha |\nabla U|_2^2+\alpha |\text{div} U|_2^2\big)+|U_t|_2^2
=\int_{\Omega} (-V\cdot \nabla V-\nabla\phi +\boldsymbol{\psi}\cdot Q(V))\cdot U_t\text{d}x \\
\le & C( |V|_\infty |\nabla V|_2 + |\nabla\phi|_2+|\boldsymbol{\psi}|_q |\nabla V|_{q^{*}})|U_t|_2 \le  Cc_2^4+\frac14|U_t|_2^2,
\end{split}
\end{equation}
which, along with  the Gronwall inequality, implies that
\begin{equation}\label{diyijie}
| U(t)|_{D^1}^2+\int_0^t |U_t(s)|_2^2 {\rm{d}} s\le Cc_0^2 \quad \ {\rm{for}} \quad 0\le t\le T_2.
\end{equation}

It follows from   Lemma \ref{df3} (Appendix A) that 
\begin{equation}\label{y7}
\begin{split}
|U|_{D^2}\le &C\big( | -U_t-V\cdot \nabla V-\nabla\phi +\boldsymbol{\psi}\cdot Q(V)|_2+ |U|_{D^1}\big)\\
\le & C\big(|U_t|_2+ |V|_6 |\nabla V|_3 + |\nabla\phi|_2+|\boldsymbol{\psi}|_q |\nabla V|_{q^{*}}+ |U|_{D^1}\big)\\
\le & C\big(|U_t|_2+c_1^2 +c_1^{\frac32}c_2^{\frac3q}\big),
\end{split}
\end{equation}
where one has used the fact that (see Lemma \ref{ale1} ( Appendix A)) 
\begin{equation}\label{momo}
\begin{split}
|f|_3\le& |f|_2^{\frac12}|f|_6^{\frac12}\le C |f|_2^{\frac12}\|f\|_1^{\frac12}
\le C\big(|f|_2+|f|_2^{\frac12}|\nabla f|_2^{\frac12}\big),\\
|f|_{q^*}\le& |f|_2^{\frac{q-3}{q}}|f|_6^{\frac3q}\le C |f|_2^{\frac{q-3}{q}}\|f\|_1^{\frac3q}
\le C\big(|f|_2+|f|_2^{\frac{q-3}{q}}|\nabla f|_2^{\frac3q}\big).
\end{split}
\end{equation}

Consequently, it follows from \eqref{diyijie}-\eqref{y7} that
\[
\int_0^t |U(s)|_{D^2}^2 {\rm{d}} s\le Cc_0^2 \quad \ {\rm{for}} \quad 0\le t\le T_2.
\]

\textbf{Step 3:} estimate on $|U|_{D^2}$.
Differentiating $\eqref{li4}_2$ with respect to $t$, multiplying the resulting equation   by $U_t$ and integrating over $\Omega$, one has
\begin{align}\label{2.19kkk}
&\frac12\frac{\text{d}}{\text{d}t} |U_t|_2^2 +\alpha |\nabla U_t|_2^2+\alpha|\text{div} U_t|_2^2
= \int_{\Omega} \big( -(V\cdot \nabla V)_t-\nabla\phi_t +(\boldsymbol{\psi}\cdot Q(V))_t\big)\cdot U_t\text{d}x\nonumber \\
\le& C\Big((|V|_{\infty} |\nabla V_t|_2+|V_t|_6|\nabla V|_3)|U_t|_2+|\phi_t|_2 |\nabla U_t|_2 \nonumber\\
&+ |\boldsymbol{\psi}|_q  |\nabla V_t|_2|U_t|_{q^{*}} +| \boldsymbol{\psi}_t|_2 |\nabla V|_6|U_t|_3\Big)\\
\le&C\big(1+c_2^4+c_2^{\frac{4q}{q-3}}\big)|U_t|_2^2+\frac{1}{c_2^2}\|V_t\|_1^2+\frac14 |\nabla U_t|_2^2+Cc_2^6,\nonumber
\end{align}
where one has used \eqref{momo} for dealing with the terms $(|U_t|_{q^{*}},|U_t|_{3})$. Then 
 integrating \eqref{2.19kkk} over $(\tau, t) (\tau\in(0,t))$, one arrives at
 \begin{equation}\label{polo}
  |U_t(t)|_2^2 + \int_{\tau}^t |\nabla U_t(s)|_2^2 \text{d} s \le   |U_t(\tau)|_2^2 +C \int_{\tau}^t \Big(1+c_2^{\frac{4q}{q-3}}\Big)|U_t(s)|_2^2\text{d} s+Cc_2^6t+C.
 \end{equation}
It follows from   the equations $\eqref{li4}_2$, 
 the continuity of $(\phi, U,\boldsymbol{\psi})$ and \eqref{vg}-\eqref{houmian} that 
 \begin{equation}
 \lim\sup_{\tau\to 0}|U_t(\tau)|_{2}\le C \big(|U_0|_{\infty} |\nabla U_0|_2 +|\nabla\phi_0|_2+ |\nabla^2 U_0|_2 +|\boldsymbol{\psi}_0|_q |\nabla U_0|_{q^{*}}\big)\le Cc_0^2.
 \end{equation}
Letting $\tau\to 0$ in \eqref{polo} and using the Gronwall inequality, one can get 
\begin{equation}\label{y9}
|U_t(t)|_2^2 +\int_0^t |U_t(s)|_{D^1}^2 \text{d}s \le Cc_0^2 \quad \ {\rm{for}} \quad 0\le t\le T_2.
\end{equation}
It thus  follows from \eqref{y7}  that
\[
|U(t)|_{D^2}\le  C\big(|U_t(t)|_2+c_1^2 +c_1^{\frac32}c_2^{\frac3q}\big)\le C c_1^{\frac32}c_2^{\frac3q} \quad \ {\rm{for}} \quad 0\le t\le T_2.
\]

Finally, according to  Lemma \ref{df3} (Appendix A), one has
\begin{equation}
\begin{split}
|U|_{D^3}\le &C\big( \| -U_t-V\cdot \nabla V-\nabla\phi +\boldsymbol{\psi}\cdot Q(V)\|_1+ |U|_{D^1}\big)\\
\le & C\big(\|U_t\|_1+ |V|_\infty \|\nabla V\|_1+|\nabla V|_6 |\nabla V|_3 + \|\nabla\phi\|_1\\
&+|\boldsymbol{\psi}|_q\|\nabla V\|_{1,q^{*}}+|\nabla\boldsymbol{\psi}|_2 |\nabla V|_\infty+ |U|_{D^1}\big)
\le  C\big(|\nabla U_t|_2 +c_2^2+c_0|V|_{D^3}\big),
\end{split}
\end{equation}
which, along with \eqref{y9}, implies that
\[
\int_0^t |U(s)|_{D^3}^2 \text{d}s \le Cc_0^2 \quad \ {\rm{for}} \quad 0\le t\le T_2.
\]
%
%
%
\textbf{Step 4:} time-weighted estimate on $U$. 
First,  differentiating $\eqref{li4}_2$ with respect to $t$, multiplying the resulting equations   by $U_{tt}$ and integrating over $\Omega$, one has
\begin{equation*}
\begin{split}
&\frac{\alpha}{2}\frac{\text{d}}{\text{d}t}\big(|\nabla U_t|_2^2+|\text{div} U_t|_2^2\big)+|U_{tt}|_2^2=-\int_{\Omega}\big((V\cdot\nabla V)_t+\nabla\phi_t-(\boldsymbol{\psi}\cdot Q(V))_t\big)U_{tt}\text{d}x\\
\le &C|U_{tt}|_2\Big(|V_t|_2|\nabla V|_\infty+|V|_\infty|\nabla V_t|_2+|\nabla\phi_t|_2+|\boldsymbol{\psi}_t|_2|\nabla V|_\infty+|\boldsymbol{\psi}|_q|\nabla V_t|_{q^*}\Big),
\end{split}
\end{equation*}
which, along with Young's inequality and \eqref{momo}, yields that
\begin{equation}\label{lrq}
\frac{\text{d}}{\text{d}t}\big(|\nabla U_t|_2^2+|\text{div} U_t|_2^2\big)+|U_{tt}|_2^2\le C\big(c_2^4\|V\|_3^2+c_2^2|\nabla V_t|_2^2+c_2^4+c_0^2|\nabla V_t|_2^{\frac{2(q-3)}{q}}\|\nabla V_t\|_1^{\frac6q}\big).
\end{equation}

Second, multiplying \eqref{lrq} by $s$ and integrating over $(\tau,t)$, one can get 
\begin{equation}\label{yilu}
\begin{split}
t|\nabla U_t|_2^2+\int_\tau^t s|U_{tt}|_2^2\text{d}s\le& 
 C\Big(\tau|\nabla U_t(\tau)|_2^2+c_0^2+c_3^6 t+c^4_3t^{\frac{q-3}{q}}\Big).
\end{split}
\end{equation}
Due to  \eqref{reg11qq} and Lemma \ref{bjr}  (Appendix A) that,  there exists a sequence $\{s_k\}$ such that 
\begin{equation}
s_k\to 0 \quad \text{and}\quad  s_k|\nabla U_t(s_k, \cdot)|_2^2\to 0 \quad \text{as}\quad k\to \infty.
\end{equation}
After choosing $\tau=s_k\to 0$  in  \eqref{yilu}, one  obtains 
\begin{equation}\label{cv}
t| U_t(t)|_{D^1}^2+\int_0^t s|U_{tt}(s)|_2^2\text{d}s\le Cc_0^2 \quad \ {\rm{for}} \quad 0\le t\le T_2.
\end{equation}

Finally, according to   Lemma \ref{df3} (Appendix A), one has
\begin{align*}
|U_t|_{D^2}\le& C\big(|U_{tt}|_2+|(V\cdot \nabla V)_t|_2+|\nabla \phi_t|_2+|(\boldsymbol{\psi}\cdot Q(V))_t|_2+| U_t|_{D^1}\big)\\
\le & C\big(|U_{tt}|_2+|V|_\infty|\nabla V_t|_2+|V_t|_2|\nabla V|_\infty+|\nabla \phi_t|_2+|\boldsymbol{\psi}|_q|\nabla V_t|_{q^*}\\
&+|\boldsymbol{\psi}_t|_2|\nabla V|_\infty+|\nabla U_t|_2\big),
\end{align*}
which, along with \eqref{momo} and \eqref{cv}, yields that for $0\le t\le T_2$,
\begin{align*}
\int_0^t s|U_t(s)|_{D^2}^2\text{d}s\le &C\big(c_0^2+c_3^4 t+c_3^6 t+c_3^4 t^{\frac{q-3}{q}}\big)\le Cc_0^2.
\end{align*}

The proof of  Lemma \ref{llm3} is complete.
\end{proof}

Finally, define the time $T^*=\min\big\{ T,(1+Cc_3)^{-\frac{6q}{q-3}}\big\}$,
and constants $c_i(i=1,2,3)$:
$$c_1=C^{\frac12}c_0, \quad c_2=c_3=C^{\frac{q}{2(q-3)}}c_1^{\frac{3q}{2(q-3)}}=C^{\frac{5q}{4(q-3)}}c_0^{\frac{3q}{2(q-3)}}.$$
It follows from Lemmas \ref{bos}-\ref{llm3} that for $0\le t\le T^*$,
\begin{equation}\label{lgg}
\begin{split}
\|U(t)\|_1^2+\int_0^t \big(\| \nabla U(s)\|_{1}^2+|U_t(s)|_2^2\big) {\rm{d}} s\le& c_1^2,\\
|U(t)|_{D^2}^2+|U_t(t)|_2^2+\int_0^t \big(|U_t(s)|_{D^1}^2+| U(s)|_{D^3}^2\big) {\rm{d}} s\le& c_2^2,\\
 t|U_t(t)|_{D^1}^2+\int_0^t \big(s|U_{tt}(s)|_2^2+s|U_t(s)|_{D^2}^2\big)\text{d}s\le &c_3^2,\\
\|\phi(t)\|_{L^{\frac{1}{\gamma-1}}\cap H^2}+ \|\phi_t(t)\|_{L^2\cap D^1}+\|\boldsymbol{\psi}(t)\|_{L^q\cap D^1}+ |\boldsymbol{\psi}_t(t)|_2\le& c_3^2.
\end{split}
\end{equation}

\subsection{Proof of Theorem \ref{thh1}}\label{bani} 
Our proof  is based on the classical iteration scheme and conclusions obtained  in  Sections 2.2-2.3. Let us denote as in Section 2.3 that 
\begin{equation*}
\begin{split}
&1+\|\phi_0\|_{L^{\frac{1}{\gamma-1}}\cap H^2}+\|U_0\|_{2}+\|\boldsymbol{\psi}_0\|_{L^q\cap D^1} \leq c_0.
\end{split}
\end{equation*}
Next, let  $U^0\in C([0,T^*];H^2)\cap  L^2([0,T^*];H^3) $ satisfy the following  problem:
\begin{equation*}\begin{split}
& \zeta_t-\triangle \zeta=0 \quad  \text{in} \quad  (0,\infty)\times \Omega;\quad \zeta(0,x)=U_0 \quad  \text{for} \quad  x\in\Omega; \\
& \zeta(t,x)|_{|x|=a}=0\quad  \text{and} \quad \zeta(t,x)\to 0 \quad    \text{as}\quad  \left|x\right|\to \infty\quad  \text{for}\quad  t\ge 0.
\end{split}
\end{equation*}
Choosing a time $T_{**}\in (0,T^*]$ small enough such that
\begin{equation}\label{jizhu}
\begin{split}
\sup\limits_{0\le t\le T_{**}}\|U^0(t)\|_1^2+\int_0^{T_{**}}\big( \|\nabla  U^0(t)\|_{1}^2+|U^0_t(t)|_2^2\big) {\rm{d}} t\le& c_1^2,\\
\sup_{0\leq t \leq T_{**}}\big(|U^0(t)|^2_{D^2}+|U^0_t(t)|^2_{2}\big)+\int_{0}^{T_{**}} \big( |U^0(t)|^2_{D^3}+|U^0_t(t)|^2_{D^1}\big)\text{d}t \leq& c^2_2,\\
\text{ess}\sup_{0\leq t \leq T_{**}} t|U^0_t(t)|^2_{D^1}+\int_{0}^{T_{**}}\big( t|U^0_{tt}(t)|^2_{2}+t|U^0_t(t)|^2_{D^2}\big)\text{d}t \leq& c^2_3.
\end{split}
\end{equation}

\begin{proof}
\textbf{Step 1:} existence. 
Let the beginning step of our iteration be $V=U^0$. Then one can get a strong solution $\big(\phi^1,U^1,\boldsymbol{\psi}^1 \big)$ of problem \eqref{li4}. Inductively, one constructs approximate sequences $\big(\phi^{k+1}, U^{k+1},\boldsymbol{\psi}^{k+1}\big)$ as follows:   given $(\phi^{k},U^{k},\boldsymbol{\psi}^k)$ for $k\geq 1$, define $(\phi^{k+1}, U^{k+1},\boldsymbol{\psi}^{k+1})$  by solving the following problem
\begin{equation}\label{li6}
\begin{cases}
\ \ \displaystyle \phi^{k+1}_t+U^{k}\cdot \nabla \phi^{k+1}+(\gamma-1)\phi^{k+1}\text{div} U^{k}=0,\\[2pt]
\ \ \displaystyle U^{k+1}_t+U^{k}\cdot\nabla U^{k} +\nabla \phi^{k+1}+LU^{k+1}=\boldsymbol{\psi}^{k+1}\cdot Q(U^{k}),\\[2pt]
\ \ \displaystyle \boldsymbol{\psi}^{k+1}_t+\sum_{l=1}^3 A_l(U^k) \partial_l\boldsymbol{\psi}^{k+1}+B(U^k)\boldsymbol{\psi}^{k+1}+\nabla \text{div}U^k=0,\\[2pt]
\ \ \displaystyle (\phi^{k+1}, U^{k+1}, \boldsymbol{\psi}^{k+1} )|_{t=0}=(\phi_0,U_0, \boldsymbol{\psi}_0)\quad \text{for}\quad x\in\Omega,\\[2pt]
\ \ \displaystyle U^{k+1}|_{|x|=a}=0\quad\text{for}\quad t\ge 0,\\[2pt]
\ \ \displaystyle (\phi^{k+1},U^{k+1}, \boldsymbol{\psi}^{k+1})\to (0,0,0)\quad \text{as}\quad |x|\to\infty\quad \text{for}\quad t\ge 0.
 \end{cases}
\end{equation}
By replacing $V$  with $U^k$ in \eqref{li4}, and $(\phi^k, U^k, \boldsymbol{\psi}^k)$ satisfies the uniform estimates \eqref{lgg}, we can solve the problem \eqref{li6}.

Next we are going to prove that the whole sequence $(\phi^k,U^k,\boldsymbol{\psi}^k)$ converges strongly to a limit $(\phi,U,\boldsymbol{\psi})$. Let
\begin{equation*}
\overline{\phi}^{k+1}=\phi^{k+1}-\phi^k,\quad  \overline{U}^{k+1}=U^{k+1}-U^k,\quad  \overline{\boldsymbol{\psi}}^{k+1}=\boldsymbol{\psi}^{k+1}-\boldsymbol{\psi}^k.
\end{equation*}
Then from \eqref{li6},  one can deduce that
 \begin{equation}
\label{eq:1.2w}
\begin{cases}
\ \  \displaystyle \overline{\phi}^{k+1}_t+U^k\cdot \nabla\overline{\phi}^{k+1} +\overline{U}^k\cdot\nabla\phi ^{k}+(\gamma-1)\big(\overline{\phi}^{k+1} \text{div}U^k +\phi ^{k}\text{div}\overline{U}^k\big)=0,\\[2pt]
\ \ \overline{U}^{k+1}_t+ U^k\cdot\nabla \overline{U}^{k}+ \overline{U}^{k} \cdot \nabla U^{k-1} +\nabla\overline\phi^{k+1} +L\overline{U}^{k+1} \\[2pt]
=\boldsymbol{\psi}^{k+1}\cdot Q(\overline{U}^k)+\overline{\boldsymbol{\psi}}^{k+1}\cdot Q(U^{k-1}),\\
\ \ \displaystyle \overline{\boldsymbol{\psi}}^{k+1}_t+\sum_{l=1}^3 A_l(U^k) \partial_l\overline{\boldsymbol{\psi}}^{k+1}+B(U^k)\overline{\boldsymbol{\psi}}^{k+1}+\nabla \text{div}\overline{U}^k=\Upsilon^k_1+\Upsilon^k_2,
\end{cases}
\end{equation}
where $\Upsilon^k_1$  and $\Upsilon^k_2$ are defined by
\begin{equation*}
\Upsilon^k_1=-\sum_{l=1}^3(A_l(U^k) \partial_l\boldsymbol{\psi}^{k}-A_l(U^{k-1}) \partial_l\boldsymbol{\psi}^{k}),\quad \Upsilon^k_2=-(B(U^k) \boldsymbol{\psi}^{k}-B(U^{k-1}) \boldsymbol{\psi}^{k}).
\end{equation*}

We first consider  $\overline{\boldsymbol{\psi}}^{k+1}$. Actually, from Remark \ref{lalala} at the end of this section, one has 
\begin{lem}\label{lpsi}
\begin{equation*}
\overline{\boldsymbol{\psi}}^{k+1}\in L^\infty([0,T_{**}];H^1)\quad \text{for} \quad k=1,2,\cdots.
\end{equation*}
\end{lem}
Then  multiplying $(\ref{eq:1.2w})_3$ by $2\overline{\boldsymbol{\psi}}^{k+1}$ and integrating over $\Omega$, one arrives at
\begin{equation*}\label{go64aa}
\begin{split}
\frac{\text{d}}{\text{d}t}|\overline{\boldsymbol{\psi}}^{k+1}|^2_2\leq& C |\overline{\boldsymbol{\psi}}^{k+1}|_2\big(|\nabla U^k|_\infty |\overline{\boldsymbol{\psi}}^{k+1}|_2+|\nabla^2 \overline{U}^k|_2+|\Upsilon^k_1 |_2+|\Upsilon^k_2|_2\big)\\
\leq &C|\overline{\boldsymbol{\psi}}^{k+1}|_2\big(\|\nabla U^k\|_2|\overline{ \boldsymbol{\psi}}^{k+1}|_2+|\nabla^2 \overline U^k|_2+|\nabla \boldsymbol{\psi}^k|_2|\overline U^k|_\infty+|\boldsymbol{\psi}^k|_q|\nabla \overline U^k|_{q^{*}}\big)\\
\leq&  C\big(\|\nabla U^k\|_2+\epsilon^{-1}(1+|\nabla\boldsymbol{\psi}^k|_2^2+|\boldsymbol{\psi}^k|_q^2)\big)|\overline{\boldsymbol{\psi}}^{k+1}|^2_2+\epsilon \|\overline{U}^k\|^2_2,
\end{split}
\end{equation*}
for $t\in[0,T_{**}]$, where $\epsilon \in (0,1)$ is a constant to be determined later.

Second, for $\overline\phi^{k+1}$, multiplying $(\ref{eq:1.2w})_1$ by $2\overline{\phi}^{k+1}$ and integrating over $\Omega$, one has 
\begin{equation}\label{fly1}
\begin{split}
\frac{\text{d}}{\text{d}t}|\overline{\phi}^{k+1}|^2_2\leq& C\big(|\nabla U^k|_\infty|\overline{\phi}^{k+1}|_2+ |\overline{U}^k|_6|\nabla \phi^k|_3+|\nabla\overline{U}^k|_2| \phi^k|_\infty\big)|\overline{\phi}^{k+1}|_2.
\end{split}
\end{equation}

Applying $\partial_{x}^{\zeta}$ ($|\zeta|=1$) to $(\ref{eq:1.2w})_{1}$, multiplying by $2 \partial_{x}^\zeta\overline{\phi}^{k+1}$ and  integrating over $\Omega$, then 
\begin{equation}\label{fly2}
\begin{split}
\frac{\text{d}}{\text{d}t}|\partial_{x}^\zeta\overline{\phi}^{k+1}|^2_{2}
\leq& C\Big(|\nabla U^k|_\infty |\nabla \overline{\phi}^{k+1}|_{2}+|\nabla \phi^k|_{6}| \nabla \overline{U}^k|_3+|  \overline{U}^k|_\infty|\nabla^2 \phi^{k}|_2\\
&+|\nabla^2 U^k|_3 | \overline{\phi}^{k+1}|_{6}+|\nabla \overline{U}^k|_3 |\nabla \phi^{k}|_6+| \phi^k|_{\infty}  |\nabla \text{div} \overline{U}^k|_2   \Big)|\nabla \overline{\phi}^{k+1}|_{2},
\end{split}
\end{equation}
which, along with  \eqref{fly1} and Young's inequality, yields  that for $ t\in[0,T_{**}]$,
\begin{equation}\label{fly3}
\frac{\text{d}}{\text{d}t}\| \overline{\phi}^{k+1}\|^2_1\leq  C\big(\| U^k\|_{3}+\epsilon^{-1}(1+\|\phi^k\|^2_2)\big)\| \overline{\phi}^{k+1}\|^2_1+\epsilon \|\overline{U}^k\|^2_2.
\end{equation}


For $\overline{U}^{k+1}$, multiplying $(\ref{eq:1.2w})_2$ by $2\overline{U}^{k+1}$ and integrating over $\Omega$, one gets
\begin{equation}\label{ghbbb}
\begin{split}
&\frac{\text{d}}{\text{d}t}|\overline{U}^{k+1}|^2_2+2\alpha|\nabla\overline{U}^{k+1} |^2_2+2\alpha|\text{div}\overline{U}^{k+1} |^2_2\\
\leq& C\Big(\big(|U^k|_\infty |\nabla\overline{U}^{k}|_2+|\overline{U}^{k}|_6|\nabla U^{k-1}|_3+|\boldsymbol{\psi}^{k+1}|_q |\nabla\overline{U}^{k}|_{q^{*}}\\
&+|\overline{\boldsymbol{\psi}}^{k+1}|_2 |\nabla U^{k-1}|_\infty\big) |\overline{U}^{k+1}|_2+|\nabla \overline{U}^{k+1}|_2 |\overline{\phi}^{k+1}|_2\Big).
\end{split}
\end{equation}


Applying $\partial_{x}^{\zeta}$ to $(\ref{eq:1.2w})_2$ ($|\zeta|=1$), multiplying by $2\partial_{x}^\zeta\overline{U}^{k+1}$ and  integrating over $\Omega$,  then
\begin{equation*}
\begin{split}
&\frac{\text{d}}{\text{d}t}|\partial_{x}^\zeta\overline{U}^{k+1}|^2_{2}+2\alpha|\nabla \partial_{x}^\zeta\overline{U}^{k+1} |^2_2+2\alpha| \partial_{x}^\zeta \text{div} \overline{U}^{k+1} |^2_2\\
=& -2\int_{\Omega} \partial_{x}^\zeta\Big( U^k\cdot\nabla \overline{U}^{k}+\overline{U}^{k} \cdot \nabla U^{k-1}+ \nabla\overline\phi^{k+1}-\boldsymbol{\psi}^{k+1}\cdot Q(\overline{U}^k)\\
&-\overline{\boldsymbol{\psi}}^{k+1}\cdot Q(U^{k-1})\Big)\cdot \partial_{x}^\zeta\overline{U}^{k+1} \text{d}x\\
\le &C\Big(|\nabla \overline{U}^{k+1}|_2\big(|\nabla U^k|_\infty |\nabla\overline{U}^{k}|_2 +| U^k|_\infty |\nabla^2\overline{U}^{k}|_{2}+|\nabla\overline{U}^{k}|_2 |\nabla U^{k-1}|_\infty\\
& +|\overline{U}^{k}|_6 |\nabla^2 U^{k-1}|_3\big)+|\boldsymbol{\psi}^{k+1}|_q |\nabla^2\overline{U}^{k}|_{2}|\nabla\overline{U}^{k+1}|_{q^*}\\
&+|\nabla\boldsymbol{\psi}^{k+1}|_2 |\nabla \overline{U}^{k}|_{6}\big(|\nabla \overline{U}^{k+1}|_2+|\nabla\overline{U}^{k+1}|_2^{\frac12}|\nabla^2\overline{U}^{k+1}|_2^{\frac12}\big)\\
&+\big(|\nabla\overline\phi^{k+1}|_2+|\overline{\boldsymbol{\psi}}^{k+1}|_2 |\nabla U^{k-1}|_\infty\big) |\nabla \partial_{x}^\zeta\overline{U}^{k+1}|_2\Big),
\end{split}
\end{equation*}
which, along with  \eqref{momo}, \eqref{ghbbb} and  Young's inequality, yields that for $ t\in[0,T_{**}]$,
\begin{equation}\label{gogo13}
\begin{split}
\frac{\text{d}}{\text{d}t}\|\overline{U}^{k+1}\|^2_1+\alpha \|\nabla \overline{U}^{k+1} \|^2_{1}
\leq & C\big(\epsilon^{-1}(1+\|U^{k}\|^2_3+\|U^{k-1}\|^2_3)+\epsilon^{-\frac{q}{q-3}}\big)\|\overline{U}^{k+1}\|^2_1\\
&+C(\|U^{k-1}\|^2_3|\overline{\boldsymbol{\psi}}^{k+1}|^2_{2}+\|\overline{\phi}^{k+1}\|^2_{1})+\epsilon\|\overline{U}^{k}\|^2_2.
\end{split}
\end{equation}

Finally, let
\begin{equation*}\begin{split}
\Gamma^{k+1}(t)=&\sup_{0\leq s \leq t}\|\overline{\phi}^{k+1}(s)\|^2_{1}+\sup_{0\leq s \leq t}\|\overline{U}^{k+1}(s)\|^2_1+\sup_{0\leq s \leq t}|\overline{\boldsymbol{\psi}}^{k+1}(s)|^2_{ 2}.
\end{split}
\end{equation*}
According to the above estimates  and  the Gronwall inequality, one concludes that
\begin{equation*}
\begin{split}
&\Gamma^{k+1}(t)+\alpha\int_{0}^{t}\|\nabla\overline{U}^{k+1}(s)\|^2_1\text{d}s\\
\leq&  \Big( C\epsilon\int_{0}^{t}  \|\nabla \overline{U}^k(s)\|^2_1\text{d}s+C\epsilon t\sup_{0\leq s \leq t}|\overline{U}^{k}(s)|^2_2 \Big)\exp{(C+C_\epsilon t)}.
\end{split}
\end{equation*}

Choose $\epsilon>0$ and $T_* \in (0,\min\{1,T_{**}\})$ small enough such that
$$
C\epsilon\exp{C}\leq \min\Big\{\frac{1}{4}, \frac{\alpha}{4}\Big\}  \quad \text{and}\quad \text{exp}(C_\epsilon T_*) \leq 2.
$$

Then one gets easily
\begin{equation*}\begin{split}
\sum_{k=1}^{\infty}\Big(  \Gamma^{k+1}(T_*)+\alpha\int_{0}^{T_*} \|\nabla\overline{U}^{k+1}(s)\|^2_1\text{d}s\Big)\leq C<\infty,
\end{split}
\end{equation*}
which means that the whole sequence $(\phi^k,U^k,\boldsymbol{\psi}^k)$ converges to a limit $(\phi, U,\boldsymbol{\psi})$ in the following strong sense:
\begin{equation}\label{str}
\begin{split}
&\phi^k\rightarrow \phi\ \ \text{in}\ \  L^\infty([0,T_*];H^1(\Omega)),\quad U^k\rightarrow U\ \ \text{in}\ \  L^\infty ([0,T_*];H^1(\Omega)),\\
&\boldsymbol{\psi}^k\rightarrow \boldsymbol{\psi} \ \  \text{in}\ \  L^\infty([0,T_*];L^2(\Omega_R)),
\end{split}
\end{equation}
where $\Omega_R=\{x\in\mathbb R^3|a<|x|\le R\}$ for any $R>a$.

On the other hand, by virtue of the uniform (with respest to $k$)  estimates \eqref{lgg}, there exists a subsequence (still denoted by $(\phi^k, U^k,\boldsymbol{\psi}^k)$) converging to the limit $(\phi, U,\boldsymbol{\psi})$ in the weak or $\text{weak}^{*}$  sense. According to the lower semi-continuity of norms, the corresponding estimates in \eqref{lgg} for $(\phi,U,\boldsymbol{\psi})$ still hold. Therefore, it is easy to show that $(\phi,U,\boldsymbol{\psi})$ is a weak solution in the sense of distributions to the  problem \eqref{eqn1} and satisfy the following regularities:
\begin{equation}\label{rjkqq}\begin{split}
&\phi> 0,\quad  \phi \in L^\infty([0,T_*];L^{\frac{1}{\gamma-1}}\cap H^2),\quad  \phi_t \in L^\infty([0,T_*];H^1),\\
& \boldsymbol{\psi} \in L^\infty([0,T_*] ; L^q\cap D^1),\quad  \boldsymbol{\psi}_t \in L^\infty([0,T_*]; L^2),\\
& U\in L^\infty([0,T_*]; H^2)\cap L^2([0,T_*] ; D^3),\quad U_t \in L^\infty([0,T_*]; L^2)\cap L^2([0,T_*] ; D^1),\\
&t^{\frac12} U_t\in L^{\infty}([0,T_*]; D^1)\cap L^2([0,T_*]; D^2),\quad t^{\frac12} U_{tt} \in L^2([0,T_*]; L^2).
\end{split}
\end{equation}

\textbf{Step 2:} uniqueness.  Let $(\phi_1,U_1,\boldsymbol{\psi}_1)$ and $(\phi_2,U_2,\boldsymbol{\psi}_2)$ be two regular solutions to the  problem \eqref{eqn1} satisfying the uniform estimates in \eqref{lgg}. Set
$$
\overline{\phi}=\phi_1-\phi_2,\quad\overline{U}=U_1-U_2,\quad  \overline{\boldsymbol{\psi}}=\boldsymbol{\psi}_1-\boldsymbol{\psi}_2.$$
Then $(\overline{\phi},\overline{U},\overline{\boldsymbol{\psi}})$ satisfies the system
 \begin{equation}
\label{zhuzhu}
\begin{cases}
\ \ \displaystyle \overline{\phi}_t+U_1\cdot \nabla\overline{\phi} +\overline{U}\cdot\nabla\phi_{2}+(\gamma-1)(\overline{\phi} \text{div}U_2 +\phi_{1}\text{div}\overline{U})=0,\\[4pt]
\ \ \displaystyle \overline{U}_t+ U_1\cdot\nabla \overline{U}+ \overline{U}\cdot \nabla U_{2}+\nabla\overline\phi +L\overline{U}=\boldsymbol{\psi}_1\cdot Q(\overline{U})+\overline{\boldsymbol{\psi}}\cdot Q(U_2),\\
\ \ \displaystyle \overline{\boldsymbol{\psi}}_t+\sum_{l=1}^3 A_l(U_1) \partial_l\overline{\boldsymbol{\psi}}+B(U_{1})\overline{\boldsymbol{\psi}}+\nabla \text{div}\overline{U}=\overline{\Upsilon}_1+\overline{\Upsilon}_2,
\end{cases}
\end{equation}
where $\overline{\Upsilon}_1$  and $\overline{\Upsilon}_2$ are defined by
\begin{equation*}
\overline{\Upsilon}_1=-\sum_{l=1}^3(A_l\big(U_{1}) \partial_l\boldsymbol{\psi}_{2}-A_l(U_{2}\big) \partial_l\boldsymbol{\psi}_{2}),\quad \overline{\Upsilon}_2=-\big(B(U_{1}) \boldsymbol{\psi}_{2}-B(U_{2}) \boldsymbol{\psi}_{2}\big).
\end{equation*}
Let
$$
\Phi(t)=\|\overline{\phi}(t)\|^2_{1}+|\overline{\boldsymbol{\psi}}(t)|^2_{ 2}+\|\overline{U}(t)\|^2_1.
$$
Similarly to the derivation of \eqref{fly1}-\eqref{gogo13}, one can also show that
\begin{equation}\label{gonm}\begin{split}
\frac{\text{d}}{\text{d}t}\Phi(t)+C\|\nabla \overline{U}(t)\|^2_1\leq G(t)\Phi (t),
\end{split}
\end{equation}
where $\displaystyle  \int_{0}^{t}G(s)\text{d}s\leq C$, for $0\leq t\leq T_*$. From  the Gronwall inequality, one concludes that
$\overline{\phi}=\overline{\boldsymbol{\psi}}=\overline{U}=0$,
then the uniqueness is obtained.

\textbf{Step 3:} The time-continuity  follows easily from the same procedure as in Lemma \ref{lem1}. Finally, Theorem \ref{thh1} is proved.

\end{proof}

\begin{rk}\label{lalala}
Now we give the proof of Lemma \ref{lpsi}.
\end{rk}

\begin{proof}
Let $X(x)\in C_0^\infty(\mathbb{R}^3)$ be a truncation function satisfying 
 \begin{equation}
0\le X(x)\le 1\quad \text{and}\quad
X(x)= \left\{
 \begin{aligned}
 &1\qquad \text{if}\quad 0\le |x|\le 1,\\
 &0\qquad\text{if}\quad|x|\ge 2.
 \end{aligned}
 \right.
 \end{equation}
 Define, for any $R> 0$, $ X_R(x)=X\big(\frac{|x|-a}{R}\big),\  \overline{\boldsymbol{\psi}}^{k+1,R}=\overline{\boldsymbol{\psi}}^{k+1}X_R$. Then from 
 $\eqref{eq:1.2w}_3$, 
\begin{equation}\label{gga}
\begin{split}
&\overline{\boldsymbol{\psi}}^{k+1,R}_t+\sum\limits_{l=1}^3 A_l(U^k)\partial_l\overline{\boldsymbol{\psi}}^{k+1,R}+B(U^k)\overline{\boldsymbol{\psi}}^{k+1,R}+\nabla\text{div}\overline U^k X_R\\
=&(\Upsilon_1^k+\Upsilon_2^k)X_R+\sum\limits_{l=1}^3A_l(U^k)\overline{\boldsymbol{\psi}}^{k+1}\partial_l X_R.
\end{split}
\end{equation}

Multiplying \eqref{gga} by $2\overline{\boldsymbol{\psi}}^{k+1,R}$ and integrating over $\Omega$, one can obtain 
\begin{equation}\label{xxx}
\begin{split}
\frac{\text{d}}{\text{d}t}|\overline{\boldsymbol{\psi}}^{k+1,R}|_2\le &C\big(|\nabla U^k|_\infty|\overline{\boldsymbol{\psi}}^{k+1,R}|_2+|\nabla\text{div}\overline U^k|_2+|U^k|_{q^*}|\overline{\boldsymbol{\psi}}^{k+1}|_q
\\
&+|\nabla\boldsymbol{\psi}^k|_2(|U^k|_\infty+|U^{k-1}|_\infty)+|\boldsymbol{\psi}^k|_q(|\nabla U^k|_{q^*}+|\nabla U^{k-1}|_{q^*})\big)\\
\le & \widetilde C|\nabla U^k|_\infty|\overline{\boldsymbol{\psi}}^{k+1,R}|_2+\widetilde C,
\end{split}
\end{equation}
where $\widetilde C>0$ is a generic constant depending on $C, q$ but independent of $R$. Then applying the Gronwall inequality to \eqref{xxx}, one gets
$$|\overline{\boldsymbol{\psi}}^{k+1,R}(t)|_2\le \widetilde C T_{**}\exp\big(\widetilde CT_{**}\big)\quad \text{for}\quad (t,R)\in[0,T_{**}]\times[0,\infty),$$
which, along with Fatou's lemma (see Lemma \ref{Fatou}  (Appendix A)), yields that 
\begin{equation}\label{dah}\overline{\boldsymbol{\psi}}^{k+1}\in L^\infty([0,T_{**}];L^2).\end{equation}
Finally, the desired conclusion can be obtained from  \eqref{dah} and $\nabla\overline{\boldsymbol{\psi}}^{k+1}\in L^\infty([0,T_{**}];L^2)$. 
\end{proof}

\subsection{\rm{Proof of Theorem \ref{zth}}}\label{com}
\subsubsection{The well-posedness theory in  Theorem \ref{zth}} 
First, it follows from the  assumption \eqref{th78qq} and Theorem \ref{thh1} that there exist a time $T_*>0$ and a unique regular solution $( \phi,U,\boldsymbol{\psi})$ in $[0,T_*]\times\Omega$ to  \eqref{eqn1} satisfying \eqref{reg11qq}.

Second, from the transformation in  \eqref{bianhuan}, one has
$$\rho(t,x)=\big(\frac{\gamma-1}{A\gamma}\phi\big)^{\frac{1}{\gamma-1}}(t,x)\quad \text{and}\quad \frac{\partial\rho}{\partial \phi}(t,x)=\frac{1}{\gamma-1}\big(\frac{\gamma-1}{A\gamma}\big)^{\frac{1}{\gamma-1}}\phi^{\frac{2-\gamma}{\gamma-1}}(t,x).$$
Then multiplying $\eqref{eqn1}_1$ by $\frac{\partial \rho}{\partial \phi}(t,x)$ yields the continuity equation $\eqref{eq:1.1}_1$; and multiplying $\eqref{eqn1}_2$ by $\rho(t,x)$ gives the momentum equations $\eqref{eq:1.1}_2$.

Thus we have shown that $(\rho,U)$ satisfies the IBVP  \eqref{eq:1.1}-\eqref{10000} with  \eqref{eqs:CauchyInit}-\eqref{e1.3} in the sense of distributions and the regularities in Definition \ref{cjk}. Moreover, it follows from the continuity equation that $\rho(t,x)>0$ for $[0,T_*]\times \Omega$. In summary, the IBVP  \eqref{eq:1.1}-\eqref{10000} with  \eqref{eqs:CauchyInit}-\eqref{e1.3} has a unique regular solution $(\rho, U)$.

\subsubsection{Spherically symmetric property  in  Theorem \ref{zth}} 
\begin{lem}\label{axl}
If  $( \rho_0,U_0)$ are  spherically symmetric in the sense of \eqref{eqs:CauchyInit}, then 
  the regular solution  $( \rho,U)(t,x)$ to the IBVP  \eqref{eq:1.1}-\eqref{10000} with  \eqref{eqs:CauchyInit}-\eqref{e1.3}  is also  a spherically symmetric one taking the form \eqref{duichenxingshi}.

\end{lem}

\begin{proof}
First, for any  orthogonal real matrix $H=(h_{kl})_{3\times 3}$,  it follows from  \eqref{eqs:CauchyInit} that 
\begin{equation}\label{vl}
\rho_0(x)=\rho_0(Hx) \quad \text{and}\quad  u_0(|x|)=u_0(|Hx|).
\end{equation}

Second, denote 
$$\widetilde\rho(t,x)=\rho(t,Hx) \quad \text{and}\quad   \widetilde{U}(t,x)=H^{\top}U(t,Hx)=((H^{\top} U)^{(1)}, \cdots, (H^{\top} U)^{(d)})^{\top}.
$$
It follows from \eqref{vl} that $(\widetilde\rho(0,x), \widetilde{U}(0,x))=(\rho_0(x),U_0(x))$. 
Moreover,  it follows from $\eqref{e1.3}$ and the fact  $|Hx|=|x|$ that 
\begin{align*}
\widetilde U(t, x)|_{|x|=a}=H^{\top}U(t, Hx)|_{|x|=a}=0\quad \text{for}\quad t\ge 0,&\\
(\widetilde\rho(t,x), \widetilde{U}(t,x))\to (0,0)\quad \text{as} \quad |x|\to\infty \quad \text{for}\quad t\ge 0.
\end{align*}
Next we need to show  that $(\widetilde\rho(t,x), \widetilde{U}(t,x))$ is also a solution to the system \eqref{eq:1.1}-\eqref{10000}. The proof is divided into two steps.


\textbf{Step 1:} $(\widetilde\rho(t,x), \widetilde{U}(t,x))$ satisfies  $\eqref{eq:1.1}_1$. In order to facilitate the discussion, in the rest of the proof of Lemma \ref{axl},  we adopt the Einstein summation convention: an index that appears exactly twice in a term is implicitly summed over. For any orthogonal real matrix $H=(h_{kl})_{d\times d}$, and vectors $x=(x_1,\cdots, x_d)^{\top}$, $\xi=(\xi_1,\cdots, \xi_d)^{\top}$,  we denote 
\begin{equation}\label{sbo}
\xi=Hx\quad  \text{with} \quad \xi_k=h_{kl}x_l, \quad h_{kl}h_{lk}=1, \quad h_{lj}h_{nj}=\delta_{ln},
\end{equation}
where $\delta_{ln}$ is the Kronecker symbol satisfying $\delta_{ln}=1,$ when $l=n$; and $\delta_{ln}=0$, otherwise.
Direct calculation gives
\begin{itemize}
\item For $\widetilde\rho_t$,  $
\widetilde\rho_t(t,x)=\rho_t(t,Hx)$;
\item For $\text{div}(\widetilde\rho\widetilde{U})$,
\begin{equation}\label{ma2}
\begin{split}
&\text{div}(\widetilde\rho\widetilde{U})=\frac{\partial}{\partial x_i}\big(\widetilde\rho (H^{\top} U)^i\big)
=\frac{\partial\rho}{\partial\xi_k}\frac{\partial\xi_k}{\partial x_i}(H^{\top} U)^i+\rho\frac{\partial}{\partial \xi_k}\big(H^{\top} U\big)^i\frac{\partial\xi_k}{\partial x_i}\\
=&\frac{\partial\rho}{\partial\xi_k}h_{kl}\delta_{li}h_{ik}U^k+\rho h_{ik}\frac{\partial U^k}{\partial\xi_k} h_{kl}\delta_{li}\\
=&\frac{\partial\rho}{\partial\xi_k}h_{ki}h_{ik}U^k+\rho h_{ik}h_{ki}\frac{\partial U^k}{\partial\xi_k}=\nabla\rho\cdot U+\rho\text{div}U=\text{div}(\rho U),
\end{split}
\end{equation}
\end{itemize}
which, along with  $\eqref{eq:1.1}_1$, implies  that $(\widetilde\rho(t,x), \widetilde{U}(t,x))$ satisfies  $\eqref{eq:1.1}_1$.


\textbf{Step 2:} $(\widetilde\rho(t,x), \widetilde{U}(t,x))$ satisfies  $\eqref{eq:1.1}_2$.
Direct calculation gives
\begin{itemize}
\item 
For $(\widetilde\rho \widetilde U)_t$,
\begin{equation*}
(\widetilde\rho \widetilde U)_t=(\rho(H^{\top} U)^i)_t=(\rho h_{ik}U^k)_t=(\rho U^k)_t h_{ik}=H^{\top}(\rho U)_t;
\end{equation*}
\item 
For $\text{div}(\widetilde\rho \widetilde U\otimes \widetilde U)$,
\begin{align*}
\text{div}(\widetilde\rho \widetilde U\otimes \widetilde U)
=&\frac{\partial}{\partial x_j}\big(\rho(H^{\top}U)^i (H^{\top}U)^j\big)
=\frac{\partial}{\partial \xi_n}\big(\rho h_{ik} h_{jl} U^k  U^l\big)\frac{\partial}{\partial x_j}\big(h_{ns}x_s\big)\\
=&\frac{\partial}{\xi_n}\big(\rho U^k  U^l\big) h_{ik}\delta_{nl}=\frac{\partial}{\xi_l}\big(\rho U^k  U^l\big) h_{ik}=H^{\top}\text{div}(\rho U\otimes U);
\end{align*}
\item
 For $\nabla P(\widetilde\rho)$,
$$\nabla P(\widetilde\rho)=\frac{\partial P(\widetilde\rho)}{\partial x_i}=\frac{\partial P(\rho)}{\partial {\xi_k}}\frac{\partial}{\partial x_i}\big(h_{kl}x_l\big)=\frac{\partial P(\rho)}{\partial {\xi_k}}h_{ki}=H^{\top}\nabla P(\rho);$$
\item 
For $\text{div}(\mu(\widetilde\rho)(\nabla \widetilde U+(\nabla\widetilde U)^{\top}))$,
\begin{align*}
&\text{div}(\mu(\widetilde\rho)(\nabla \widetilde U+(\nabla\widetilde U)^{\top}))
=\frac{\partial}{\partial x_j} \Big(\mu(\widetilde\rho)\Big(\frac{\partial}{\partial \xi_l}\big(h_{jk} U^k\big) h_{li}+\frac{\partial}{\partial \xi_k}\big(h_{il} U^l\big) h_{kj}\Big)\Big)\\
=&\frac{\partial}{\partial \xi_n}\Big(\mu(\rho)\Big(\frac{\partial U^k}{\partial{\xi_l}} h_{jk}h_{li}+\frac{\partial  U^l}{\partial{\xi_k}} h_{il}h_{kj}\Big)\Big) \frac{\partial}{\partial x_j}\big(h_{ns}x_s\big)\\
=&\frac{\partial}{\partial\xi_k}\Big(\mu(\rho)\Big(\frac{\partial U^k}{\partial{\xi_l}}h_{li}+\frac{\partial U^l}{\partial{\xi_k}} h_{il}\Big)\Big)=H^{\top}\text{div}\big(\mu(\rho)(\nabla U+(\nabla U)^{\top})\big);
\end{align*}

\item For $\nabla (\lambda(\widetilde\rho)\text{div} \widetilde U)$
\begin{align*}
\nabla (\lambda(\widetilde\rho)\text{div} \widetilde U)=&\frac{\partial}{\partial x_i}\big(\lambda(\rho)\partial_j(H^{\top} U)^j\big)=\frac{\partial}{\partial x_i}\Big(\lambda(\rho)\frac{\partial}{\partial \xi_k}\big(h_{jk} U^k\big)\frac{\partial}{\partial x_j}\big(h_{ks}x_s\big)\Big)\\
=&\frac{\partial}{\partial \xi_l}\Big(\lambda(\rho)\frac{\partial}{\xi_k}\big(h_{jk} U^k\big)h_{kj}\Big) \frac{\partial}{\partial x_i}\big(h_{ls}x_s\big)\\
=&\frac{\partial}{\partial \xi_l}\Big(\lambda(\rho)\frac{\partial U^k}{\partial{\xi_k}}\Big) h_{jk} h_{kj} h_{li}=H^{\top}\nabla(\lambda(\rho)\text{div} U);
\end{align*}
\end{itemize}
which, along with  $\eqref{eq:1.1}_2$, yields that  that 
$(\widetilde\rho(t,x), \widetilde{U}(t,x))$  also satisfies  $\eqref{eq:1.1}_2$. Then $(\widetilde\rho(t,x), \widetilde{U}(t,x))$  is also a regular solution to the IBVP  \eqref{eq:1.1}-\eqref{10000} with  \eqref{eqs:CauchyInit}-\eqref{e1.3}, which along with the uniqueness obtained  in Theorem \ref{zth}, yields that 
\begin{equation}\label{cls}
\rho(t,Hx)=\rho(t,x), \quad H^{\top}U(t, Hx)=U(t,x).
\end{equation}
Then one has  $\rho(t,x)=\rho(t,|x|)$.   It remains to show $U(t,x)=u(t,|x|)\dfrac{x}{|x|}$ for some $u$. Let $x\in \mathbb{R}^3$ be any arbitrary displacement vector, and $H$ be the  matrix that performs a 180-degree rotation about the axis parallel to $x$.
Then one has
$$Hx=x \quad \text{and}\quad U(t,Hx)=U(t,x).$$ 

Next let $\widehat x=\frac{x}{|x|}$, i.e., a unit vector parallel to $x$.  Let  $(\widehat x, \widehat y, \widehat z)$ be an orthonormal basis. 
Then, for any fixed time $t$,  there exist some real constants $x_1, y_1, z_1$ such that 
\begin{equation}\label{bee}
U(t,x)=x_1\widehat x+y_1\widehat y+z_1\widehat z,
\end{equation}
which, along with the fact that $H$ is a 180-degree rotation about  $\widehat x$, yields that 
\begin{equation}\label{bbss}
HU(t,x)=x_1\widehat x-y_1\widehat y-z_1\widehat z.
\end{equation}
It thus follows from \eqref{bee}-\eqref{bbss} that 
$$x_1\widehat x+y_1\widehat y+z_1\widehat z=x_1\widehat x-y_1\widehat y -z_1\widehat z,$$
which means that  $y_1=z_1=0$, and  for any fixed time $t$,
$$U(t,x)=x_1\widehat x=x_1\frac{x}{|x|},$$
which, along with the arbitrary choice of $x$ and  $t$, implies that  $U(t,x)$ is a radial vector.

The proof of Lemma \ref{axl} is complete.
\end{proof}
Thus the proof of  Theorem \ref{zth} is complete.

\section{Global-in-time spherically symmetric estimates}
The purpose of this section is to  establish the global-in-time energy estimates  in the spherically symmetric Eulerian  coordinate. To this end, we first consider the following reformulation of the IBVP  \eqref{eq:1.1}-\eqref{10000} with  \eqref{eqs:CauchyInit}-\eqref{e1.3}.
Throughout this section, we adopt the following simplified notations, most of them are for the standard homogeneous and inhomogeneous Sobolev spaces: for $I_a=[a,\infty)$ with some $a>0$,
\begin{equation*}\begin{split}
 \displaystyle
 &  L^p=L^p(I_a),\quad H^s=H^s(I_a), \quad D^{k,l}=D^{k,l}(I_a),\\
  \displaystyle
 & D^k=D^{k,2}(I_a),\quad 
  W^{m,p}= W^{m,p}(I_a), \quad  |f|_p=\|f\|_{L^p(I_a)},\\
   \displaystyle
  & \|f\|_s=\|f\|_{H^s(I_a)},\quad 
  \|f\|_{m,p}=\|f\|_{W^{m,p}(I_a)},\quad  |f|_{D^{k,l}}=\|\nabla^k f\|_{L^l(I_a)},\\
   \displaystyle
  &|f|_{D^{k}}=\|\nabla^k f\|_{L^2(I_a)},\quad 
\|f\|_{X_1 \cap X_2}=\|f\|_{X_1}+\|f\|_{X_2},\quad \int f \text{d}r=\int_{I_a} f \text{d}r.
 \end{split}
\end{equation*}

\subsection{\rm{Reformulation in the  spherically symmetric Eulerian coordinate}}\label{yyds}
 Let $( \rho,U)(t,x)$ in $[0,T_*]\times\Omega$ be the  unique regular solution   to the IBVP \eqref{eq:1.1}-\eqref{10000} with  \eqref{eqs:CauchyInit}-\eqref{e1.3} obtained in Theorem \ref{zth}, which has the following form
\begin{equation}\label{e1.4}
\rho(t,x)=\rho(t,r), \quad U(t,x)=u(t,r)\frac{x}{r},\quad r=|x|.
\end{equation}
Denote  $m=d-1$ ($d=2,3$). Then the IBVP \eqref{eq:1.1}-\eqref{10000} with  \eqref{eqs:CauchyInit}-\eqref{e1.3} can be rewritten into 
\begin{equation}\label{e1.5}
\begin{cases}
\displaystyle 
\ \rho_t+(\rho u)_r+\frac{m\rho u}{r}=0,\\[8pt]
\displaystyle
\ (\rho u)_t+(\rho u^2)_r+P_r-2\alpha\Big(\rho\big(u_r+\frac{m}{r} u\big)\Big)_r+\frac{2\alpha m\rho_r u}{r}+\frac{m\rho u^2}{r} =0,\\[8pt]
\displaystyle
\ (\rho(0,r), u(0,r))=(\rho_0(r), u_0(r))\quad \text{for}\quad r\in  I_a,\\[8pt]
\displaystyle
\ u(t,r)|_{r=a}=0 \quad \text{for} \quad  t\geq 0,\\[8pt]
\displaystyle
\  \left(\rho(t,r),u(t,r)\right)\to \left(0,0\right)\quad  \text{as}\quad  r\to \infty\quad  \text{for}\quad   t\ge 0.
\end{cases}
\end{equation}

Then by Theorem  \ref{zth}, Lemma  \ref{poss}  and Remark \ref{rec} (Appendix B), one has 
\begin{lem}\label{rth1} Let \eqref{cd1} hold  except $\gamma >\frac{3}{2}$.
Assume the initial data $(\rho_0(r), u_0(r))$ satisfy 
\begin{equation}\label{etm}
\begin{split}
&0<r^m\rho_0\in L^1,\quad  \big(r^{\frac{m}{2}}\rho_0^{\gamma-1}, r^{\frac{m}{2}}u_0\big)\in H^2, \quad  r^{\frac{m}{q}}(\ln\rho_0)_r\in L^q,\\
&(\ln\rho_0)_r\in L^\infty,\quad  r^{\frac{m}{2}}\big(r^{-1}(\ln\rho_0)_r,(\ln\rho_0)_{rr}\big)\in L^2,
\end{split}
\end{equation}
for some $q\in(3,6]$. Then there exist a positive time $T_*>0$ and a unique smooth solution $(\rho(t,r), u(t,r))$ in $[0,T_*]\times I_a$ to the   problem \eqref{e1.5}  satisfying 

\begin{equation}\label{spd}
\begin{split}
& 0<r^m\rho \in C([0,T_*];L^1),\quad   r^{\frac{m}{2}}\rho^{\gamma-1} \in C([0,T_*];H^2),\\
&r^{\frac{m}{2}}(\rho^{\gamma-1})_t\in C([0,T_*];H^1),\quad    r^{\frac{m}{q}}(\ln \rho)_r\in C([0,T_*];L^q),\\
& r^{\frac{m}{2}}\big(r^{-1}(\ln \rho)_r, (\ln \rho)_{rr}\big)\in C([0,T_*];L^2),\quad  r^{\frac{m}{2}}(\ln \rho)_{tr}\in C([0,T_*];L^2),\\
& r^{\frac{m}{2}}u\in C([0,T_*];H^2(I_a))\cap L^2([0,T_*]; D^3), \quad (\ln\rho)_r\in  L^\infty([0,T_*]\times I_a),\\
& r^{\frac{m}{2}}u_t\in C([0,T_*];L^2)\cap L^2([0,T_*];D^1),\\
& t^{\frac12}r^{\frac{m}{2}}u_{tt}\in L^2([0,T_*];L^2),\quad 
  t^{\frac12}r^{\frac{m}{2}}u_{tr}\in L^\infty([0,T_*];L^2)\cap L^2([0,T_*];D^1).
\end{split}
\end{equation}
\end{lem}
\begin{proof}The desired well-posedness to the   problem \eqref{e1.5}  is just a corollary of  Theorem  \ref{zth}, Lemma  \ref{poss}  and Remark \ref{rec} (in Appendix B).
We only need to  show $(\ln\rho)_r\in  L^\infty([0,T_*]\times I_a)$ if $(\ln\rho_0)_r\in L^\infty$.

Now we need to introduce the so-called  effective velocity:
\begin{equation}\label{evmm}
 v=u+\varphi(\rho)_r=u+2\alpha \rho^{-1}\rho_r,
\end{equation}
where $\varphi(\rho)$ is a function of $\rho$ defined by $\varphi'(\rho)=2\mu(\rho)/\rho^2$.
Then it follows from  the initial assumption \eqref{etm} that 
\begin{equation}\label{evmmnn}
 v_0(r)=v_0=u_0+\varphi(\rho_0)_r=u_0+2\alpha \rho^{-1}_0(\rho_0)_r\in L^\infty.
\end{equation}

Moreover,  it follows from  the equations  $\eqref{e1.5}_1$-$\eqref{e1.5}_2$ and the definition of $v$ that 
\begin{equation}\label{e-1.16mm}
v_t+uv_r+\frac{A\gamma}{2\alpha}\rho^{\gamma-1}v-\frac{A\gamma}{2\alpha}\rho^{\gamma-1}u=0.
\end{equation}

Via the  standard characteristic method, one can obtain 
\begin{equation}\label{trymm}
v=\left(v_0+\int_0^t \frac{A\gamma}{2\alpha} \rho^{\gamma-1}u\  \exp \Big(\int_0^s \frac{A\gamma}{2\alpha} \rho^{\gamma-1}{\rm{d}}\tau\Big){\rm{d}}s\right)\exp\left(-\int_0^t \frac{A\gamma}{2\alpha} \rho^{\gamma-1} {\rm{d}}s\right),
\end{equation}
which, along with the regularities of $(\rho^{\gamma-1},u)$ and  \eqref{evmm}-\eqref{evmmnn}, yields that $v, (\ln\rho)_r \in  L^\infty([0,T_*]\times I_a)$.

 The  proof of Lemma \ref{l4.4} is complete.

\end{proof}

Then the main conclusion of this section can be stated as follows.
\begin{thm}\label{jordan}
Let \eqref{cd1} hold.
Assume that the initial data $(\rho_0(r), u_0(r))$ satisfy \eqref{etm},
then \eqref{e1.5}  admits a unique global classical solution $(\rho(t,r), u(t,r))$
in $(0,\infty)\times I_a$  satisfying the  regularities in \eqref{spd} with $T_*$ replaced by arbitrarily large  $0<T<\infty$. 

\end{thm}

 Let $T>0$ be some time and $(\rho(t,r), u(t,r))$ be  regular solutions to the  problem \eqref{e1.5} in $[0,T]\times I_a$ obtained in Lemma \ref{rth1}. The main aim in the rest of this section is to  establish the global-in-time  a priori estimates for these solutions.   Hereinafter, we denote  $C_0$ {\rm (resp. $C_0^a$)} a  generic positive constant depending only on $(\rho_0,u_0, A, \gamma,\alpha)$  {\rm (resp. $(C_0,a)$)}; $C$ {\rm (resp. $C^a$)} a generic positive constant depending only on $(C_0, T)$ {\rm (resp. $(C,a)$)}, which may be different  from line to line. 
\par\smallskip

\subsection{The $L^\infty$ estimate of  $\rho$}\label{like}
We consider the upper bound of the mass density $\rho$ in $[0,T]\times I_a$.
First, the standard energy estimates yield that
\begin{lem}\label{l4.1}
For any $T>0$, it holds that, for $0\leq t\leq T$,
\begin{equation*}\label{e-1.1}
\int r^m\left(\frac12\rho u^2+\frac{A}{\gamma-1}\rho^\gamma\right)(t,\cdot){\rm{d}}r+\int_0^t\int (2\alpha r^m \rho u_r^2+2\alpha m r^{m-2}\rho u^2) {\rm{d}}r{\rm{d}}s\leq C_0.
\end{equation*}
\end{lem}
\begin{proof}
First, multiplying $\eqref{e1.5}_2$ by $r^m u$ and using $\eqref{e1.5}_1$,  one has
\begin{equation}\label{eap1}
\begin{split}
&\Big(\frac{r^m}{2} \rho u^2\Big)_t+\Big(\frac{r^m}{2}\rho u^3+2\alpha m r^{m-1}\rho u^2\Big)_r+r^m P_r u\\
&-2\alpha\Big(r^m \rho u\big(u_r+\frac{m}{r} u\big)\Big)_r+2\alpha  r^{m}\rho u_r^2+2\alpha m r^{m-2}\rho u^2=0.
\end{split}
\end{equation}

Second, integrating \eqref{eap1} over $I_a$, one gets
\begin{equation*}\label{eap2}
\begin{split}
&\frac{\text{d}}{\text{d}t}\int \frac{r^m}{2}\rho u^2 \text{d}r+2\alpha\int r^m \rho u_r^2 \text{d}r+2\alpha m\int  r^{m-2}\rho u^2 \text{d}r=-\int r^m P_r u\text{d}r\\
=&-A\gamma \int r^m\rho^{\gamma-2}\rho_r \rho u \text{d}r=-\frac{A\gamma}{\gamma-1}\int r^m(\rho^{\gamma-1})_r\rho u\text{d}r\\
=&-\frac{A\gamma}{\gamma-1}\int (r^m\rho^{\gamma}u)_r \text{d}r +\frac{A\gamma}{\gamma-1}\int \rho^{\gamma-1}(r^m\rho u)_r \text{d}r\\
=&-\frac{A\gamma}{\gamma-1}\int \rho^{\gamma-1}(r^m\rho )_t \text{d}r=-\frac{A}{\gamma-1}\frac{\text{d}}{\text{d}t}\int  r^m\rho^\gamma \text{d}r,
\end{split}
\end{equation*}
where one has used the equation $\eqref{e1.5}_1$ and $\eqref{e1.5}_4$-$\eqref{e1.5}_5$.
Then the desired conclusion can be achieved by an  integration  over $[0,t]$.

The proof of Lemma \ref{l4.1} is complete.
\end{proof}

Second, we give the well-known BD entropy estimates.
\begin{lem}[\cite{bd2}]\label{l4.2}
For any $T>0$, it holds that
\begin{equation*}\label{e-1.2}
\int r^m\left(\frac12\rho\left|u+\varphi(\rho)_r\right|^2+\frac{A}{\gamma-1}\rho^\gamma\right)(t,\cdot){\rm{d}}r+ 2A\alpha\gamma\int_0^t\int r^m\rho^{\gamma-2}\rho_r^2{\rm{d}}r{\rm{d}}s\leq C_0
\end{equation*}
for $ 0\leq t\leq T$, where $\varphi'(\rho)=\frac{2\mu(\rho)}{\rho^2}$.
\end{lem}
\begin{proof}
According to $\eqref{e1.5}_1$-$\eqref{e1.5}_2$, one has
\begin{align}
   &\varphi(\rho)_{tr}+(\varphi(\rho)_r u)_r+(\rho\varphi'(\rho) u_r)_r+\Big(\frac{m}{r}\rho u\varphi'(\rho)\Big)_r=0,\label{e5.2}\\
   & \rho(u_t+uu_r)+P_r-2\alpha\Big(\rho\big(u_r+\frac{m}{r} u\big)\Big)_r+\frac{2\alpha m \rho_r u}{r}=0.\label{e5.1}
\end{align}
Multiplying \eqref{e5.2} by $\rho$, one has 
\begin{equation}\label{es}
\rho \varphi(\rho)_{tr}+\rho u\varphi(\rho)_{rr}+(\rho^2\varphi'(\rho) u_r)_r+\Big(\frac{m}{r}\rho^2 \varphi'(\rho) u\Big)_r-\frac{m}{r}\rho\varphi'(\rho)\rho_r u=0.
\end{equation}
Then adding \eqref{es}  to  \eqref{e5.1}, by the definition of $v=u+\varphi(\rho)_r$,  one can obtain 
\begin{equation}\label{e5.3}
  \rho(v_t+u v_r)+P_r+\Big(\big(\rho^2\varphi'(\rho)-2\alpha\rho\big)\big(u_r+\frac{m}{r} u\big)\Big)_r+\frac{ m\rho_r u}{r}(2\alpha-\rho\varphi'(\rho))=0,
\end{equation}
which, along with $\varphi'(\rho)=\frac{2\mu(\rho)}{\rho^2}$ in \eqref{e5.3}, yields that 
 \begin{equation}\label{ess}
  \rho(v_t+u v_r)+P_r=0.
\end{equation}
Multiplying \eqref{ess} by $r^m v$ and 
integrating  over $[0,t]\times I_a$, one gets
$$\int r^m\left(\frac12\rho v^2+\frac{A}{\gamma-1}\rho^\gamma\right)(t,\cdot){\rm{d}}r+ 2A\alpha\gamma\int_0^t\int r^m\rho^{\gamma-2}\rho_r^2{\rm{d}}r{\rm{d}}s\leq C_0,$$
where one has used the equation $\eqref{e1.5}_1$ and $\eqref{e1.5}_4$-$\eqref{e1.5}_5$.

The proof of Lemma \ref{l4.2} is complete. 
\end{proof}

Next we show the regular solution $(\rho(t,r),u(t,r))$ keeps the conservation of total mass.
\begin{lem}\label{ms}
For any $T>0$, it holds that
$$\int r^m\rho(t,\cdot) {\rm{d}}r=\int r^m\rho_0{\rm{d}}r \quad {\rm{for}}\ \ 0\leq t\leq T.$$
\end{lem}

\begin{proof}
It follows from $\eqref{e1.5}_1$ that 
\begin{equation*}
    \frac{\text{d}}{\text{d}t}\int r^m\rho(t,\cdot)\text{d}r=-\int (r^m \rho u)_r(t,\cdot)\text{d}r=0,
\end{equation*}
where one has used $\eqref{e1.5}_4$ and  $r^m \rho u(t,\cdot)\in W^{1,1}$.

The proof of Lemma \ref{ms} is complete.
\end{proof}

Now we are ready to give the uniform  upper bound of the density.

\begin{lem}\label{l4.3}
For any $T>0$, it holds that
\begin{equation*}\label{e4.3}
|\rho(t,\cdot)|_\infty\leq C^a \quad {\rm{for}}\ \ 0\leq t\leq T.
\end{equation*}
\end{lem}

\begin{proof}
First, according to Lemmas \ref{l4.1}-\ref{l4.2}, one gets
\begin{equation}\label{bdguji}
    \|(\sqrt\rho)_r\|_{L^\infty([0,T];L^2)}\le C^a.
\end{equation}
Second, according to Lemma \ref{ms}, one can obtain 
\begin{equation}\label{shouhengguji}
    \|\sqrt\rho\|_{L^\infty([0,T];L^2)}\le C^a.
\end{equation}
It thus follows from \eqref{bdguji}, \eqref{shouhengguji} and Sobolev embedding theorem (in Appendix A) that 
$$|\sqrt\rho|_\infty\le C\|\sqrt\rho\|_1\le C^a,$$
which implies that $$|\rho(t,\cdot)|_\infty\le C^a\quad \text{for}\quad 0\le t\le T.$$

The proof of Lemma \ref{l4.3} is complete.
\end{proof}

\subsection{The $L^2$ estimate of $r^{\frac{m}{2}}u$}
We consider the $L^2$ estimate of $r^{\frac{m}{2}}u$. For this purpose,  one first needs to show  the following several  auxiliary lemmas. The first one is on the $L^{\tilde q}$ estimate of  $(r^m\rho)^{\frac{1}{\tilde q}}u$ for any  $2\le \tilde q< \infty$.

\begin{lem} \label{lma}
Assume  $2\le \tilde q< \infty$. Then for any  $T>0$, it holds that 
\begin{equation*}
|(r^m\rho)^{\frac{1}{\tilde q}} u(t, \cdot)|_{\tilde q} \leq C^a \quad {\rm{for}}\ \ 0\leq t\leq T,
\end{equation*}
where the constant $C^a$ depends on $\tilde q$.
\end{lem}
\begin{proof}
First, multiplying $\eqref{e1.5}_2$ by $r^m |u|^{p}u$ $(p\ge2)$ and integrating the resulting equation over $I_a$, one arrives at 

\begin{equation}\label{add}
\begin{split}
&\frac{1}{p+2} \frac{\text{d}}{\text{d}t}\big|(r^m\rho)^{\frac{1}{p+2}} u\big|_{p+2}^{p+2}  +2\alpha (p+1) \big|(r^m\rho)^{\frac12} u^{\frac{p}{2}} u_r\big|_2^2\\
&+2\alpha m|(r^{m-2}\rho)^{\frac{1}{p+2}}u|_{p+2}^{p+2}\\
=& A(p+1)\int r^m \rho^{\gamma} |u|^p u_r{\rm{d}}r+Am\int r^{m-1} \rho^{\gamma} |u|^p u{\rm{d}}r\triangleq\sum_{i=1}^2 \text{J}_i,
\end{split}
\end{equation}
where one has used the equation $\eqref{e1.5}_1$ and $\eqref{e1.5}_4$-$\eqref{e1.5}_5$.

Using H\"older's inequality and Young's inequality, one can obtain that 
\begin{equation}\label{addd}
\begin{split}
\text{J}_1\le &C|(r^m\rho)^{\frac12}|u|^{\frac{p}{2}} u_r|_2|r^{\frac{m}{2}}\rho^{\gamma-\frac12}|u|^{\frac{p}{2}}|_2\\
\le&\frac{\alpha(p+1)}{32}\big|(r^m \rho)^{\frac12} u^{\frac{p}{2}} u_r\big|_2^2+C\int r^m\rho^{2\gamma-1}|u|^p{\rm{d}}r.
\end{split}
\end{equation}
Then we need to estimate the last term on the right-hand side of \eqref{addd}. According to Lemmas \ref{l4.1} and \ref{ms}-\ref{l4.3}, one has 

\begin{equation}\label{adddd}
\begin{split}
&\int r^m \rho^{2\gamma-1}|u|^2{\rm{d}}r\le C|(r^m\rho)^{\frac14} u|_4|(r^m\rho)^{\frac12} u|_2|(r^m \rho)^{\frac14}|_4|\rho|_\infty^{2(\gamma-1)}\\
\le &C^a|(r^m \rho)^{\frac14} u|_4|(r^m \rho)^{\frac12} u|_2|r^m\rho|_1^{\frac14}\\
\le & C^a\big(1+|(r^m\rho)^{\frac14} u|_4^4\big)\quad \text{for}\quad p=2,\\[4pt]
&\int r^m\rho^{2\gamma-1}|u|^p{\rm{d}}r=\int r^m\rho^{2(\gamma-1)}\rho^{\frac{p-2}{p}}|u|^{\frac{(p-2)(p+2)}{p}}\rho^{\frac{2}{p}}|u|^{\frac{4}{p}}{\rm{d}}r\\
\le& C|\rho|_\infty^{2(\gamma-1)}\big|(r^m\rho)^{\frac{p-2}{p}}u^{\frac{(p-2)(p+2)}{p}}\big|_{\frac{p}{p-2}}\big|(r^m\rho)^{\frac{2}{p}}u^{\frac{4}{p}}\big|_{\frac{p}{2}}\\
\le&C^a\big|(r^m\rho)^{\frac{1}{p+2}}u\big|_{p+2}^{\frac{(p-2)(p+2)}{p}}|r^m\rho u^2|_1^{\frac{2}{p}}\\
\le&C^a\big(1+|(r^m\rho)^{\frac{1}{p+2}}u|_{p+2}^{p+2}\big)\quad \text{for}\quad p>2.
\end{split}
\end{equation}
Consequently, according to \eqref{addd}-\eqref{adddd}, one can obtain that for any $p\ge 2$,
$$\text{J}_1\le \frac{\alpha(p+1)}{32}\big|(r^m\rho)^{\frac12} u^{\frac{p}{2}} u_r\big|_2^2+C^a\big(1+|(r^m\rho)^{\frac{1}{p+2}}u|_{p+2}^{p+2}\big).$$

For the term $\text{J}_{2}$, one can similarly get
\begin{align*}
\text{J}_2=&Am\int r^{m-1} \rho^{\gamma} |u|^p u{\rm{d}}r\\
\le& C|(r^m\rho)^{\frac{p+1}{p+2}}|u|^{p+1}|_{\frac{p+2}{p+1}}|(r^m\rho)^{\frac{1}{p+2}}|_{p+2}|\rho|_\infty^{\gamma-1}|r^{-1}|_\infty\\
\le&C^a|(r^m\rho)^{\frac{1}{p+2}}u|_{p+2}^{p+1}|r^m\rho|_1^{\frac{1}{p+2}}\le C^a\big(1+|(r^m\rho)^{\frac{1}{p+2}}u|_{p+2}^{p+2}\big).
\end{align*}

Substituting the estimates for $\text{J}_i$ $(i=1,2)$ into \eqref{add}, one has 
\begin{equation}\label{ban}
\begin{split}
&\frac{\text{d}}{\text{d}t}\big|(r^m\rho)^{\frac{1}{p+2}} u\big|_{p+2}^{p+2} + \big|(r^m\rho)^{\frac12} u^{\frac{p}{2}} u_r\big|_2^2+|(r^{m-2}\rho)^{\frac{1}{p+2}}u|_{p+2}^{p+2}\\
\le& C^a\big(1+|(r^m\rho)^{\frac{1}{p+2}}u|_{p+2}^{p+2}\big).
\end{split}
\end{equation}

Second, integrating \eqref{ban} over $[0,t]$, one can obtain that   
\begin{equation*}
\begin{split}
&\big|(r^m\rho)^{\frac{1}{p+2}} u\big|_{p+2}^{p+2}  +\int_0^t \big(\big|(r^m\rho)^{\frac12} u^{\frac{p}{2}} u_r\big|_2^2+|(r^{m-2}\rho)^{\frac{1}{p+2}}u|_{p+2}^{p+2}\big){\rm{d}}s\\
\leq &\big|(r^m\rho_0)^{\frac{1}{p+2}} u_0\big|_{p+2}^{p+2}+C^a\Big(t+\int_0^t \big|(r^m\rho)^{\frac{1}{p+2}} u\big|_{p+2}^{p+2}\text{d}s\Big),
\end{split}
\end{equation*}
which, along with the Gronwall inequality and \eqref{etm}, yields that
\begin{equation}\label{poloo}
\begin{split}
&\big|(r^m\rho)^{\frac{1}{p+2}} u(t,\cdot)\big|_{p+2}^{p+2} +\int_0^t\big( \big|(r^m\rho)^{\frac12}u^{\frac{p}{2}} u_r\big|_2^2+|(r^{m-2}\rho)^{\frac{1}{p+2}}u|_{p+2}^{p+2}\big)(s,\cdot) {\rm{d}}s \\
\le &C^a \big(1+\big|(r^m\rho_0)^{\frac{1}{p+2}} u_0\big|_{p+2}^{p+2}\big)\\
\le & C^a\big(1+|r^m\rho_0|_1 |u_0|_{\infty}^{p+2}\big)\le C^a\quad {\rm{for}}\ \ 0\leq t\leq T.
\end{split}
\end{equation}

Finally, the desired conclusion can be achieved  from \eqref{poloo} and Lemma \ref{l4.1}. 

The proof of Lemma \ref{lma} is complete.
\end{proof}

The following lemma gives the estimate of $\int_0^t |\rho^\iota u(s,\cdot)|_\infty {\rm{d}}s$ $(\frac12<\iota<1)$, which plays a crucial role in obtaining the $L^\infty$ estimate of the  effective velocity $v$ (see \eqref{ev}).

\begin{lem}\label{cru}
For any  $T>0$, it holds that 
\begin{equation*}
\int_0^t |\rho^{\iota }u(s,\cdot)|_{\infty} {\rm{d}}s\leq C^a \quad {\rm{for}}\ \ 0\leq t\leq T.
\end{equation*}
\end{lem}

\begin{proof}
First, it follows from   direct calculations that 
\begin{equation*}
(\rho^\iota u)_r=\iota \rho^{-\frac12}\rho_r\rho^{\iota-\frac12} u+ \rho^{\iota-\frac12}\rho^{\frac12} u_r.
\end{equation*}
According to Lemmas \ref{l4.1}-\ref{ms}, \ref{lma} and $\frac12<\iota<1$, one gets for $ 0\leq t\leq T$,
\begin{equation*}
|\rho^{-\frac12}\rho_r(s,\cdot)|_2+|\rho^{\iota-\frac12}u(s,\cdot)|_{\frac{2}{2\iota-1}}+|\rho^{\iota-\frac12}(s,\cdot)|_{\frac{2}{2\iota-1}}+\int_0^t |\rho^{\frac{1}{2}}u_r(s,\cdot)|_2^2 {\rm{d}}s\leq C^a,
\end{equation*}
which yields that 
\begin{equation}\label{sgg}
\int_0^t |(\rho^{\iota}u)_r(s,\cdot)|_{\frac{1}{\iota}}^2\text{d}s\le C^a \quad {\rm{for}}\ \ 0\leq t\leq T.
\end{equation}

Second, notice that $\frac{1}{\iota}\in(1,2)$ and  $\rho^\iota u$ decays to zero at infinity, then by using Newton-Leibniz formula and H\"older's inequality, one obtains 
\begin{equation}\label{fake}|\rho^{\iota}u|_\infty\le C|\rho^{\iota}u|_2^{1-\Xi}|(\rho^{\iota}u)_r|^\Xi_{\frac{1}{\iota}}\quad \text{with}\quad \Xi=\frac{1}{3-2\iota}\in(0,1).\end{equation}
Thus it follows from  \eqref{fake},  Lemmas \ref{l4.1}, \ref{l4.3} and \eqref{sgg} that

$$\int_0^t |\rho^{\iota}u(s,\cdot)|_{\infty} {\rm{d}}s\leq C \int_0^t\Big(1+|(\rho^{\iota}u)_r|^2_{\frac{1}{\iota}}\Big)(s,\cdot)\text{d}s\leq C^a \quad {\rm{for}}\ \ 0\leq t\leq T.$$

The proof of Lemma \ref{cru} is complete.
\end{proof}

Next we show the  $L^\infty$ estimate of the effective velocity:
\begin{equation}\label{ev}
 v=u+\varphi(\rho)_r=u+2\alpha \rho^{-1}\rho_r.
\end{equation}

\begin{lem}\label{l4.4}
Assume $\gamma >\frac32$ additionally. Then for any  $T>0$, it holds that 
\begin{equation*}
|v(t,\cdot)|_{\infty}\leq C^a\quad {\rm{for}}\ \ 0\leq t\leq T.
\end{equation*}
\end{lem}

\begin{proof}
According to the proof of Lemma \ref{rth1}, one has that $v$ satisfies the damped transport equation \eqref{e-1.16mm}.
Then via the  standard characteristic method, one can obtain 
\begin{equation}\label{try}
v=\left(v_0+\int_0^t \frac{A\gamma}{2\alpha} \rho^{\gamma-1}u\  \exp \Big(\int_0^s \frac{A\gamma}{2\alpha} \rho^{\gamma-1}{\rm{d}}\tau\Big){\rm{d}}s\right)\exp\left(-\int_0^t \frac{A\gamma}{2\alpha} \rho^{\gamma-1} {\rm{d}}s\right).
\end{equation}
According to \eqref{try} and Lemma \ref{l4.3}, one deduces that
\begin{equation}\label{e-1.17}
|v|_\infty\le  C^a\Big(|v_0|_\infty+\int_0^t|\rho^{\gamma-1}u|_\infty{\rm{d}}s\Big)\\
\le C^a\Big(|v_0|_\infty+\int_0^t |\rho|^{\gamma-1-\iota}_{\infty} |\rho^{\iota}u|_\infty{\rm{d}}s\Big),
\end{equation}
which, along with   Lemmas \ref{l4.3} and  \ref{cru}, yields that for any $\gamma\ge 1+\iota>\frac32$,
\begin{equation*}
|v(t, \cdot)|_\infty\le C^a \quad {\rm{for}}\ \ 0\leq t\leq T.
\end{equation*}

 The  proof of Lemma \ref{l4.4} is complete.
\end{proof}

For simplicity of the following calculations,   one  introduces the following  variables \begin{equation}\label{tr}
 \phi=\frac{A\gamma}{\gamma-1}\rho^{\gamma-1}, \quad \psi=\frac{1}{\gamma-1}(\ln \phi)_r=(\ln \rho)_r,
 \end{equation}
then  the  problem  \eqref{e1.5} can be rewritten into
\begin{equation}\label{e2.2}
\left\{
\begin{aligned}
&\ \phi_t+u\phi_r+(\gamma-1)\phi u_r+m(\gamma-1)r^{-1}\phi u=0,\\[6pt]
&\ u_t+u u_r+\phi_r-2\alpha u_{rr}=2\alpha\psi u_r+2\alpha m( r^{-1} u)_r,\\[6pt]
& \psi_t+(\psi u)_r+u_{rr}+m(r^{-1}u)_r=0,\\[6pt]
&\ (\phi(0,r),u(0,r),\psi(0,r))=(\phi_0(r),u_0(r),\psi_0(r))\\
=&\Big(\frac{A\gamma}{\gamma-1}\rho_0^{\gamma-1}(r),u_0(r),(\ln \rho_0(r))_r\Big)\quad \text{for} \quad r\in I_a,\\[6pt]
&\ u(t,r)|_{r=a}=0  \quad \text{for} \quad  t\geq 0,\\[6pt]
&\  (\phi(t,r),u(t,r),\psi(t,r))\to (0,0,0) \quad \text{as}\quad  r \to \infty  \quad \text{for} \quad  t\geq 0.
\end{aligned}
\right.
\end{equation}

The following lemma shows the  boundedness of $\int_0^t|u(s,\cdot)|^2_\infty{\rm{d}}s$, which will be used to obtain the $L^2$ estimate of $r^{\frac{m}{2}}u$.
\begin{lem}\label{ele}
Assume $\gamma >\frac32$ additionally. Then for any  $T>0$, it holds that 
$$|u(t,\cdot)|_2^2+\int_0^t(|u_r|_2^2+|r^{-1}u|_2^2+|u|^2_\infty)(s,\cdot){\rm{d}}s\le C^a \quad {\rm{for}}\ \  0\le t\le T.$$
\end{lem}

\begin{proof}
First, multiplying $\eqref{e2.2}_2$ by $u$ and integrating over $I_a$, one has 
\begin{equation}\label{ell}
\frac12\frac{\text{d}}{\text{d}t}|u|_2^2+2\alpha|u_r|_2^2+\alpha m|r^{-1}u|_2^2
=-\int \big(uu_r +\phi_r-2\alpha\psi u_r\big)u{\rm{d}}r\triangleq \sum_{i=3}^5 \text{J}_i.
\end{equation}
According to H\"older's inequality,  Lemmas \ref{ms}-\ref{l4.3}  and \ref{l4.4}, one  can obtain that 
\begin{align*}
\text{J}_3=&-\int u^2 u_r\text{d}r=0,\\
\text{J}_4=&-A\gamma\int \rho^{\gamma-2}\rho_r u{\rm{d}}r=-\frac{A\gamma}{2\alpha}\int \rho^{\gamma-1}(v-u)u{\rm{d}}r\\
\le &C(|(r^m\rho)^{\frac12}|_2|u|_2|v|_\infty|\rho|_\infty^{\gamma-\frac32}|r^{-\frac{m}{2}}|_\infty+|\rho|_\infty^{\gamma-1}|u|_2^2)\\
\le & C^a (1+|u|_2^2),\\
\text{J}_5=&2\alpha\int \rho^{-1}\rho_r u_r u{\rm{d}}r=\int (v-u)u u_r{\rm{d}}r\\
\le & C|v|_\infty|u|_2|u_r|_2\le \frac{\alpha}{4}|u_r|_2^2+C^a|u|_2^2.
\end{align*}

Substituting the estimates for $\text{J}_i$ $(i=3,4,5)$ into \eqref{ell} and integrating the resulting equation over $[0,t]$, one arrives that 
$$|u|_2^2+\int_0^t (|u_r|_2^2+|r^{-1}u|_2^2){\rm{d}}s\le C^a\Big(t+\int_0^t |u|_2^2{\rm{d}}s\Big),$$
which, along with the Gronwall inequality, yields that 
\begin{equation}\label{pol}
|u(t,\cdot)|_2^2+\int_0^t (|u_r|_2^2+|r^{-1}u|_2^2)(s,\cdot)\text{d}s\le C^a  \quad {\rm{for}}\ \  0\le t\le T.
\end{equation}
It thus follows from \eqref{pol} that 
$$\int_0^t |u(s,\cdot)|^2_\infty{\rm{d}}s\le C\int_0^t( |u|_2^2+|u_r|_2^2)(s,\cdot){\rm{d}}s\le C^a.$$

The proof of Lemma \ref{ele} is complete.
\end{proof}

 Now we give the  $L^2$ estimate of $r^{\frac{m}{2}}u$.

\begin{lem}\label{l4.5}
Assume $\gamma >\frac32$ additionally. Then for any $T>0$,  it holds that 
\begin{equation*}
|r^{\frac{m}{2}}u(t,\cdot)|_2^2+\int_0^t\big( |r^{\frac{m}{2}} u_r|_2^2+|r^{\frac{m-2}{2}}u|_2^2\big)(s, \cdot){\rm{d}}s\leq C^a\quad {\rm{for}}\ \  0\le t\le T.
\end{equation*}
\end{lem}

\begin{proof}
Multiplying $\eqref{e2.2}_2$ by $r^m u$ and integrating over $I_a$, one has 
\begin{equation}\label{e4.15}
\begin{split}
&\frac12\frac{\text{d}}{\text{d}t} |r^{\frac{m}{2}}u|_2^2 +2\alpha | r^{\frac{m}{2}}u_r|_2^2+2\alpha m |r^{\frac{m-2}{2}}u|_2^2\\
=&-\int r^m\big(uu_r+\phi_r-2\alpha\psi u_r\big) u{\rm{d}}r\triangleq  \sum_{i=6}^{8} \text{J}_i.
\end{split}
\end{equation}
Using  H\"older's  inequality, Young's inequality, Lemmas \ref{ms}-\ref{l4.3}  and \ref{l4.4}, we estimate $\text{J}_i$ $(i=6,7,8)$ as follows:
\begin{equation*}
\begin{split}
\text{J}_{6}=&-\int r^m u^2 u_r{\rm{d}}r
\le  C\int r^{m-1}|u|^3 {\rm{d}}r\\
\le & C^a|u|_\infty|r^{\frac{m}{2}}u|_2^2,\\
\text{J}_{7}=&-A\gamma \int r^m \rho^{\gamma-2}\rho_r u{\rm{d}}r=-\frac{A\gamma}{2\alpha}\int r^m \rho^{\gamma-1} (v-u) u {\rm{d}}r\\
\leq& C\big(| (r^m\rho)^{\frac12} |_{2}|r^{\frac{m}{2}}u|_{2}|v|_\infty|\rho|_{\infty}^{\gamma-\frac32}+|\rho|_{\infty}^{\gamma-1}|r^{\frac{m}{2}}u|_{2}^2\big)\\
\leq & C^a(1+|r^{\frac{m}{2}}u|_{2}^2),\\
\text{J}_{8}=&2\alpha\int r^m \rho^{-1}\rho_r u_r u{\rm{d}}r=\int r^m\left(v-u\right)u_r u{\rm{d}}r\\
\leq &C^a\big(|v|_{\infty}| r^{\frac{m}{2}}u_r|_{2}|r^{\frac{m}{2}}u|_{2}+|u|_\infty|r^{\frac{m}{2}}u|_2^2\big)\\
\leq& \frac14|r^{\frac{m}{2}}u_r|_2^2+C^a(1+|u|_\infty)|r^{\frac{m}{2}}u|_{2}^2.
\end{split}
\end{equation*}
Substituting the estimates for $\text{J}_i$ $(i=6,7,8)$ into \eqref{e4.15}, using the Gronwall  inequality and Lemma \ref{ele}, one has
\begin{equation*}\label{e4.16}
|r^{\frac{m}{2}}u(t,\cdot)|_2^2+\int_0^t\big(| r^{\frac{m}{2}}u_r|_2^2+|r^{\frac{m-2}{2}}u|_2^2\big)(s,\cdot){\rm{d}}s\leq C^a\quad {\rm{for}}\ \ 0\leq t\leq T.
\end{equation*}

The proof of Lemma \ref{l4.5} is complete.
\end{proof}

\subsection{The first order  estimates of $r^{\frac{m}{2}}u$}
Now we  consider the  estimates of $|r^{\frac{m}{2}}u_r|_2$.
\begin{lem}\label{l4.6}

Assume $\gamma >\frac32$ additionally. Then for any $T>0$,  it holds that 
\begin{equation*}\label{e-1.21}
\big(| r^{\frac{m}{2}}u_r|_2^2+|\psi|_\infty\big)(t,\cdot)+\int_0^t\big( |r^{\frac{m}{2}}u_t|_2^2+|r^{\frac{m}{2}}u_{rr}|_2^2+|r^{\frac{m}{2}}u_r|_{\infty}\big)(s,\cdot){\rm{d}}s\leq C^a \quad {\rm{for}}\ \ 0\leq t\leq T.
 \end{equation*}
\end{lem}

\begin{proof}
First, multiplying $\eqref{e2.2}_2$ by $r^m u_t$ and integrating over $I_a$, one arrives at 
\begin{equation}\label{e-1.20}
    \begin{split}
&\alpha\frac{\text{d}}{\text{d}t}| r^{\frac{m}{2}}u_r|_2^2+| r^{\frac{m}{2}}u_t|_2^2\\
=&-\int r^m\big(u u_r +\phi_r-2\alpha \psi u_r+2\alpha m r^{-2}u\big) u_t{\rm{d}}r\triangleq \sum_{i=9}^{12} \text{J}_i.
     \end{split}
\end{equation}

According to  H\"older's inequality, Lemmas \ref{ms}-\ref{l4.3}, \ref{l4.4} and  \ref{l4.5}, one can obtain that 
\begin{equation*}\label{strange}
    \begin{split}
 \text{J}_{9}=&-\int r^m u u_r u_t{\rm{d}}r\leq C|u|_{\infty}|r^{\frac{m}{2}}u_r|_{2}|r^{\frac{m}{2}}u_t|_{2}\\
        \le & C|u|_\infty^2|r^{\frac{m}{2}}u_r|_{2}^2+\frac18 |r^{\frac{m}{2}}u_t|_{2}^2,\\
 \text{J}_{10}=&-A\gamma\int r^m\rho^{\gamma-2}\rho_r u_t{\rm{d}}r
      = -\frac{A\gamma}{2\alpha}\int r^m\rho^{\gamma-1} (v-u) u_t{\rm{d}}r\\
      \le & C \big(|(r^m\rho)^{\frac12}|_{2} |v|_\infty|\rho|_\infty^{\gamma-\frac32} + |\rho|_{\infty}^{\gamma-1} |r^{\frac{m}{2}}u|_2\big) |r^{\frac{m}{2}}u_t|_2\\
     \le & C^a+\frac18 |r^{\frac{m}{2}}u_t|^2_2,\\
\text{J}_{11}=&2\alpha \int r^m\rho^{-1}\rho_r u_r u_t{\rm{d}}r=\int r^m(v-u)u_r u_t{\rm{d}}r\\
           \le &C\left(|v|_{\infty}+|u|_{\infty}\right)|r^{\frac{m}{2}}u_r|_{2}|r^{\frac{m}{2}}u_t|_{2}\\
          \le &C^a(1+|u|_{\infty}^2) |r^{\frac{m}{2}}u_r|_2^2+\frac18|r^{\frac{m}{2}}u_t|_{2}^2,\\
\text{J}_{12}=&-2\alpha m \int r^{m-2} u u_t{\rm{d}}r\\
\le & C|r^{\frac{m}{2}}u|_2|r^{-2}|_\infty|r^{\frac{m}{2}}u_t|_2\le C^a+\frac18 |r^{\frac{m}{2}}u_t|^2_2.       
    \end{split}
\end{equation*}

Substituting the estimates for $\text{J}_i$ $(i=9,\cdots,12)$ into \eqref{e-1.20}, one gets
\begin{equation}\label{e-1.23}
\frac{\text{d}}{\text{d}t}|r^{\frac{m}{2}}u_r|_2^2+|r^{\frac{m}{2}}u_t|_2^2\leq C^a(1+|u|_\infty^2)|r^{\frac{m}{2}}u_r|^2_{2}+C^a,
\end{equation}
which, along with  the Gronwall inequality and Lemma \ref{ele}, yields that 
 \begin{equation}\label{eg}
 |r^{\frac{m}{2}}u_r(t,\cdot)|_2^2+\int_0^t  |r^{\frac{m}{2}}u_t(s,\cdot)|_2^2{\rm{d}}s\leq C^a \quad {\rm{for}}\ \ 0\le t\le T.
\end{equation}
Consequently, it follows from \eqref{eg} and Lemmas \ref{l4.4}-\ref{ele} that
\begin{equation}\label{shuoming}
|\psi(t,\cdot)|_{\infty}\leq C(|v(t,\cdot)|_{\infty}+|u(t,\cdot)|_{\infty})\leq C^a(1+\|u(t,\cdot)\|_1)\leq C^a \quad {\rm{for}}\ \ 0\le t\le T.
\end{equation}

Second, according to $\eqref{e2.2}_2$, Lemmas \ref{ms}-\ref{l4.3}, \ref{l4.4} and \eqref{eg}-\eqref{shuoming}, one has 
\begin{equation}\label{sec}
\begin{split}
&|r^{\frac{m}{2}}u_{rr}|_2\\
\le& C\big(|r^{\frac{m}{2}}u_t|_2+|r^{\frac{m}{2}} uu_r|_2+|r^{\frac{m}{2}} \phi_r|_2+|r^{\frac{m}{2}} \psi u_r|_2+|r^{\frac{m-2}{2}}u_r|_2+|r^{\frac{m-4}{2}}u|_2\big)\\
\le &C^a\big(|r^{\frac{m}{2}}u_t|_2+|r^{\frac{m}{2}}u_r|_2|u|_\infty+|(r^m\rho)^{\frac12}|_2|v|_\infty|\rho|_\infty^{\gamma-\frac32}+|r^{\frac{m}{2}}u|_2|\rho|_\infty^{\gamma-1}\\
&+|r^{\frac{m}{2}}u_r|_2|\psi|_\infty+\|u\|_1\big)\le C^a(1+|r^{\frac{m}{2}}u_t|_2),
\end{split}
\end{equation}
which, along with \eqref{eg}, yields that
\begin{align*}
\int_0^t |r^{\frac{m}{2}}u_{rr}(s,\cdot)|_2^2{\rm{d}}s\le & C^a\int_0^t\big(1+|r^{\frac{m}{2}}u_t(s,\cdot)|_2^2\big){\rm{d}}s\le C^a,\\
\int_0^t|r^{\frac{m}{2}}u_r(s,\cdot)|_\infty{\rm{d}}s\le & C\int_0^t\|r^{\frac{m}{2}}u_r(s,\cdot)\|_1{\rm{d}}s\\
\le &  C^a\int_0^t (1+|r^{\frac{m}{2}}u_{rr}(s,\cdot)|_2^2){\rm{d}}s\le C^a.
\end{align*}

The proof of Lemma \ref{l4.6} is complete. 
\end{proof}

%
%
%
\subsection{The lower order estimates of $\psi$}
Now we show  the estimates on $|r^{\frac{m-2}{2}}\psi|_2$ and  $|r^{\frac{m}{q}}\psi|_q$ for some $q\in(3,6]$.

 
 \begin{lem}\label{l4.7}
 Assume $\gamma >\frac32$ additionally. Then for any $T>0$, it holds that
 \begin{equation*}\label{e4.34}
|r^{\frac{m-2}{2}}\psi(t,\cdot)|_{2}+|r^{\frac{m}{q}}\psi(t,\cdot)|_q\leq C^a \quad {\rm{for}}\ \  0\leq t\leq T.
\end{equation*}
 \end{lem}
 
 \begin{proof}
First,  multiplying \eqref{e-1.16mm} by $r^m|v|^{q-2}v$ and integrating the resulting equation over $I_a$, it follows from Lemma \ref{l4.3} and $\eqref{e1.5}_4$-$\eqref{e1.5}_5$ that
 \begin{equation*}
 \begin{split}
\frac{\text{d}}{\text{d}t}|r^{\frac{m}{q}}v|_q\leq &C\big((|u_r|_{\infty}+|r^{-1}|_\infty|u|_\infty)|r^{\frac{m}{q}}v|_{q}+|\rho|^{\gamma-1}_{\infty}|r^{\frac{m}{q}}v|_{q}+|\rho|^{\gamma-1}_{\infty}|r^{\frac{m}{q}}u|_{q}\big)\\
\le& C^a(1+|u_r|_{\infty}+|u|_\infty )|r^{\frac{m}{q}}v|_{q}+C^a,
\end{split}
 \end{equation*}
which, along with Lemmas \ref{l4.5}-\ref{l4.6} and the Gronwall inequality,  yields that
\begin{equation}\label{e4.35}
|r^{\frac{m}{q}}v(t,\cdot)|_q\leq C^a \quad {\rm{for}}\ \ 0\leq t\leq T.
\end{equation}
It thus follows from the definitions of $(v,\psi)$ in \eqref{evmm}, \eqref{tr}, \eqref{e4.35} and Lemmas  \ref{l4.5}-\ref{l4.6} that 
\begin{equation*}
|r^{\frac{m}{q}}\psi(t,\cdot)|_{q}\leq C\big(|r^{\frac{m}{q}}v(t,\cdot)|_{q}+|r^{\frac{m}{q}}u(t,\cdot)|_{q}\big)\leq C^a\quad {\rm{for}}\ \ 0\leq t\leq T.
\end{equation*}

Second, repeating the above procedure again, one can similarly get
$$|r^{\frac{m-2}{2}}\psi|_2\le C^a|r^{\frac{m}{2}}\psi|_2\le C^a\big(|r^{\frac{m}{2}}v(t,\cdot)|_{2}+|r^{\frac{m}{2}}u(t,\cdot)|_{2}\big)\leq C^a\quad {\rm{for}}\ \ 0\leq t\leq T.$$

The proof of Lemma \ref{l4.7} is complete. 
 \end{proof}

\subsection{The  second  order  estimates of $r^{\frac{m}{2}}u$}
Now we show the estimate of $|r^{\frac{m}{2}}u_{rr}|_2$.
\begin{lem}\label{l4.8}
 Assume $\gamma >\frac32$ additionally. Then for any $T>0$,  it holds that 
 \begin{equation*}
\big(|r^{\frac{m}{2}}u_t|^2_2+|u_{rr}|_2^2+|r^{\frac{m}{2}} u_{rr}|_2^2\big)(t,\cdot)+\int_0^t \big( |r^{\frac{m}{2}}u_{tr}|_2^2+|r^{\frac{m-2}{2}}u_t|_2^2\big)(s,\cdot) {\rm{d}}s  \le C^a 
\end{equation*}
for $0\leq t\leq T.$
\end{lem}
 
\begin{proof} 
 First,  according to  Lemmas \ref{ms}-\ref{l4.3},  \ref{l4.4} and \ref{l4.5}, one gets
 \begin{equation*}
 \begin{split}
&|r^{\frac{m}{2}}\phi_r|_2= |A\gamma r^{\frac{m}{2}}\rho^{\gamma-1}\rho^{-1}\rho_r|_2\\
\leq &  C\big(|\rho|^{\gamma-\frac32}_{\infty} |(r^m\rho)^{\frac12}|_2|v|_\infty+|r^{\frac{m}{2}}u|_2|\rho|_\infty^{\gamma-1}\big)\leq C^a.
\end{split}
\end{equation*}
Then it follows from $\eqref{e2.2}_1$ that 
\begin{align*}
|r^{\frac{m}{2}}\phi_t|_2\le& C\big(|r^{\frac{m}{2}}u\phi_r|_2+|r^{\frac{m}{2}}\phi u_r|_2+|r^{\frac{m-2}{2}}\phi u|_2\big)\\
\le&C^a\big(|u|_\infty|r^{\frac{m}{2}}\phi_r|_2+|\phi|_\infty|r^{\frac{m}{2}}u_r|_2+|\phi|_\infty|u|_2\big)\le C^a.
\end{align*}

According to \eqref{sec}, one has
\begin{equation}\label{uxxl1}
|u_{rr}|_2+|r^{\frac{m}{2}}u_{rr}|_2\le C^a\big(1+|r^{\frac{m}{2}}u_t|_2\big).
\end{equation}

Second, differentiating $(\ref{e2.2})_2$ with respect to $t$, 
multiplying the resulting equation by $r^mu_t$ and integrating over $I_a$, one has 
\begin{equation}\label{util2}
\begin{split}
&\frac{1}{2} \frac{\text{d}}{\text{d}t}|r^{\frac{m}{2}}u_t|^2_2+2\alpha|r^{\frac{m}{2}}u_{tr}|^2_2+2\alpha m|r^{\frac{m-2}{2}}u_t|_2^2\\
=&-\int r^m\big((uu_r)_t+\phi_{tr}-2\alpha(\psi u_r)_t \big)u_t {\rm{d}}r \triangleq  \sum_{i=13}^{15}\text{J}_i.
\end{split}
\end{equation}

Using the equations $\eqref{e2.2}_1$, $\eqref{e2.2}_3$, H\"older's inequality, Young's inequality, Lemmas \ref{l4.3}, \ref{l4.5}-\ref{l4.7} and \eqref{uxxl1}, we estimate $\text{J}_i$ $(i=13,14,15)$ as follows:
 \begin{equation*}
\begin{split}
\text{J}_{13}=&-\int r^m(u u_r)_t u_t {\rm{d}}r
=- \int r^m( u u_{tr} u_t+u_t^2 u_r ){\rm{d}}r\\
\le & C\big( |u|_{\infty} |r^{\frac{m}{2}}u_{tr}|_2 |r^{\frac{m}{2}}u_t|_2+ |u_r|_{\infty} |r^{\frac{m}{2}}u_t|_2^2\big)\\
\le & C^a\big(1+\|u\|_2^2\big) |r^{\frac{m}{2}}u_t|_2^2+\frac{\alpha}{8}|r^{\frac{m}{2}}u_{tr}|^2_2\\
\le & C^a\big(1+|r^{\frac{m}{2}}u_{t}|_2^2\big) |r^{\frac{m}{2}}u_t|_2^2+\frac{\alpha}{8}|r^{\frac{m}{2}}u_{tr}|^2_2,\\
\text{J}_{14}=& -\int r^m\phi_{tr}  u_t{\rm{d}}r = \int \phi_t (mr^{m-1}u_t+r^mu_{tr}){\rm{d}}r \\
\le& C (|r^{\frac{m}{2}}\phi_t|_2 |r^{\frac{m}{2}}u_{t}|_2|r^{-1}|_\infty+|r^{\frac{m}{2}}u_{tr}|_2|r^{\frac{m}{2}}\phi_t|_2)\\
\le& C^a(1+|r^{\frac{m}{2}}u_t|_2^2)+\frac{\alpha}{8}|r^{\frac{m}{2}}u_{tr}|^2_2,\\
\text{J}_{15}=&2\alpha\int r^m (\psi u_r)_t u_t {\rm{d}}r \\
=&2\alpha\int r^m\Big(\psi u_{tr} -\big((u \psi)_r+u_{rr}+(mr^{-1}u)_r\big)u_r\Big)u_t{\rm{d}}r\\
\leq &C^a\Big(|\psi|_\infty|r^{\frac{m}{2}}u_{tr}|_2|r^{\frac{m}{2}}u_t|_{2}+(1+|\psi|_\infty)|u|_\infty\big(|r^{\frac{m}{2}}u_r|_2|r^{\frac{m}{2}}u_t|_2\\
&+|r^{\frac{m}{2}}u_{rr}|_2|r^{\frac{m}{2}}u_t|_{2}+ |r^{\frac{m}{2}}u_r|_2|r^{\frac{m}{2}}u_{tr}|_2\big)+|r^{\frac{m}{2}}u_t|_2|u_r|_\infty|r^{\frac{m}{2}}u_{rr}|_2\Big)\\
\leq &C^a\big(1+ |r^{\frac{m}{2}}u_t|_2\big)|r^{\frac{m}{2}}u_t|_2^2+C^a+\frac{\alpha}{8}|r^{\frac{m}{2}}u_{tr}|^2_2.
\end{split}
\end{equation*}

Substituting the estimates for $\text{J}_i$ $(i=13,14,15)$ into \eqref{util2}, one can obtain 
\begin{equation}\label{lg}
 \frac{\text{d}}{\text{d}t}|r^{\frac{m}{2}}u_t|^2_2+|r^{\frac{m}{2}}u_{tr}|^2_2+|r^{\frac{m-2}{2}}u_t|_2^2\le C^a \big(1+(1+|r^{\frac{m}{2}}u_{t}|_2^2 )| r^{\frac{m}{2}}u_t|^2_2\big).
\end{equation}
Integrating \eqref{lg} over $(\tau, t)(\tau\in (0,t))$, one has 
\begin{equation}\label{lg1}
\begin{split}
&|r^{\frac{m}{2}}u_t(t,\cdot)|_2^2+\int_\tau^t\big( |r^{\frac{m}{2}}u_{tr}|_2^2+|r^{\frac{m-2}{2}}u_t|_2^2\big)(s,\cdot) {\rm{d}}s\\
\leq& |r^{\frac{m}{2}}u_t(\tau,\cdot)|_2^2+C^a\int_\tau^t\big(1+\big(1+|r^{\frac{m}{2}}u_{t}|_2^2 \big)| r^{\frac{m}{2}} u_t|^2_2\big)(s,\cdot) {\rm{d}}s.
\end{split}
\end{equation}

It follows from $\eqref{e2.2}_2$ that 
\begin{equation}
\begin{split}
&|r^{\frac{m}{2}}u_t(\tau,\cdot)|_2\\
\leq & C^a\big( |u|_\infty|r^{\frac{m}{2}}u_r|_2+|r^{\frac{m}{2}}\phi_r|_2+|r^{\frac{m}{2}}u_{rr}|_2+|\psi|_\infty|r^{\frac{m}{2}}u_r|_2+\|u\|_1\big)(\tau),
\end{split}
\end{equation}
which, together with the time continuity of $(\rho,u)$ and \eqref{etm},  yields that
\begin{equation*}
\begin{split}
&\lim\sup_{\tau\to 0} |r^{\frac{m}{2}}u_t(\tau,\cdot)|_2\\
\leq & C^a\big(|u_0|_\infty\|r^{\frac{m}{2}} u_0\|_1+|r^{\frac{m}{2}}(\phi_0)_r|_2+|\psi_0|_\infty|r^{\frac{m}{2}}(u_0)_r|_2+\|r^{\frac{m}{2}}u_0\|_2\big)\leq C_0^a.
\end{split}
\end{equation*}
Letting $\tau\to 0$ in \eqref{lg1} and using the Gronwall inequality, one can obtain
\[
|r^{\frac{m}{2}}u_t(t,\cdot)|^2_2+\int_0^t \big( |r^{\frac{m}{2}}u_{tr}|_2^2+|r^{\frac{m-2}{2}}u_t|_2^2\big)(s,\cdot) \text{d}s \le C^a \quad {\rm{for}}\ \ 0\leq t\leq T.
\]
It thus follows from  \eqref{uxxl1} that
\[
|u_{rr}(t,\cdot)|_2+ |r^{\frac{m}{2}}u_{rr}(t,\cdot)|_2 \le C^a\big(1+|r^{\frac{m}{2}}u_t(t,\cdot)|_2\big)\le C^a  \quad {\rm{for}}\ \ 0\leq t\leq T.
\]

The proof of Lemma \ref{l4.8} is complete.
\end{proof}

\subsection{The  high  order  estimates of $r^{\frac{m}{2}}\phi$ and $r^{\frac{m}{2}} \psi$}

The following lemma provides  the high order estimates of $r^{\frac{m}{2}}\phi$ and $r^{\frac{m}{2}} \psi$.

\begin{lem}\label{l4.9}  
 Assume $\gamma >\frac32$ additionally. Then for any $T>0$,  it holds that 
\begin{equation*}
\begin{split}
&|r^{\frac{m}{2}}\phi_{rr}(t,\cdot)|^2_2+|r^{\frac{m}{2}}\phi_{tr}(t,\cdot)|^2_2+|r^{\frac{m}{2}}\psi_r(t,\cdot)|_2^2+|r^{\frac{m}{2}}\psi_t(t,\cdot)|^2_{2}\\
&+\int_{0}^{t}\big(|u_{rrr}|^2_2+|r^{\frac{m}{2}}u_{rrr}|_2^2+|r^{\frac{m}{2}}\phi_{tt}|^2_{2}\big)(s,\cdot){\rm{d}}s\leq C^a\quad {\rm{for}}\ \ 0\leq t\leq T.
\end{split}
\end{equation*}

 \end{lem}

 \begin{proof} 
 The proof will be divided into two steps.
 
\textbf{Step 1:} the estimate on $r^{\frac{m}{2}}u_{rrr}$.
First, it follows from $\eqref{e2.2}_2$, H\"older's inequality, Young's inequality,  Lemmas \ref{l4.5}-\ref{l4.6} and  \ref{l4.8} that 
 \begin{equation}\label{uxxxl1}
\begin{split}
|r^{\frac{m}{2}}u_{rrr}|_2\le& C\big(|r^{\frac{m}{2}}u_{tr}|_{2}+|r^{\frac{m}{2}}(u  u_r)_r|_2+|r^{\frac{m}{2}}\phi_{rr}|_2+|r^{\frac{m}{2}}(\psi u_r)_r|_2\\
&+|r^{\frac{m}{2}}(r^{-1}u_r)_r|_2+|r^{\frac{m}{2}}(r^{-2}u)_r|_2\big)\\
\le & C^a\big(|r^{\frac{m}{2}}u_{tr}|_2+|u|_{\infty} | r^{\frac{m}{2}}u_{rr}|_{2}+| u_r|_{\infty}| r^{\frac{m}{2}}u_r|_2+|r^{\frac{m}{2}}\phi_{rr}|_2\\
&+|r^{\frac{m}{2}}\psi_r|_{2}|u_r|_{\infty}+|r^{\frac{m}{2}}u_{rr}|_2|\psi|_\infty+\|r^{\frac{m}{2}}u\|_2\big)\\
\le & C^a\big(1+|r^{\frac{m}{2}}u_{tr}|_2+|r^{\frac{m}{2}} \psi_r|_2+|r^{\frac{m}{2}}\phi_{rr}|_2\big).
\end{split}
\end{equation}

Second,  considering   the last two terms  on the right-hand side of \eqref{uxxxl1}, 
for $|r^{\frac{m}{2}}\psi_r|_2$, 
differentiating $\eqref{e2.2}_3$ with respect to $r$, multiplying the resulting equation  by $r^m\psi_r$ and integrating over $I_a$, one can obtain 
 \begin{equation}\label{he1}
 \begin{split}
\frac{\text{d}}{\text{d}t}|r^{\frac{m}{2}} \psi_r|^2_2=&-2\int r^m\big((\psi u)_{rr}+u_{rrr}+m(r^{-1}u)_{rr}\big)\psi_r{\rm{d}}r\\
\le & C^a|r^{\frac{m}{2}}\psi_r|_2\big( \|u\|_2|r^{\frac{m}{2}}\psi_r|_2+|\psi|_\infty|r^{\frac{m}{2}}u_{rr}|_2+|r^{\frac{m}{2}}u_{rrr}|_2+\|r^{\frac{m}{2}}u\|_2 \big)\\
\le& C^a\big(  1+ |r^{\frac{m}{2}}\psi_r|_2^2 +|r^{\frac{m}{2}} u_{rrr}|_2^2   \big)\\
\le& C^a\big(1+|r^{\frac{m}{2}}\psi_r|_2^2+|r^{\frac{m}{2}}u_{tr}|_2^2+|r^{\frac{m}{2}}\phi_{rr}|_2^2\big),
\end{split}
\end{equation}
where one has used  \eqref{uxxxl1}, Lemmas \ref{l4.5}-\ref{l4.6} and  \ref{l4.8}.

For $|r^{\frac{m}{2}}\phi_{rr}|_2$. Applying $\partial_{r}^2$ to $\eqref{e2.2}_1$,  multiplying the resulting equation by $r^m\phi_{rr}$  and integrating over $I_a$, one has
\begin{equation}\label{phixxl1}
\begin{split}
\frac{\text{d}}{\text{d}t}|r^{\frac{m}{2}} \phi_{rr}|^2_2=&-2\int r^m\big((u\phi_r)_{rr}+(\gamma-1)(\phi u_r)_{rr}+m(\gamma-1)(r^{-1}\phi u)_{rr}\big)\phi_{rr}{\rm{d}}r\\
\le &C^a  \big(|u_r |_\infty |r^{\frac{m}{2}}\phi_{rr} |_2^2+|r^{\frac{m}{2}}u_{rr}|_2|r^{\frac{m}{2}}\phi_{rr}|_2|\phi_r|_\infty\\
&+|\phi|_\infty|r^{\frac{m}{2}}\phi_{rr}|_2(|r^{\frac{m}{2}}u_{rrr}|_2+|r^{\frac{m}{2}}u_{rr}|_2)+|u|_\infty|r^{\frac{m}{2}}\phi_{rr}|_2^2\\
&+|r^{\frac{m}{2}}\phi_{rr}|_2|r^{\frac{m}{2}}u_r|_\infty|\phi_r|_2+|r^{\frac{m}{2}}\phi_{rr}|_2\|\phi\|_1\|u\|_1\big) \\
\le& C^a\big(1+|r^{\frac{m}{2}}\phi_{rr}|_2^2+|r^{\frac{m}{2}}u_{rrr}|_2^2\big)\\
\le& C^a\big(1+|r^{\frac{m}{2}}\phi_{rr}|_2^2+|r^{\frac{m}{2}}u_{tr}|_2^2+|r^{\frac{m}{2}}\psi_r|_2^2\big),
\end{split}
\end{equation}
where one has used \eqref{uxxxl1} and Lemmas \ref{l4.5}-\ref{l4.6}, \ref{l4.8}.

Combining \eqref{he1}-\eqref{phixxl1}, one can obtain that 
 \begin{equation*}
\frac{\text{d}}{\text{d}t}\big(|r^{\frac{m}{2}} \psi_r|^2_2+| r^{\frac{m}{2}}\phi_{rr}|^2_2\big)
\le C^a\big(  1+|r^{\frac{m}{2}}u_{tr}|_2^2+ |r^{\frac{m}{2}}\psi_r|_2^2 +| r^{\frac{m}{2}}\phi_{rr}|_2^2 \big),
\end{equation*}
which, along with the Gronwall  inequality and Lemma \ref{l4.8}, yields that
\begin{equation}\label{chenxs}
|r^{\frac{m}{2}}\psi_r(t,\cdot)|_2+|r^{\frac{m}{2}}\phi_{rr}(t,\cdot)|_2\le C^a \quad {\rm{for}}\ \ 0\leq t\leq T.
\end{equation}
At last, it follows from \eqref{uxxxl1}, \eqref{chenxs} and Lemma \ref{l4.8} that
\begin{equation}
\begin{split}
&\int_{0}^{t}\big(|u_{rrr}|^2_2 + |r^{\frac{m}{2}}u_{rrr}|_2^2\big)(s,\cdot) \text{d} s\\
 \le& C^a \int_{0}^{t}\big(1+ |r^{\frac{m}{2}}\psi_r|_2^2+|r^{\frac{m}{2}}u_{tr}|_2^2
 +|r^{\frac{m}{2}} \phi_{rr}|_2^2 \big) (s,\cdot)\text{d} s\le C^a\quad {\rm{for}}\ \ 0\leq t\leq T.
\end{split}
\end{equation}

\textbf{Step 2:} estimates on time derivatives of the density. First, according to $\eqref{e2.2}_3$,  H\"older's  inequality, Lemmas \ref{l4.5}-\ref{l4.6}, \ref{l4.8} and \eqref{chenxs}, one has
\begin{equation*}
\begin{split}
|r^{\frac{m}{2}}\psi_t|_2\le& C\big(|r^{\frac{m}{2}}(\psi u)_r|_2+| r^{\frac{m}{2}}u_{rr}|_2+|r^{\frac{m}{2}}(r^{-1}u)_r|_2\big)\\
\le& C^a\big(|r^{\frac{m}{2}}u_r|_2|\psi|_\infty+|r^{\frac{m}{2}}\psi_r|_2|u|_\infty+\|r^{\frac{m}{2}}u\|_2\big)\le C^a.
\end{split}
\end{equation*}

Second, according to  $\eqref{e2.2}_1$, H\"older's  inequality, Lemmas \ref{l4.5}-\ref{l4.6}, \ref{l4.8} and \eqref{chenxs}, one obtains
\begin{equation*}
\begin{split}
|r^{\frac{m}{2}}\phi_{tr}|_2 \le & C\big(|r^{\frac{m}{2}}(u\phi_r)_{r}|_2+|r^{\frac{m}{2}}(\phi u_r)_r|_2+|r^{\frac{m}{2}}(r^{-1}\phi u)_r|_2\big)\\
 \le& C^a\big(| u|_{\infty}  |r^{\frac{m}{2}}\phi_{rr}|_2+| u_r|_{\infty} |r^{\frac{m}{2}}\phi_r|_2+|\phi|_{\infty} |r^{\frac{m}{2}}u_{rr}|_2\\
 &+\|r^{\frac{m}{2}}\phi\|_1\|u\|_1\big)\le C^a,\\
|r^{\frac{m}{2}}\phi_{tt}|_2\le& C(|r^{\frac{m}{2}}(u\phi_r)_t|_2+|r^{\frac{m}{2}}(\phi u_r)_t|_2+|r^{\frac{m-2}{2}}(\phi u)_t|_2\big)\\
\le & C^a\big(|r^{\frac{m}{2}}u_t|_2|\phi_r|_{\infty}+|u|_{\infty}|r^{\frac{m}{2}}\phi_{tr}|_2+|r^{\frac{m}{2}}\phi_t|_2| u_r|_{\infty}\\
&+|\phi|_{\infty}|r^{\frac{m}{2}} u_{tr}|_2+|r^{\frac{m}{2}}\phi_t|_2|u|_\infty+|\phi|_\infty|r^{\frac{m}{2}}u_t|_2\big)\\
\le & C^a\big(1+|r^{\frac{m}{2}}u_{tr}|_2\big), 
\end{split}
\end{equation*}
which, along with Lemma \ref{l4.8}, yields that
\[
 \int_{0}^{t}|r^{\frac{m}{2}}\phi_{tt}(s,\cdot)|^2_{2}\text{d}s \leq C^a\quad {\rm{for}}\ \ 0\leq t\le T.
\]

The proof of Lemma \ref{l4.9} is complete.
\end{proof}

\subsection{ Time-weighted energy estimates of  $r^{\frac{m}{2}}u$}
At last,  one gives  the time-weighted energy estimates of  $r^{\frac{m}{2}}u$.

\begin{lem}\label{l4.10}  
 Assume $\gamma >\frac32$ additionally. Then for any $T>0$,  it holds that 
\begin{equation*}
\begin{split}
&t|r^{\frac{m}{2}} u_{tr}(t,\cdot)|^2_2+\int_{0}^{t} s \big(|r^{\frac{m}{2}}u_{tt}|^2_2+ |r^{\frac{m}{2}}u_{trr}|_2^2 \big)(s,\cdot){\rm{d}}s\leq C^a\quad {\rm{for}}\ \ 0\leq t\leq T.
\end{split}
\end{equation*}
\end{lem}
\begin{proof}
First, differentiating $(\ref{e2.2})_2$ with respect to $t$, multiplying the resulting equation  by $ r^mu_{tt}$ and integrating over $I_a$, it follows from Lemmas \ref{l4.5}-\ref{l4.6} and  \ref{l4.8}-\ref{l4.9} that 
\begin{equation*}\label{etrq}
\begin{split}
\alpha\frac{\text{d}}{\text{d}t} |r^{\frac{m}{2}}u_{tr}|_2^2+ |r^{\frac{m}{2}}u_{tt}|_2^2
=&-\int r^m\big((u u_r)_t+\phi_{tr}-2\alpha(\psi u_r)_t+2\alpha m r^{-2} u_{t}\big)u_{tt}{\rm{d}}r\\
\le& C^a|r^{\frac{m}{2}}u_{tt}|_2\big(|r^{\frac{m}{2}}u_t|_2|u_r|_\infty+|r^{\frac{m}{2}}u_{tr}|_2|u|_\infty+|r^{\frac{m}{2}}\phi_{tr}|_2\\
&+|r^{\frac{m}{2}}\psi_t|_2|u_r|_\infty+|\psi|_\infty|r^{\frac{m}{2}}u_{tr}|_2+|r^{\frac{m}{2}}u_t|_2\big)\\
\le &C^a|r^{\frac{m}{2}}u_{tt}|_2\big(1+|r^{\frac{m}{2}}u_{tr}|_2\big),
\end{split}
\end{equation*}
which, along with Young's inequality, yields that 
\begin{equation}\label{etrq1}
\frac{\text{d}}{\text{d}t}|r^{\frac{m}{2}}u_{tr}|_2^2+ |r^{\frac{m}{2}}u_{tt}|_2^2\leq C^a\big(1+|r^{\frac{m}{2}}u_{tr}|_2^2\big).
\end{equation}

Second, multiplying \eqref{etrq1} by $s$ and integrating with respect to $s$ over $[\tau, t]$ for $\tau\in(0,t)$, one has 
\begin{equation}\label{etrq2}
t|r^{\frac{m}{2}} u_{tr}(t,\cdot)|_2^2+ \int_\tau^t s|r^{\frac{m}{2}}u_{tt}(s,\cdot)|_2^2\text{d}s\leq \tau|r^{\frac{m}{2}}u_{tr}(\tau,\cdot)|_2^2+C^a.
\end{equation}

It follows from Lemma \ref{l4.8} that $$r^{\frac{m}{2}} u_{tr}\in L^2([0,T];L^2),$$
which, along with  Lemma \ref{bjr} (Appendix A), implies that there exists a sequence $\{s_k\}$ such that 
\begin{equation*}
s_k\to 0, \quad \text{and}\quad  s_k|r^{\frac{m}{2}} u_{tr}(s_k, \cdot)|_2^2\to 0 \quad \text{as}\quad k\to \infty.
\end{equation*}
After choosing $\tau=s_k\to 0$  in  \eqref{etrq2}, one can obtain
\begin{equation}\label{etrq3}
t|r^{\frac{m}{2}} u_{tr}(t,\cdot)|_2^2+ \int_0^t s|r^{\frac{m}{2}}u_{tt}(s,\cdot)|_2^2\text{d}s\leq C^a \quad {\rm{for}}\ \ 0\leq t\leq T.
\end{equation}

Finally, according to $(\ref{e2.2})_2$,   Lemmas \ref{l4.5}-\ref{l4.6} and  \ref{l4.8}-\ref{l4.9}, one has
 \begin{equation*}
 \begin{split}
 &|r^{\frac{m}{2}}u_{trr}|_2 \\
 \le & C\big(|r^{\frac{m}{2}}u_{tt}|_{2}+|r^{\frac{m}{2}}(u  u_r)_t|_2+|r^{\frac{m}{2}}\phi_{tr}|_2+|r^{\frac{m}{2}}(\psi u_r)_t|_2+|r^{\frac{m-2}{2}}u_{tr}|_2+|r^{\frac{m-4}{2}}u_t|_2\big)\\
\le & C^a\big(1+|r^{\frac{m}{2}}u_{tt}|_2 +|r^{\frac{m}{2}}u_t|_2|u_r|_\infty+|r^{\frac{m}{2}}u_{tr}|_2|u|_\infty+|r^{\frac{m}{2}}\phi_{tr}|_2+|\psi|_\infty|r^{\frac{m}{2}}u_{tr}|_2\\
&+|r^{\frac{m}{2}}\psi_t|_2|u_r|_\infty+|r^{\frac{m}{2}}u_{tr}|_2\big)\le  C^a\big(1 + |r^{\frac{m}{2}}u_{tt}|_2+ |r^{\frac{m}{2}}u_{tr}|_{2}\big),
 \end{split}
\end{equation*}
which, along with \eqref{etrq3}, implies that
\[
\int_{0}^{t} s  |r^{\frac{m}{2}}u_{trr}(s,\cdot)|_2^2{\rm{d}}s\leq C^a\quad {\rm{for}}\ \ 0\leq t\leq T.
\]

The proof of Lemma \ref{l4.10} is complete.
\end{proof}

\subsection{\textbf{Proof of Theorem \ref{jordan}}}\label{solo} 
 Based on the local-in-time well-posedness obtained in Lemma \ref{rth1} and the global-in-time a priori estimates established in Lemmas \ref{l4.1}-\ref{l4.10}, now we are ready to give the proof of Theorem \ref{jordan}.
 
\textbf{Step 1:} the global well-posedness of regular solutions.  First, Lemma \ref{rth1} guarantees that there exists a local-in-time regular solution $(\rho(t,r), u(t,r))$ to the problem \eqref{e1.5} in $[0,T_*]\times I_a$ for some $T_*>0$. Let $\bar{T}>0$ be the life span of the regular solution shown in Lemma \ref{rth1}. It is obvious that $\bar{T}\ge T_*$. Then we claim that $\bar{T}=\infty$. Otherwise, if $\bar{T}<\infty$, 
according to  the uniform a priori estimates obtained in Lemmas \ref{l4.1}-\ref{l4.10} and the standard weak compactness theory, one can know that for any sequence $\{t_k\}_{k=1}^\infty$ satisfying $0<t_k<\bar{T}$ and $$t_k\to \bar{T} \quad \text{as}\quad k\to \infty,$$ 
there exists a subsequence $\{t_{1k}\}_{k=1}^\infty\subset \{t_{k}\}_{k=1}^\infty$ and functions $(\phi, u, \psi)(\bar{T},r)$ such that 
\begin{equation}\label{f2}
\begin{split}
    r^{\frac m2}\phi(t_{1k},r)\rightharpoonup  r^{\frac m2}\phi(\bar{T},r) \quad &\text{in\,\,} H^2 \quad \text{as}\,\, k\to\infty,\\ 
     r^{\frac m2}u(t_{1k},r)\rightharpoonup  r^{\frac m2}u(\bar{T},r) \quad &\text{in\,\,} H^2 \quad  \text{as}\,\, k\to\infty,\\
     r^{\frac mq}\psi(t_{1k},r)\rightharpoonup  r^{\frac mq}\psi(\bar{T},r) \quad &\text{in\,\,} L^q \quad  \text{as}\,\, k\to\infty,\\
     \psi(t_{1k},r)\stackrel{*}\rightharpoonup  \psi(\bar{T},r) \quad &\text{in\,\,} L^\infty \quad  \text{as}\,\, k\to\infty,\\
     r^{\frac {m-2}{2}}\psi(t_{1k},r)\rightharpoonup  r^{\frac {m-2}{2}}\psi(\bar{T},r) \quad& \text{in\,\,} L^2 \quad  \text{as}\,\, k\to\infty,\\
     r^{\frac m2}(\psi(t_{1k},r))_r\rightharpoonup  r^{\frac m2}(\psi(\bar{T},r))_r \quad &\text{in\,\,} L^2 \quad  \text{as}\,\, k\to\infty.
\end{split}
\end{equation}

Second, we want  to show that functions $(\phi, u, \psi)(\bar{T},r)$ satisfy all the initial assumptions shown  in Lemma \ref{rth1}, which include \eqref{etm} and  the following relationship between $\phi$ and $\psi$: 
\begin{equation}\label{ff} 
\psi=\frac{1}{\gamma-1}\big(\ln\phi\big)_r.
\end{equation}
It follows from \eqref{f2} that \eqref{etm} except $r^m\rho\in L^1$ still  holds at the  time $t=\bar{T}$.

 Next for  the relation in \eqref{ff}, we need to  consider the following equation
\begin{equation}\label{limit-3}
\phi_t+u\phi_r+(\gamma-1)\phi u_r+m(\gamma-1)r^{-1}\phi u=0,
\end{equation}
which holds in $[0,\bar{T})\times I_a$ in the classical sense. Actually, if we regard  $(\phi(\bar{T},r), u(\bar{T},r),\\
\psi(\bar{T},r))$  and
$$\phi_t(\bar{T},r)=-u(\bar{T}, r)\phi_r(\bar{T}, r)-(\gamma-1)\phi(\bar{T}, r)\big( u_r(\bar{T}, r)+mr^{-1}u(\bar{T}, r)\big),$$ 
as the extended definitions of  $(\phi(t, r), u(t, r), \psi(t, r),\phi_t(t, r))$ at the  time $t=\bar{T}$, then one has
$$
\text{ess}\sup_{0\leq t \leq \bar{T}}(\|\phi(t,\cdot)\|_2+\|\phi_t(t,\cdot)\|_1)\leq C^a<\infty,$$
which, together with the Sobolev embedding theorem (Appendix A), implies that 
\begin{equation}\label{f3}
    \phi(t, r)\in C([0,\bar{T}];H^1).
\end{equation}
It follows from the first line in  $\eqref{f2}$ and the consistency of weak convergence and  strong convergence in $H^1$ space that 
 \begin{equation}\label{f4}
    \phi(t_{1k},r)\to  \phi(\bar{T},r) \quad \text{in} \quad H^1 \quad \text{as}\quad  k\to\infty.
\end{equation} 
Notice that for any $R>a$, there exists a generic constant $C(\bar{T}, R)$ such that 
\begin{equation}\label{f5}
\phi(t,r)\ge C(\bar{T}, R)\quad \text{for\,\,any}\quad (t,r)\in[0,\bar{T}]\times (a,R],
\end{equation}
which, along with the last four lines in $\eqref{f2}$ and \eqref{f4}, implies that \eqref{ff} holds.

It remains to show that $r^m\rho(\bar{T},r)\in L^1$. According to \eqref{f3} and the Sobolev embedding theorem (Appendix A), one can obtain   $$\phi(t,r)\in C([0,\bar{T}]\times I_a),$$ which, along with the fact $\rho=\big(\frac{\gamma-1}{A\gamma}\phi\big)^{\frac{1}{\gamma-1}}$, yields that
\begin{equation}\label{xiaohan}
r^m\rho(t,r)\in C([0,\bar{T}]\times I_a).
\end{equation}
Moreover, one has  
\begin{equation}\label{yizhi}
\int r^m\rho(t,r)\text{d}r=\int r^m\rho_0(r)\text{d}r \quad  \text{and}\quad  0\leq t<\bar{T}.
\end{equation}

It follows from \eqref{xiaohan}-\eqref{yizhi} and Fatou's lemma (Appendix A) that,
$$\int r^m\rho(\bar{T},r)\text{d}r=\int \liminf_{k\to\infty} r^m\rho(t_{1k},r)\text{d}r\le \liminf_{k\to\infty}\int r^m\rho(t_{1k},r)\text{d}r<\infty,$$
which implies that $r^m\rho(\bar{T},r)\in L^1$. Thus we have shown that $(\rho(\bar{T},r),u(\bar{T},r))$ satisfy all the initial assumptions on the initial data of Lemma \ref{rth1}.

Finally, if  we solve \eqref{e1.5} with the initial time $\bar{T}$, then Lemma \ref{rth1} ensures that for some constant $T_0>0$, $(\rho, u)(t,r)$ is  the unique regular solution in $[\bar{T},\bar{T}+T_0]\times I_a$ to this problem. It follows from the boundedness of all required norms of the solution $(\rho,u)(t,r)$ in  $[0,\bar{T}+T_0]\times I_a$ and the standard arguments for proving the time continuity of the regular solution  that, $(\rho, u)(t,r)$ is actually  the unique regular solution in $[0,\bar{T}+T_0]\times I_a$ to the  problem \eqref{e1.5}, which  contradicts to the fact that $0<\bar{T}<\infty$ is the maximal existence time.

\begin{rk}
According to the  proof in this subsection, the definitions of $(\phi(\bar{T},r)$,\\
$u(\bar{T},r), \psi(\bar{T},r))$ do not depend on the choice of the time sequences $\{t_k\}_{k=1}^\infty$.
\end{rk}


\textbf{Step 2:} the global well-posedness of classical solutions. Now we show that the regular solution obtained above is indeed a classical one in $(0,\infty)\times I_a$. 

Actually,  for any $T>0$, first  due to   $\phi>0$ and 
$$\rho=\big(\frac{\gamma-1}{A\gamma}\phi\big)^{\frac{1}{\gamma-1}},$$
one has 
\begin{equation}\label{jichulianxu}
(\rho, \rho_t, \rho_r, u,  u_r)\in C([0,T]\times I_a).
\end{equation}

Second, for the term $u_t$, according to   Lemma \ref{l4.10}, one has 
\begin{equation}\label{wtw}
\begin{split}
&t^{\frac12}u_{tr}\in L^\infty([0,T];L^2),\quad t^{\frac12}u_{tt}\in L^2([0,T];L^2),\\
& t^{\frac12}u_{trr}\in L^2([0,T];L^2), \quad t^{\frac12}u_{ttr}\in L^2([0,T];H^{-1}),
\end{split}
\end{equation}
which, along with  the classical Sobolev embedding  theorem:
\begin{equation}\label{qian}\begin{split}
&L^2([0,T];H^1)\cap W^{1,2}([0,T];H^{-1})\hookrightarrow C([0,T];L^2),
\end{split}
\end{equation}
 yields that for any $0<\tau<T$,
\begin{equation}\label{utlianxu}
tu_t\in C([0,T];H^1) \quad \text{and}\quad u_t\in C([\tau,T]\times I_a).
\end{equation}

Finally, it remains to  show that
 $$ u_{rr} \in C([\tau,T]\times I_a)$$ 
 for any $0<\tau<T$. Actually, it follows from the fact $\rho>0$, \eqref{jichulianxu} and \eqref{utlianxu} that
$$u_{rr}=\frac{1}{2\alpha}\big(u_t+u u_r+A\gamma\rho^{\gamma-2}\rho_r-2\alpha\rho^{-1}\rho_r u_r-2\alpha m (r^{-1} u)_r\big)\in C([\tau,T]\times I_a).$$

The proof of Theorem \ref{jordan} is complete. 


\section{Proof of Theorems \ref{th1}-\ref{thshallow} }\label{se46}

With Theorems \ref{zth} and \ref{jordan}  at hand, 
now  we turn to the  proof of Theorems \ref{th1}-\ref{thshallow}. In the rest of this section, we denote  $r=|x|$.
\subsection{Proof of Theorem \ref{th1}}
\begin{proof}

\textbf{Step 1}: the existence of the global  unique regular solution.
First,  in Theorem \ref{th1}, one assumes that  the initial data $$(\rho_0(x), U_0(x))= ( \rho_0(r),u_0(r)\frac{x}{r})$$
be   spherically symmetric and   satisfy \eqref{id1}, which, along with Lemma \ref{poss} and Remark  \ref{rec} in Appendix B, yields that  the  data $(\rho_0(r), u_0(r))$ satisfy \eqref{etm}.
Then it follows from Theorem \ref{jordan} that the IBVP \eqref{e1.5}  admits a unique global classical solution $(\rho(t,r), u(t,r))$
in $(0,\infty)\times I_a$  satisfying the  regularities in \eqref{spd} with $T_*$ replaced by  arbitrarily large time  $0<T<\infty$. 
Then we  denote
\begin{equation}\label{jjj}
\rho(t,x)=\rho(t,r),\quad U(t,x)=u(t,r)\frac{x}{r},
\end{equation}
which, along with the boundary condition and the far field behavior:
\begin{equation}\label{e1.5kkk}
\begin{cases}
\displaystyle
\ u(t,r)|_{r=a}=0 \quad \text{for} \quad  t\geq 0,\\[6pt]
\displaystyle
\  \left(\rho(t,r),u(t,r)\right)\to \left(0,0\right)\quad  \text{as}\quad  r\to \infty\quad  \text{for}\quad   t\ge 0,
\end{cases}
\end{equation}
in problem \eqref{e1.5}, yields that the pair $(\rho(t,x), U(t,x))$ satisfies the assumptions \eqref{eqs:CauchyInit}-\eqref{e1.3}.
Moreover, it follows from the definition of  $(\rho(t,x), U(t,x))$ in \eqref{jjj}, direct calculations and the equations of $(\rho(t,r), u(t,r))$ in  \eqref{e1.5} that $(\rho(t,x),U(t,x))$ also satisfies the equations \eqref{eq:1.1}-\eqref{10000} pointwisely.

Second,   similarly to the proof of Lemma  \ref{poss}  and Remark \ref{rec},   based on  the regularities  of $(\rho(t,r),u(t,r))$ in \eqref{spd},  one has that for any $T>0$, $(\rho(t,x), U(t,x))$  satisfies   the regularities shown  in Definition \ref{cjk} and \eqref{er2}. In summary,  we have shown that the pair $(\rho(t,x), U(t,x))$ constructed by \eqref{jjj} is actually a regular solution to the IBVP \eqref{eq:1.1}-\eqref{10000} with  \eqref{eqs:CauchyInit}-\eqref{e1.3}.

Finally, the uniqueness follows easily from the same procedure as in Section 2.4.

\textbf{Step 2}: the existence of the global  unique classical solution. 
Now we  show   that the regular solution  obtained  above  is indeed a classical one  in $(0,\infty)\times \Omega$, which can be achieved by the following lemma.  

\begin{lem}\label{boss}
Since 
\begin{equation}\label{123333}
\begin{split}
&(\rho, \rho_t, \rho_r, u,  u_r)\in C([0,\infty)\times I_a), \quad  (u_t, u_{rr}) \in C((0,\infty)\times I_a),
\end{split}
\end{equation}
then one has 
\begin{equation}\label{126666}
\begin{split}
&(\rho, \rho_t, \nabla \rho, U,  \nabla U)\in C([0,\infty)\times \Omega), \quad  (U_t, \nabla^2 U) \in C((0,\infty)\times \Omega).
\end{split}
\end{equation}
\end{lem}

\begin{proof}
It follows from the definition of  $(\rho(t,x), U(t,x))$ in \eqref{jjj} and  direct calculations that 
\begin{equation}\label{wwo}
    \begin{split}
        \frac{\partial\rho(t,x)}{\partial x_i}=&(\rho(t,r))_r\frac{x_i}{r},\quad \frac{\partial\rho(t,x)}{\partial t}=(\rho(t,r))_t,\quad\frac{\partial U(t,x)}{\partial  t}=(u(t,r))_t\frac{x}{r},\\
        \frac{\partial U^j(t,x)}{\partial x_i}=&\frac{\partial}{\partial x_i}\Big(u(t,r)\frac{x_j}{r}\Big)=(u(t,r))_r\frac{x_i x_j}{r^2}+ u(t,r)\frac{\delta_{ij}r^2-x_i x_j}{r^3},\\
       \frac{\partial^2 U^k(t,x)}{\partial x_i\partial x_j}=&(u(t,r))_{rr}\frac{x_i x_j x_k}{r^3}+(u(t,r))_r\Big(\frac{\delta_{ij}x_k+\delta_{ik}x_j+\delta_{jk}x_i}{r^2}-3\frac{x_i x_j x_k}{r^4}\Big)\\
       &+u(t,r)\Big(\frac{3x_i x_j x_k}{r^5}-\frac{\delta_{jk}x_i+\delta_{ji}x_k+\delta_{ik}x_j}{r^3}\Big),
    \end{split}
\end{equation}
which, along with \eqref{123333}, yields \eqref{126666} holds.

The proof of Lemma \ref{boss} is complete.

\end{proof}
The proof of Theorem \ref{th1} is complete.

\end{proof}

\subsection{Proof of Theorem \ref{thshallow} }
\begin{proof}
\textbf{Step 1}: the cases $\mathcal{V}=\text{div}(h D(W))$, or $\text{div}(2h D(W))$.
 Actually,  when $\mathcal{V}=\text{div}(h D(W))$, this is a special case of
system \eqref{eq:1.1} with 
$$d=2,\quad \alpha=\frac{1}{2},\quad \beta=0,\quad  \delta=1 \quad \text{and} \quad  \gamma=2;$$ when $\mathcal{V}=\text{div}(2h D(W))$, this is also a special case of system \eqref{eq:1.1} with 
$$d=2,\quad \alpha=1,\quad \beta=0,\quad  \delta=1 \quad \text{and} \quad  \gamma=2.$$
Therefore, one simply replaces $(\rho,U)$ by $(h,W)$ in Theorem \ref{th1} to obtain the same conclusion  for these two classes of shallow water models, without further modifications.

\textbf{Step 2}: the case $\mathcal{V}=\text{div}(h \nabla W)$. We still hope  to  prove the desired global-in-time well-posedness based on  the  framework established in the proof of Theorem \ref{th1}.
Since   system \eqref{shallow} with $\mathcal{V}=\text{div}(h \nabla W)$ is not a special case of  the system \eqref{eq:1.1}-\eqref{10000} formally, some necessary explanations need to  be given.

First, according to the proof of Theorem \ref{zth}, one has that the corresponding local-in-time well-posedenss can be obtained  via the completely same argument used in \S 2.

Second, according to the proof of Theorem \ref{jordan}, for getting  the corresponding  global-in-time  a priori estimates, one  needs to consider a   reformulated problem in the spherically symmetric Eulerian  coordinate.

Concerning the spherically symmetric solutions $(h(t,x),W(t,x))$ taking  the  form: $$(h,W)(t,x)
=(h(t,r),w(t,r)\dfrac{x}{r}),$$ then the IBVP \eqref{shallow} with \eqref{shallowchu}-\eqref{shallowbian} can be rewritten as
\begin{equation}\label{eqshallow}
   \left\{
    \begin{split}
     &\ h_t+(hw)_r+\frac{ h w}{r}=0,\\[5pt]
     &\ (hw)_t+(hw^2)_r+(h^2)_r-\Big(h\big(w_r+\frac{1}{r}w\big)\Big)_r+\frac{ h_r w}{r}+\frac{ h w^2}{r}=0,\\[5pt]
     &\ (h(0,r),w(0,r))=(h_0(r),w_0(r))\quad \text{for} \quad r\in I_a,\\[5pt]
     &\ w(t,r)|_{r=a}=0  \quad \text{for} \quad  t\geq 0,\\[5pt]
     &\  (h(t,r),w(t,r))\to (0,0) \quad \text{as}\quad  r \to \infty  \quad \text{for} \quad  t\geq 0.
    \end{split}
    \right.
\end{equation}
It is obvious that the  problem \eqref{eqshallow}  is actually  a special case of problem  \eqref{e1.5} with 
$$d=2,\quad \alpha=\frac{1}{2}\quad \text{and} \quad \gamma=2.$$
Therefore, one simply replaces $(\rho,u)$ by $(h,w)$ in the proof of  Theorem \ref{jordan} to obtain the same uniform a priori estimates shown in Lemmas \ref{l4.1}-\ref{l4.10}  for this  shallow water model, without further modifications. 

Then the proof of  Theorem \ref{thshallow} is complete.

\end{proof}

\appendix

\section{ Basic lemmas}


This appendix is devoted to listing  some useful lemmas which were used frequently in the previous sections.
The first one is the  Sobolev embedding theorem (see Theorem 4.12 (page 85) of  \cite{af}).
%

\begin{lem}[\cite{af}]\label{ale1}
Assume  $\mathbb{D}$ is  a domain in $\mathbb R^d$ which  satisfies the cone condition, and $(j,m,p)$ are all constants satisfying  $j\ge 0$, $m^*\ge 1$ and $1\le p<\infty$.
\begin{itemize}
\item[$({\rm{i}})$] If $m^*p<d$, then 
$$W^{j+m^*, p}(\mathbb{D})\hookrightarrow W^{j,q}(\mathbb{D})\quad \text{for} \quad    p\le q\le \frac{dp}{d-m^*p};$$
\item[$({\rm{ii}})$] if $m^*p=d$, then $$W^{m^*,p}(\mathbb{D})\hookrightarrow L^q(\mathbb{D})\quad \text{for} \quad  p\le q<\infty;$$
\item [$({\rm{iii}})$] if $m^*p>d$ or $m^*=d$ and $p=1$, then $$W^{j+m^*, p}(\mathbb{D})\hookrightarrow C_B^j(\mathbb{D}),$$
\end{itemize}
 where $C_B^j(\mathbb{D})=\{f\in C^j(\mathbb{D})|D^\alpha f\in L^\infty(\mathbb{D}), |\alpha|\le j\}.$ The embedding constants here  depend only on $d,m^*, p,q,j$ and the property of the cone $\mathscr{C}$ in the cone condition.
\end{lem} 

\begin{rk}[see Section 4.6 (page 82) of \cite{af}]
$\mathbb{D}$ satisfies the cone condition if there exists a finite cone $\mathscr{C}$ such that each $x\in \mathbb{D}$ is the vertex of a finite cone $\mathscr{C}_x$ contained in  $\mathbb{D}$ and congruent to  $\mathscr{C}$. Note that $\mathscr{C}_x$  need not   to be obtained from $\mathscr{C}$ by parallel transformation, but simply by rigid motion.
\end{rk}

The second one is the well-known   Fatou's lemma which can be found in \cite{realrudin}.
\begin{lem}[\cite{realrudin}]\label{Fatou}
Given a measure space $(V,\mathcal{F},\nu)$ and a set $X\in \mathcal{F}$, let  $\{f_n\}$ be a sequence of $(\mathcal{F} , \mathcal{B}_{\mathbb{R}_{\geq 0}} )$-measurable non-negative functions $f_n: X\rightarrow [0,\infty]$. Define the function $f: X\rightarrow [0,\infty]$ by setting 
$$
f(x)= \liminf_{n\rightarrow \infty} f_n(x),$$
for every $x\in X$. Then $f$ is $(\mathcal{F},  \mathcal{B}_{\mathbb{R}_{\geq 0}})$-measurable, and
$$
\int_X f(x) \text{\rm d}\nu \leq \liminf_{n\rightarrow \infty} \int_X f_n(x) \text{\rm d}\nu.
$$
\end{lem}

The third one is used to obtain the time-weighted estimates of the velocity.
\begin{lem}[\cite{bjr}]\label{bjr}
If $f(t,\cdot)\in L^2([0,T]; L^2(\mathcal {O}))$ ($\mathcal {O}$ can be any domain in $\mathbb R^d$), then there exists a sequence $\{s_k\}$ such that
$$
s_k\rightarrow 0 \quad \text{and}\quad s_k \|f(s_k,\cdot)\|^2_{L^2(\mathcal {O})}\rightarrow 0 \quad \text{as} \quad k\rightarrow\infty.
$$
\end{lem}

\begin{proof}
Denote $h(s)=\|f(s,\cdot)\|^2_{L^2(\mathcal {O})}$, then one can see that $0\leq h(s)\in L^1([0,T])$. We claim that for any $k\ge 1$, there exists $0<s_k<\frac{1}{N_k}$ such that 
$$s_kh(s_k)<\frac{1}{N_k}\to 0\quad \text{as}\quad k\to \infty,$$
where $N_k=N_0+k$ and $N_0$ is a positive constant satisfying $\frac{1}{N_0}\le \frac{T}{2}$.
Indeed, assume by contradiction that there exist some $k\ge1$ such that for any $s\in\big(0,\frac{1}{N_k}\big)$, 
$$sh(s)\ge \frac{1}{N_k},$$
which yields  that 
$$\int_0^T h(s)\text{d}s\ge \int_0^{\frac{1}{N_k}} \frac{1}{s N_k}\text{d}s=\infty.$$
This contradicts with the fact $ h(s)\in L^1([0,T])$.
Thus the claim holds. 

The proof of Lemma \ref{bjr} is complete. 
\end{proof}

The following one is on compactness theories  obtained via the Aubin-Lions Lemma.
\begin{lem}[\cite{js}]\label{gyl}
 Let $X_0\subset X\subset X_1$ be three Banach spaces.  Suppose that $X_0$ is compactly embedded in $X$ and $X$ is continuously embedded in $X_1$. Then the following statements hold
\begin{enumerate}
\item If $J$ is bounded in $L^p([0,T];X_0)$ for $1\leq p < \infty$, and $\frac{\partial J}{\partial t}$ is bounded in $L^1([0,T];X_1)$, then $J$ is relatively compact in $L^p([0,T];X)$.\\

\item If $J$ is bounded in $L^\infty([0,T];X_0)$  and $\frac{\partial J}{\partial t}$ is bounded in $L^p([0,T];X_1)$ for $p>1$, then $J$ is relatively compact in $C([0,T];X)$.
\end{enumerate}
\end{lem}

The last one provides the regularity estimates for the Lam\'e operator in the  exterior domain  $\Omega$ in  $\mathbb R^d$ $(d=2\ \text{or} \ 3)$. We consider 
\begin{equation}\label{aue}
\left\{
\begin{aligned}
&LU=-\alpha\Delta U-\alpha\nabla\text{div}U=F \quad \text{in} \quad \Omega,\\
&U|_{\partial \Omega}=0,\quad U(x)\to 0\quad \text{as}\quad |x|\to\infty.
\end{aligned}
\right.
\end{equation}

\begin{lem}[\cite{CK3}]\label{df3}
Let $\Omega$ be  an exterior domain in  $\mathbb R^d$ with smooth boundary. If $U\in D^{1,q}(\Omega)$ is a weak solution to \eqref{aue}, then for any $1<q<\infty$, 
$$\|U\|_{D^{k+2,q}(\Omega)}\le C\big(\|F\|_{W^{k,q}(\Omega)}+\|U\|_{D^{1,q}(\Omega)}\big),$$
where  $C$ is a positive constant  independent of $(U,F)$.

\end{lem}

\section{ Conversion of  Sobolev spaces}

In order to  understand the well-posedness theories  established   in Theorems \ref{th1}-\ref{thshallow}, \ref{zth} and \ref{jordan} clearly, this appendix is devoted to showing the conversion of some Sobolev  spaces between   the pure M-D  coordinate and  the spherically symmetric coordinate.


\begin{lem}\label{poss}
Let $m=d-1$. If the initial data $(\rho_0(x),U_0(x))$ have the following form:
\begin{equation}\label{ssf}
    \rho_0(x)=\rho_0(r),\quad U_0(x)=u_0(r)\frac{x}{r}, 
\end{equation}
where $r=|x|$, then the following two assertions are equivalent:
\begin{equation*}
\begin{split}
    &({\rm{i}})\quad   \big(\rho_0^{\gamma-1}(x), U_0(x)\big)\in H^2(\Omega),\quad \nabla\ln\rho_0(x)\in L^q(\Omega)\cap L^\infty(\Omega)\cap D^1(\Omega).\\
    & ({\rm{ii}})\quad  r^{\frac{m}{2}}\Big(\rho_0^{\gamma-1}(r),(\rho_0^{\gamma-1}(r))_r,r^{-1}(\rho_0^{\gamma-1}(r))_r, (\rho_0^{\gamma-1}(r))_{rr}\Big)\in L^2(I_a),\\
    &\qquad \ r^{\frac{m}{2}}\Big(r^{-2}u_0(r),r^{-1}u_0(r), u_0(r), (u_0(r))_r, r^{-1}(u_0(r))_r, (u_0(r))_{rr}\Big)\in L^2(I_a),\\
    &\qquad \ r^{\frac{m}{q}}(\ln\rho_0(r))_r\in L^q(I_a),\quad r^{\frac{m}{2}}\big(r^{-1}(\ln\rho_0(r))_r,(\ln\rho_0(r))_{rr}\big)\in L^2(I_a),\\
    &\qquad \  (\ln\rho_0(r))_r\in L^\infty(I_a).\\ 
    \end{split}
\end{equation*}
\end{lem}

\begin{proof}
First, it follows from direct calculations that
\begin{equation}\label{mp}
    \begin{split}
        \frac{\partial\rho_0^{\gamma-1}(x)}{\partial x_i}=&(\rho_0^{\gamma-1}(r))_r\frac{x_i}{r},\quad \frac{\partial\ln\rho_0(x)}{\partial x_i}=(\ln\rho_0(r))_r\frac{x_i}{r},\\ \frac{\partial^2 \rho_0^{\gamma-1}(x)}{\partial x_i\partial x_j}=&( \rho_0^{\gamma-1}(r))_{rr}\frac{x_i x_j}{r^2}+(\rho_0^{\gamma-1}(r))_r\frac{\delta_{ij}r^2-x_i x_j}{r^3},\\
        \frac{\partial U_0^j(x)}{\partial x_i}=&\frac{\partial}{\partial x_i}\Big(u_0(r)\frac{x_j}{r}\Big)=(u_0(r))_r\frac{x_i x_j}{r^2}+ u_0(r)\frac{\delta_{ij}r^2-x_i x_j}{r^3},\\
       \frac{\partial^2 U_0^k(x)}{\partial x_i\partial x_j}=&(u_0(r))_{rr}\frac{x_i x_j x_k}{r^3}+( u_0(r))_r\Big(\frac{\delta_{ij}x_k+\delta_{ik}x_j+\delta_{jk}x_i}{r^2}-3\frac{x_i x_j x_k}{r^4}\Big)\\
       &+u_0(r)\Big(\frac{3x_i x_j x_k}{r^5}-\frac{\delta_{jk}x_i+\delta_{ji}x_k+\delta_{ik}x_j}{r^3}\Big),\\
       \frac{\partial^2\ln\rho_0(x)}{\partial x_i\partial x_j}=&(\ln\rho_0(r))_{rr}\frac{x_i x_j}{r^2}+(\ln\rho_0(r))_r\frac{\delta_{ij}r^2-x_i x_j}{r^3}.
    \end{split}
\end{equation}

Second, via introducing the unit spherical coordinate transformation in $\mathbb R^d (d=2, 3)$
\begin{equation}
    \begin{split}
     &x'_1=\cos\varphi_1,\quad x'_2=\sin\varphi_1\cos\varphi_2,\\
     &\cdots\cdots\\
     &x'_{d-1}=\sin\varphi_1\sin\varphi_2\cdots\sin\varphi_{d-2}\cos\varphi_{d-1},\\ &x'_d=\sin\varphi_1\sin\varphi_2\cdots\sin\varphi_{d-2}\sin\varphi_{d-1},
    \end{split}
\end{equation}  
where $x'=(x'_1,\cdots, x'_d)\in\Omega'$, $\Omega'$ is the unit sphere in $\mathbb R^d$, $\varphi_k\in [0,\pi] (k=1,\cdots, d-2)$ and  $\varphi_{d-1}\in [0,2\pi]$, then one can obtain that for any real integrable function $f(x)$
\begin{equation}\label{mvp}
\int_{\Omega}f(x)\text{d}x=\int_{\Omega'}\int_a^\infty f(rx')r^m\text{d}r\text{d}x'.
\end{equation}

Therefore, one can easily verify that $({\rm{i}})$ and $({\rm{ii}})$ are equivalent by using \eqref{ssf}-\eqref{mvp}.
\end{proof}

\begin{rk}\label{rec}
For any fixed  constant $a>0$, one can also show that $({\rm{ii}})$ is equivalent to
\begin{equation*}
\begin{split}    
&r^{\frac{m}{2}}\big(\rho_0^{\gamma-1}(r), u_0(r)\big)\in H^2(I_a), \quad r^{\frac{m}{q}}(\ln\rho_0(r))_r\in L^q(I_a),\\ 
&r^{\frac{m}{2}}\big(r^{-1}(\ln\rho_0(r))_r,(\ln\rho_0(r))_{rr}\big)\in L^2(I_a), \quad (\ln\rho_0(r))_r\in L^\infty(I_a).
\end{split}
\end{equation*}
Indeed, according to Lemma \ref{poss}  and the following facts
\begin{align*}
&(r^{\frac{m}{2}}\rho_0^{\gamma-1}(r))_r={\frac{m}{2}}r^{\frac{m-2}{2}}\rho_0^{\gamma-1}(r)+r^{\frac{m}{2}}(\rho_0^{\gamma-1}(r))_r,\\
&(r^{\frac{m}{2}}\rho_0^{\gamma-1}(r))_{rr}={\frac{m(m-2)}{4}} r^{\frac{m-4}{2}}\rho_0^{\gamma-1}(r)+mr^{\frac{m-2}{2}}(\rho_0^{\gamma-1}(r))_r+r^{\frac{m}{2}}(\rho_0^{\gamma-1}(r))_{rr},\\
&(r^{\frac{m}{2}} u_0(r))_r=\frac{m}{2} r^{\frac{m-2}{2}}u_0(r)+r^{\frac{m}{2}}(u_0(r))_r,\\
& (r^{\frac{m}{2}}u_0(r))_{rr}={\frac{m(m-2)}{4}} r^{\frac{m-4}{2}}u_0(r)+mr^{\frac{m-2}{2}}(u_0(r))_r+r^{\frac{m}{2}}(u_0(r))_{rr},
\end{align*}
one can get the desired conclusion. Therefore, the  initial assumption \eqref{etm} in  Lemma \ref{rth1} is reasonable. 
\end{rk}

\bigskip
\noindent{\bf Acknowledgement:}
The research  was  supported in part  by  National Natural Science Foundation of China under the Grant  12101395.
The research of Shengguo Zhu was also supported in part by The Royal Society--Newton International Fellowships Alumni AL/201021 and AL/211005.

\bigskip
\noindent{\bf Conflict of Interest:} The authors declare  that they have no conflict of
interest. 
The authors also  declare that this manuscript has not been previously  published, and will not be submitted elsewhere before your decision.

\end{document}